\documentclass[a4paper]{amsart}

\usepackage{amsfonts}
\usepackage{amssymb}
\usepackage{amsthm}
\usepackage{amsmath}
\usepackage[all]{xy}
\usepackage{tabularx}

\SelectTips{cm}{}

\theoremstyle{plain}
\newtheorem{theorem}{Theorem}[section]
\newtheorem*{theorem-non}{Theorem}
\newtheorem{proposition}[theorem]{Proposition}
\newtheorem*{proposition-non}{Proposition}
\newtheorem{corollary}[theorem]{Corollary}

\newtheorem*{conjecture-non}{Conjecture}
\newtheorem{lemma}[theorem]{Lemma}
\newtheorem*{lemma-non}{Lemma}

\theoremstyle{definition}
\newtheorem{definition}[theorem]{Definition}

\newtheorem{notation}[theorem]{Notation}

\theoremstyle{remark}
\newtheorem{remark}[theorem]{Remark}

\newtheorem{example}[theorem]{Example}
\newtheorem{examples}[theorem]{Examples}
\newtheorem*{example-non}{Example}

\numberwithin{equation}{section}

\DeclareMathOperator{\colim}{colim}
\DeclareMathOperator{\im}{im}
\DeclareMathOperator{\diag}{diag}

\DeclareMathOperator{\Idem}{Idem}

\newcommand{\id}{\textup{id}}
\newcommand{\pr}{\textup{pr}}
\newcommand{\map}{\textup{map}}

\newcommand{\trans}{\textup{trans}}

\newcommand{\pt}{\textup{pt}}

\newcommand{\Aut}{\textup{Aut}}
\newcommand{\GL}{\textup{Gl}}

\def\C{\mathbb C}

\def\N{\mathbb N}
\def\Q{\mathbb Q}
\def\R{\mathbb R}
\def\Z{\mathbb Z}

\title{The Farrell-Jones Conjecture for virtually solvable groups}
\author{Christian Wegner}
\subjclass[2000]{Primary 19D10; Secondary 19A31, 19B28}
\keywords{Farrell-Jones Conjecture, algebraic $K$- and $L$-theory of group rings, solvable groups}
\address{Hausdorff Research Institute for Mathematics (HIM) \\ Poppelsdorfer Allee 45 \\ 53115 Bonn \\ Germany}
\email{wegner@him.uni-bonn.de}

\begin{document}

\begin{abstract}
We prove the $K$- and $L$-theoretic Farrell-Jones Conjecture (with coefficients in additive categories) for virtually solvable groups.
\end{abstract}

\maketitle

\section{Introduction}

The Farrell-Jones Conjecture for algebraic $K$-theory predicts the structure of the algebraic $K$-groups $K_n(RG)$ for a group $G$ and a ring $R$. There is also an $L$-theoretic version.
The Farrell-Jones Conjecture plays an important role in the classification and geometry of manifolds. It implies a variety of well-known conjectures, e.g., the Bass-, Borel-, Kaplansky- and Novikov-Conjecture.

This article is dedicated to the Farrell-Jones Conjecture for virtually solvable groups. The main result is
\begin{theorem} \label{thm}
Let $G$ be a virtually solvable group. Then $G$ satisfies the $K$- and $L$-theoretic Farrell-Jones Conjecture (with coefficients in additive categories) with respect to the family of virtually cyclic subgroups.
\end{theorem}
This result has already been used in several papers, e.g., \cite{GMR13}, \cite{KLR14}, \cite{Rue13}.

The Farrell-Jones Conjecture for virtually solvable groups has been studied by several mathematicians. In \cite{FL03} Farrell and Linnell show that if the fibered isomorphism conjecture is true for all nearly crystallographic groups, then it is true for all virtually solvable groups. Bartels, Farrell, and L\"uck prove that virtually poly-$\Z$-groups satisfy the Farrell-Jones Conjecture (see \cite[Theorem~1.1]{BFL14}). In \cite{FW14} Farrell and Wu show that solvable Baumslag-Solitar groups satisfy the Farrell-Jones Conjecture.

In this article we modify the results of \cite{FL03} and show that the Farrell-Jones Conjecture holds for all virtually solvable groups if the semi-direct products $G_w := \Z[w,w^{-1}] \rtimes_{\cdot w} \Z$ with $w$ a non-zero algebraic number satisfy the Farrell-Jones Conjecture with finite wreath product (see section~\ref{sec-reducing} and, in particular, Proposition~\ref{prop-sdp}). This is mainly done by using inheritance properties, in particular for extensions (see Propositions~\ref{prop-FJCw-1} and \ref{prop-FJCw-2}).

Note that $G_w$ is isomorphic to the Baumslag-Solitar group $BS(1,w)$ if $w$ is a natural number. Unfortunately, the methods in \cite{FW14} do not directly carry over to the general case $w$ a non-zero algebraic number, even not to the case $w$ a rational number.
The main idea to overcome this problem is to combine two methods of proof for the Farrell-Jones Conjecture:
\begin{enumerate}
\item The Farrell-Hsiang method. This is the method used in the proofs for virtually poly-$\Z$-groups and solvable Baumslag-Solitar groups.
\item Transfer reducibility. This method using the geodesic flow was pioneered by Farrell-Jones and later extended to hyperbolic groups (see \cite{BLR08b}) and CAT(0)-groups (see \cite{BL12a}, \cite{Weg12}).
\end{enumerate}
We call groups which satisfy the combined method Farrell-Hsiang-Jones groups (''FHJ groups'' for short, see Definition~\ref{def-FHJ}). We prove that the Farrell-Jones Conjecture holds for FHJ groups (see Proposition~\ref{prop-FHJ} and Corollary~\ref{cor-FHJ}). Finally, we show for all non-zero algebraic numbers $w$ that the group $G_w$ is a FHJ group (see Proposition~\ref{prop-final}). The proof requires a variety of techniques, for example, geodesic flows on CAT(0)-spaces, group actions on trees, classical algebraic number theory.

This article is organized as follows. We begin with a brief survey of the Farrell-Jones Conjecture (see section~\ref{sec-FJC}). In section~\ref{sec-reducing} we show that the Farrell-Jones Conjecture holds for all virtually solvable groups if the semi-direct products $G_w := \Z[w,w^{-1}] \rtimes_{\cdot w} \Z$ with $w$ a non-zero algebraic number satisfy the Farrell-Jones Conjecture with finite wreath product. The section~\ref{sec-FHJ} is dedicated to FHJ groups. We finally prove the Farrell-Jones Conjecture with finite wreath product for the groups $G_w$ in section~\ref{sec-f}.

\section{The Farrell-Jones Conjecture} \label{sec-FJC}

This section contains a brief survey of the Farrell-Jones Conjecture. The original conjecture was stated in \cite[section 1.6]{FJ93}. Our formulations differ from the original ones and are more general.

Let $G$ be a group and let $\mathcal{F}$ be a family of subgroups (i.e., a set of subgroups which is closed under conjugation and taking subgroups). The original conjecture uses the family of virtually cyclic subgroups.

\begin{definition}[$K$-theoretic Farrell-Jones Conjecture]
We say that $G$ satisfies the \emph{$K$-theoretic Farrell-Jones Conjecture with respect to the family $\mathcal{F}$} if the assembly map
\begin{equation}
H^G_m(E_{\mathcal{F}}G;\mathbf{K}_\mathcal{A}) \to H^G_m(\pt;\mathbf{K}_\mathcal{A}) \cong K_m(\int_G \mathcal{A}) \label{assembly-K}
\end{equation}
induced by the projection $E_\mathcal{F}G \to \pt$ is an isomorphism for all $m \in \Z$ and every additive $G$-category $\mathcal{A}$.
\end{definition}
Here $E_{\mathcal{F}}G$ denotes the classifying space of the group $G$ with respect to the family $\mathcal{F}$. Any additive $G$-category $\mathcal{A}$ induces a covariant functor $\mathbf{K}_\mathcal{A}$ from the orbit category of $G$ to the category of spectra with (strict) maps of spectra as morphisms (see \cite[Definition 3.1]{BR07}). We denote the associated $G$-homology theory by $H^G_*(-;\mathbf{K}_\mathcal{A})$ (see \cite[sections 4 and 7]{DL98}).

There is also an $L$-theoretic version of the Farrell-Jones Conjecture.
\begin{definition}[$L$-theoretic Farrell-Jones Conjecture]
We say that $G$ satisfies the \emph{$L$-theoretic Farrell-Jones Conjecture with respect to the family $\mathcal{F}$} if the assembly map
\begin{equation}
H^G_m(E_{\mathcal{F}}G;\mathbf{L}^{\langle -\infty \rangle}_\mathcal{A}) \to H^G_m(\pt;\mathbf{L}^{\langle -\infty \rangle}_\mathcal{A}) \cong L^{\langle -\infty \rangle}_m(\int_G \mathcal{A}) \label{assembly-L}
\end{equation}
induced by the projection $E_\mathcal{F}G \to \pt$ is an isomorphism for all $m \in \Z$ and every additive $G$-category $\mathcal{A}$ with involution.
\end{definition}

We will use the following abbreviation.
\begin{notation}[FJC]
A group $G$ satisfies \emph{FJC} with respect to $\mathcal{F}$ if the $K$- and $L$-theoretic Farrell-Jones Conjecture with respect to $\mathcal{F}$ hold for $G$.
If the family $\mathcal{F}$ is not mentioned, it is by default the family of virtually cyclic subgroups.
\end{notation}

In the following subsections we will present two well-known methods to prove the Farrell-Jones Conjecture for specific groups: The Farrell-Hsiang method and transfer reducibility. We will end this section with a survey of the Farrell-Jones Conjecture with finite wreath products. For more information on the Farrell-Jones Conjecture we refer to the survey article \cite{LR05}.

\subsection{The Farrell-Hsiang method}

Recall that a (finite) group $H$ is said to be \emph{hyper-elementary} if it can be written as an extension $1 \to C \to H \to P \to 1$, where $C$ is a cyclic group and $P$ is a $p$-group for some prime number $p$ which does not divide the order of $C$.

\begin{definition}[Farrell-Hsiang group] \label{def-FH}
Let $G$ be a finitely generated group and let $\mathcal{F}$ be a family of subgroups.
We call $G$ a \emph{Farrell-Hsiang group with respect to the family $\mathcal{F}$} if the following holds for a (and hence any) fixed word metric $d_G$:\\
There exists a natural number $N$ and surjective homomorphisms $\alpha_n \colon G \to F_n$ ($n \in \N$) onto finite groups $F_n$ such that the following condition is satisfied. For any hyper-elementary subgroup $H$ of $F_n$ there exist a simplicial complex $E_H$ of dimension at most $N$ with a cell preserving simplicial $\alpha_n^{-1}(H)$-action whose stabilizers belong to $\mathcal{F}$ and an $\alpha_n^{-1}(H)$-equivariant map $f_H \colon G \to E_H$ such that $d_G(g,h) < n$ implies $d^1_{E_H}(f_H(g),f_H(h)) < 1/n$ for all $g, h \in G$. (Here $d^1_{E_H}$ denotes the $l^1$-metric on $E_H$.)
\end{definition}

\begin{remark}
Suppose that $\mathcal{F}$ is closed under overgroups of finite index, i.e., $H<K<G$ with $H \in \mathcal{F}$ and $[K:H]<\infty$ implies $K \in \mathcal{F}$.
Then we can drop the requirement "cell preserving" for the simplicial $\alpha_n^{-1}(H)$-action in the definition above, since we can achieve a cell preserving action by passing to the barycentric subdivision of $E_H$ (compare \cite[Lemma 5.5]{BLRR14}).
\end{remark}

The following result is taken from \cite[Theorem 1.2]{BL12c}.
\begin{proposition} \label{prop-FH-elem}
Let $G$ be a Farrell-Hsiang group with respect to a family $\mathcal{F}$ of subgroups.
Then $G$ satisfies FJC with respect to $\mathcal{F}$.
\end{proposition}

\subsection{Transfer reducibility}

In this subsection we consider a further method for proving the Farrell-Jones Conjecture: Transfer reducibility. The definition of transfer reducibility needs some preparation.

\begin{definition}[strong homotopy action]
A \emph{strong homotopy action} of a group $G$ on a topological space $X$ is a continuous map
\[
 \Psi \colon \coprod_{j=0}^\infty \big( (G \times [0,1])^j \times G \times X \big) \to X
\]
with the following properties:
\begin{enumerate}
 \item $\Psi(\ldots,g_l,0,g_{l-1},\ldots) = \Psi(\ldots,g_l,\Psi(g_{l-1},\ldots))$
 \item $\Psi(\ldots,g_l,1,g_{l-1},\ldots) = \Psi(\ldots,g_l \cdot g_{l-1},\ldots)$
 \item $\Psi(e,t_j,g_{j-1},\ldots) = \Psi(g_{j-1},\ldots)$
 \item $\Psi(\ldots,t_l,e,t_{l-1},\ldots) = \Psi(\ldots,t_l \cdot t_{l-1},\ldots)$
 \item $\Psi(\ldots,t_1,e,x) = \Psi(\ldots,x)$
 \item $\Psi(e,x) = x$
\end{enumerate}
\end{definition}

\begin{examples} \label{ex-sha}
\begin{enumerate}
\item Every action of a group $G$ on a topological space $X$ induces a strong homotopy action by \label{ex-sha-1}
\[
\Psi(g_j,t_j,g_{j-1},t_{j-1},\ldots,g_1,t_1,g_0,x) := g_j g_{j-1} \cdots g_1 g_0 x.
\]
\item More generally, strong homotopy actions appear in the following situation: Let $Y$ be a $G$-space and let $H \colon Y \times [0,1] \to Y$ be a deformation retraction onto a subspace $X \subseteq Y$ (i.e., $H_0(Y) = X$, $H_0|_X = \id_X$ and $H_1 = \id_Y$) such that $H_t \circ H_{t'} = H_{t \cdot t'}$ for all $t,t' \in [0,1]$. In this situation we define \label{ex-sha-2}
\[
 \Omega \colon \coprod_{j=0}^\infty \big( (G \times [0,1])^j \times G \times X \big) \to Y
\]
inductively by $\Omega(g_0,x) := g_0 x$ and $\Omega(g_j,t_j,g_{j-1},\dots) := g_j H_{t_j}(\Omega(g_{j-1},\ldots))$ for $j \geq 1$. Then $\Psi := H_0 \circ \Omega$ is a strong homotopy action.
\end{enumerate}
\end{examples}

\begin{definition}[metric $d_{\Psi,S,k,\Lambda}$]
Let $\Psi$ be a strong homotopy $G$-action on a metric space $(X,d_X)$. Let $S \subseteq G$ be a finite symmetric subset which contains the trivial element $e \in G$. Let $k$ be a natural number.
\begin{enumerate}
 \item For $g \in G$ we define $F_g(\Psi,S,k) \subset \map(X,X)$ by
       \[
        F_g(\Psi,S,k) := \big\{ \Psi(g_k,t_k,\ldots,g_0,?) \colon X \to X \, \big| \, g_i \in S, t_i \in [0,1], g_k \cdot \ldots \cdot g_0 = g \big\}.
       \]
 \item For $(g,x) \in G \times X$ we define $S^1_{\Psi,S,k}(g,x) \subset G \times X$ as the subset consisting of all $(h,y) \in G \times X$ with the following property: There are $a,b \in S$, $f \in F_a(\Psi,S,k)$ and $\tilde{f} \in F_b(\Psi,S,k)$ such that $f(x)=\tilde{f}(y)$ and $h = g a^{-1} b$. For $n \in \N^{\geq 2}$ we set
       \[
        S^n_{\Psi,S,k}(g,x) := \big\{ S^1_{\Psi,S,k}(h,y) \, \big| \, (h,y) \in S^{n-1}_{\Psi,S,k}(g,x) \big\}.
       \]
 \item For $\Lambda \in \R^{>0}$ we define the quasi-metric $d_{\Psi,S,k,\Lambda}$ on $G \times X$ as the largest quasi-metric on $G \times X$ satisfying
     \begin{itemize}
      \item $d_{\Psi,S,k,\Lambda}\big((g,x),(g,y)\big) \leq \Lambda \cdot d_X(x,y)$ for all $g \in G$, $x,y \in X$ and
      \item $d_{\Psi,S,k,\Lambda}\big((g,x),(h,y)\big) \leq 1$ for all $(h,y) \in S^1_{\Psi,S,k}(g,x)$.
     \end{itemize}
\end{enumerate}
\end{definition}
We remind the reader that the difference between a metric and a quasi-metric is that in the latter case the distance $\infty$ is allowed. Note that the quasi-metric $d_{\Psi,S,k,\Lambda}$ is $G$-invariant with respect to the $G$-action $g(h,x) := (gh,x)$ on $G \times X$. See \cite[Definition~3.4]{BL12a} for a construction of the quasi-metric $d_{\Psi,S,k,\Lambda}$. The following lemma is taken from \cite[Lemma 3.5]{BL12a}.
\begin{lemma} \label{lem-qm}
\begin{enumerate}
\item The subset $S$ generates $G$ if and only if $d_{\Psi,S,k,\Lambda}$ is a metric.\label{lem-qm1}
\item Let $(g,x),(h,y) \in G \times X$ and $n \in \N$. Then $(h,y) \in S^n_{\Psi,S,k}(g,x)$ if and only if $d_{\Psi,S,k,\Lambda}((g,x),(h,y)) \leq n$ for all $\Lambda > 0$.\label{lem-qm2}
\item The topology on $G \times X$ induced by $d_{\Psi,S,k,\Lambda}$ coincides with the product topology.\label{lem-qm3}
\end{enumerate}
\end{lemma}

\begin{example}
If the strong homotopy action $\Psi$ is given by an action of a group $G$ (see Example~\ref{ex-sha} (\ref{ex-sha-1})) then
\[
S^n_{\Psi,S,k}(g,x) = \big\{ (gk,k^{-1}x) \, \big| \, k \in S^{2n} \big\}.
\]
\end{example}

\begin{definition}[almost strongly transfer reducible] \label{def-astr}
Let $\mathcal{F}$ be a family of subgroups of $G$.
The group $G$ is called \emph{almost strongly transfer reducible over $\mathcal{F}$} if there exists a natural number $N$ with the following property:
For every finite symmetric subset $S \subseteq G$ containing the trivial element $e \in G$ and every natural numbers $k, n$ there are
\begin{itemize}
 \item a compact, contractible, controlled $N$-dominated metric space $X$,
 \item a strong homotopy $G$-action $\Psi$ on $X$ and
 \item a $G$-invariant cover $\mathcal{U}$ of $G \times X$ by open sets
\end{itemize}
such that
\begin{enumerate}
 \item for all $U \in \mathcal{U}$ the subgroup $G_U := \{ g \in G \mid gU = U \}$ belongs to $\mathcal{F}$,
 \item $\dim(\mathcal{U}) \leq N$,
 \item for every $(g,x) \in G \times X$ there exists $U \in \mathcal{U}$ with $S^n_{\Psi,S,k}(g,x) \subseteq U$.
\end{enumerate}
\end{definition}
Recall that a metric space $X$ is \emph{controlled $N$-dominated} if for every $\epsilon > 0$ there is a finite CW-complex $K$ of dimension at most $N$, maps $i \colon X \to K$, $p \colon K \to X$, and a homotopy $H \colon X \times [0,1] \to X$ between $p \circ i$ and $\id_X$ such that for every $x \in X$ the diameter of $\{H(x,t) \mid t \in [0,1]\}$ is at most $\epsilon$.

We are interested in transfer reducibility because of the following proposition (see \cite[Proposition 5.4 (ii)]{BLRR14} for a proof).
\begin{proposition} \label{prop-astr}
Let $G$ be a group which is almost strongly transfer reducible over a family $\mathcal{F}$ of subgroups of $G$.
Suppose that $\mathcal{F}$ is virtually closed, i.e., $K < H < G$ with $K \in \mathcal{F}$ and $[H:K] < \infty$ implies $H \in \mathcal{F}$.
Then $G$ satisfies FJC with respect to $\mathcal{F}$.
\end{proposition}

\begin{examples} \label{ex-astr}
\begin{enumerate}
 \item Hyperbolic groups are almost strongly transfer reducible over the family of virtually cyclic subgroups. See \cite[Proposition~2.1]{BL12a}.\label{ex-astr-1}
 \item CAT(0)-groups are almost strongly transfer reducible over the family of virtually cyclic subgroups. Hence, CAT(0)-groups satisfy FJC. See \cite[Theorem~3.4]{Weg12}.\label{ex-astr-2}
\end{enumerate}
\end{examples}

\subsection{The Farrell-Jones Conjecture with finite wreath products}

Let $G$ and $H$ be groups. The (restricted) wreath product $G \wr H$ is defined as the semi-direct product $(\bigoplus_{h \in H} G) \rtimes_\phi H$, where the map $\phi \colon \, H \to \Aut(\bigoplus_{h \in H} G)$ is given by $\phi(k)((g_h)_{h \in H}) := (g_{hk})_{h \in H}$.

\begin{definition}[FJCw]
We say that a group $G$ satisfies the \emph{Farrell-Jones Conjecture with finite wreath product} if for any finite group $F$ the wreath product $G \wr F$ satisfies the $K$- and $L$-theoretic Farrell-Jones Conjecture (with respect to the family of virtually cyclic subgroups).
We will use the abbreviation \emph{FJCw} for "Farrell-Jones Conjecture with finite wreath product".
\end{definition}

The main reason to study FJCw (instead of FJC) is the fact that FJCw is virtually closed (i.e., closed under overgroups of finite index, see Proposition~\ref{prop-FJCw-1}). The Farrell-Jones Conjecture with finite wreath product has first been used in \cite{FR00}, see also \cite[Definition 2.1]{Rou08}.

\begin{examples} \label{ex-FJCw-1}
\begin{enumerate}
\item Since the class of CAT(0)-groups is closed under finite wreath products, we conclude from Example~\ref{ex-astr}~(\ref{ex-astr-2}) that CAT(0)-groups satisfy FJCw.
\item The following statement is the main result of \cite{BLRR14}: Let $R$ be a ring whose underlying abelian group is finitely generated. Let $G$ be a group which is commensurable to a subgroup of $GL_n(R)$ for some $n \in \N$. Then $G$ satisfies FJCw.
\end{enumerate}
\end{examples}

The Farrell-Jones Conjecture with finite wreath product has useful inheritance properties.
\begin{proposition}[Inheritance properties, part 1] \label{prop-FJCw-1}
\begin{description}
\item[Subgroups] Let $G$ be a group which satisfies FJCw. Then every subgroup $H < G$ satisfies FJCw.
\item[Direct products] If $G_1$ and $G_2$ satisfy FJCw, then the direct product $G_1 \times G_2$ satisfies FJCw.
\item[Finite wreath products] Let $F$ be a finite group. If $G$ satisfies FJCw then $G \wr F$ satisfies FJCw.
\item[Overgroups of finite index] Let $G < \Gamma$ be subgroup of finite index. If $G$ satisfies FJCw then $\Gamma$ satisfies FJCw.
\item[Colimits] Let $\{G_i \mid i \in I\}$ be a directed system of groups (with not necessarily injective structure maps). If each $G_i$ satisfies FJCw, then the colimit $\colim_{i \in I} G_i$ satisfies FJCw.
\end{description}
\end{proposition}
\begin{proof}
Note that \cite[Lemma~2.3]{BL12a} also holds for the class $\mathcal{FJ}^K$ of groups satisfying the $K$-theoretic Farrell-Jones Conjecture by the remark below the proof of \cite[Lemma~2.3]{BL12a} and \cite[Corollary~1.2]{Weg12}.
Therefore, the inheritance properties ''subgroups'', ''direct products'', and ''colimits'' are true for FJC (instead of FJCw).
The inheritance properties for FJCw can be proven as follows.
\begin{description}
\item[Subgroups] Note that $H \wr F$ is a subgroup of $G \wr F$ for every finite group $F$.
\item[Direct products] This follows from the fact that $(G_1 \times G_2) \wr F$ is a subgroup of $(G_1 \wr F) \times (G_2 \wr F)$.
\item[Finite wreath products] Since $(G \wr F) \wr F'$ is a subgroup of $G^{F \times F'} \wr (F \wr F')$ (compare \cite[Proof of Lemma~2.5]{FL03}), this follows from the fact that FJCw is closed under subgroups and direct products.
\item[Overgroups of finite index] There exists a subgroup $N < G$ of finite index which is normal in $\Gamma$. By \cite[Lemma~2.6]{FL03}, $N \wr \Gamma/N$ has a subgroup isomorphic to $\Gamma$. Now, the statement follows since FJCw is closed under subgroups and finite wreath products.
\item[Colimits] Note that $(\colim_{i \in I} G_i) \wr F = \colim_{i \in I} (G_i \wr F)$.
\end{description}
\end{proof}

\begin{example} \label{ex-FJCw-2}
Finitely generated abelian groups are CAT(0)-groups and hence satisfy FJCw (see Example~\ref{ex-FJCw-1}). Since FJCw is closed under colimits (see Proposition~\ref{prop-FJCw-1}), all abelian groups satisfy FJCw.
\end{example}

The author could not find the following proposition in the literature.
\begin{proposition} \label{prop-FH}
Let $G$ be a Farrell-Hsiang group with respect to the family $\mathcal{F}_G := \{ H < G \mid H \mbox{ satisfies FJCw} \}$. Then $G$ satisfies FJCw.
\end{proposition}
\begin{proof}
Suppose that $G$ is a Farrell-Hsiang group with respect to the family $\mathcal{F}_G$.
Let $F$ be a finite group. We will show that $G \wr F$ is a Farrell-Hsiang group with respect to $\mathcal{F}_{G \wr F} := \{ H < G \wr F \mid H \mbox{ satisfies FJCw} \}$. Then the statement follows from Proposition~\ref{prop-FH-elem} and the transitivity principle described in \cite[Theorem~2.10]{BFL14}.

We fix a finite symmetric generating set $S_G$ for $G$ and define the following finite symmetric generating sets: $S_{G^F} := \{(s_f)_{f \in F} \in G^F \mid s_f \in S_G\}$, $S_F := F$, $S_{G \wr F} := \{ (x,f) \mid x \in S_{G^F}, f \in F \}$.
We denote the respective word length functions by $l_G$, $l_{G^F}$, $l_{G \wr F}$ and the respective word metrics by $d_G$, $d_{G^F}$, $d_{G \wr F}$.

Since $G$ is a Farrell-Hsiang group, there exist a natural number $N_G$ and surjective homomorphisms $\alpha_{G,n} \colon G \to F_{G,n}$ ($n \in \N$) onto finite groups $F_{G,n}$ such that the following condition is satisfied: For any hyper-elementary subgroup $H$ of $F_{G,n}$ there exist a simplicial complex $E_{G,H}$ of dimension at most $N_G$ with a simplicial $\alpha_{G,n}^{-1}(H)$-action whose stabilizers belong to $\mathcal{F}_G$ and an $\alpha_{G,n}^{-1}(H)$-equivariant map $f_{G,H} \colon G \to E_{G,H}$ such that $d^1_{E_{G,H}}(f_{G,H}(g),f_{G,H}(h)) < 1/n$ if $d_G(g,h) < n$.

We define $\alpha_{G \wr F,n} \colon G \wr F \to F_{G,n|F|} \wr F =: F_{G \wr F,n}$ to be the surjective homomorphism induced by $\alpha_{G,n|F|}$, i.e., $\alpha_{G \wr F,n}((g_f)_{f \in F}, f') := ((\alpha_{G,n|F|}(g_f))_{f \in F}, f')$.

Let $H < F_{G \wr F,n}$ be a hyper-elementary subgroup. We will construct a simplicial complex $E_{G \wr F,H}$ of dimension at most $|F| \cdot (N_G + 1)^{|F|} - 1$ with a simplicial $\alpha_{G \wr F,n}^{-1}(H)$-action whose stabilizers belong to $\mathcal{F}_{G \wr F}$ and an $\alpha_{G \wr F,n}^{-1}(H)$-equivariant map $f_{G \wr F,H} \colon G \wr F \to E_{G,H}$ such that $d^1_{E_{G \wr F,H}}(f_{G \wr F,H}(g),f_{G \wr F,H}(h)) < 1/n$ if $d_{G \wr F}(g,h) < n$.

Since the class of hyper-elementary groups is closed under taking subgroups and quotients, the groups $H_f := \pr_f(H \cap F_{G,n|F|}^F) < F_{G,n|F|}$ are hyper-elementary. We define a simplicial complex $E'_{G \wr F,H}$ as follows. The vertices of $E'_{G \wr F,H}$ are tuples $(v_f)_{f \in F}$, where $v_f$ is a vertex of $E_{G,H_f}$. Distinct vertices $v^{(0)}, \ldots, v^{(m)}$ span an $m$-simplex if and only if for all $f \in F$ the vertices $v^{(0)}_f, \ldots, v^{(m)}_f$ are contained in a common simplex of $E_{G,H_f}$. The simplicial complex $E'_{G \wr F,H}$ is of dimension at most $(N_G+1)^{|F|}-1$.

For every $f \in \pr_F(H)$ we fix an element $h_f \in \alpha_{G \wr F,n}^{-1}(H)$ with $\pr_F(h_f) = f$. A simplicial $\alpha_{G \wr F,n}^{-1}(H)$-action on $\coprod_{f \in F} E'_{G \wr F,H} = F \times E'_{G \wr F,H}$ is defined as follows. Let $g \in \alpha_{G \wr F,n}^{-1}(H)$, $f' \in F$, and let $(v_f)_{f \in F}$ be a vertex in $E'_{G \wr F,H}$. We define
\[
g \big(f',(v_f)_{f \in F}\big) := \big(\pr_F(g)f',(\pr_f(g')v_f)_{f \in F}\big)
\]
with $g' := h_{\pr_F(g)f'}^{-1} \cdot g \cdot h_{f'} \in \alpha_{G \wr F,n}^{-1}(H) \cap G^F = \prod_{f \in F} \alpha_{G,n|F|}^{-1}(H_f)$.

We fix a right coset transversal $T \subset F$ for $G^F \rtimes \pr_F(H) < G \wr F$. Then every element in $G \wr F$ can be uniquely written as $h_{\overline{f}} \cdot g \cdot t$ with $\overline{f} \in \pr_F(H)$, $g = (g_f)_{f \in F} \in G^F$, and $t \in T$. We define a map $f'_{G \wr F,H} \colon G \wr F \to F \times E'_{G \wr F,H}$ by
\[
f'_{G \wr F,H}(h_{\overline{f}} \cdot g \cdot t) = \Big(\overline{f},i\big((f_{G,H_f}(g_f))_{f \in F}\big)\Big),
\]
where $i \colon \prod_{f \in F} E_{G,H_f} \to E'_{G \wr F,H}$ is given by
\[
\Big(\displaystyle\sum_{v_f \text{ vertex in } E_{G,H_f}} t_{v_f} \cdot v_f\Big)_{f \in F} \mapsto \sum_{(v_f)_{f \in F} \text{ vertex in } E'_{G \wr F,H}} \prod_{f \in F} t_{v_f} \cdot (v_f)_{f \in F}
\]
($t_{v_f} \geq 0$, $\sum_{v_f} t_{v_f} = 1$).
The map $f'_{G \wr F,H}$ is $\alpha_{G \wr F,n}^{-1}(H)$-equivariant.

We would like to define a map $f_{G \wr F,H} \colon G \wr F \to F \times E'_{G \wr F,H}$ by $f_{G \wr F,H}(x) := \sum_{f \in F} \frac{1}{|F|} f'_{G \wr F,H} (xf)$. But in general, this linear combination will not lie in a simplex of $F \times E'_{G \wr F,H}$. That is why we define the simplicial complex $E_{G \wr F,H}$ to be the smallest simplicial complex with the same vertices as $F \times E'_{G \wr F,H}$ such that every linear combination $\sum_{f \in F} \frac{1}{|F|} f'_{G \wr F,H} (xf)$ ($x \in G \wr F$) lies in $E_{G \wr F,H}$. Then the map
\[
f_{G \wr F,H} \colon G \wr F \to E_{G \wr F,H}, x \mapsto \sum_{f \in F} \frac{1}{|F|} f'_{G \wr F,H} (xf)
\]
is well defined. The simplicial complex $E_{G \wr F,H}$ is of dimension at most $|F| \cdot (N_G + 1)^{|F|} - 1$.

It remains to show that $d^1_{E_{G \wr F,H}}(f_{G \wr F,H}(x),f_{G \wr F,H}(y)) < 1/n$ if $d_{G \wr F}(x,y) < n$. Let $g,g' \in G^F$ and $f,f' \in F$ with $x = gf$ and $y = g'f'$. Then $f_{G \wr F,H}(x) = f_{G \wr F,H}(g)$ and $f_{G \wr F,H}(y) = f_{G \wr F,H}(g')$. For every $\overline{f} \in F$ there exist $a(\overline{f}) \in \pr_F(H)$, $b(\overline{f}), c(\overline{f}) \in G^F$, and $d(\overline{f}) \in T \subset F$ such that $g \overline{f} = h_{a(\overline{f})} b(\overline{f}) d(\overline{f})$ and $g'\overline{f} = h_{a(\overline{f})} c(\overline{f}) d(\overline{f})$.

We conclude
\begin{eqnarray*}
f_{G \wr F,H}(g) & = & \frac{1}{|F|} \sum_{\overline{f} \in F} \Big(a(\overline{f}),i\big((f_{G,H_f}(b(\overline{f})_f))_{f \in F}\big)\Big), \\
f_{G \wr F,H}(h) & = & \frac{1}{|F|} \sum_{\overline{f} \in F} \Big(a(\overline{f}),i\big((f_{G,H_f}(c(\overline{f})_f))_{f \in F}\big)\Big).
\end{eqnarray*}
Hence,
\begin{align*}
& d^1_{E_{G \wr F,H}}(f_{G \wr F,H}(g),f_{G \wr F,H}(h)) \\
& \leq \frac{1}{|F|} \sum_{\overline{f} \in F} d^1_{E'_{G \wr F,H}}\Big(i\big((f_{G,H_f}(b(\overline{f})_f))_{f \in F}\big),i\big((f_{G,H_f}(c(\overline{f})_f))_{f \in F}\big)\Big) \\
& \leq \frac{1}{|F|} \sum_{\overline{f}, f \in F} d^1_{E_{G,H_f}}\big(f_{G,H_f}(b(\overline{f})_f),f_{G,H_f}(c(\overline{f})_f)\big).
\end{align*}
We calculate $d_{G \wr F}(x,y) = l_{G \wr F}(y^{-1}x) = l_{G \wr F}({f'}^{-1}{g'}^{-1}gf)$. The construction of the generating sets $S_{G \wr F}$ and $S_{G^F}$ -- in particular, the invariance under conjugation with elements in $F$ -- implies
\[
l_{G \wr F}({f'}^{-1}{g'}^{-1}gf) = l_{G \wr F}({g'}^{-1}gf{f'}^{-1}) \geq l_{G^F}({g'}^{-1}g).
\]
Hence, $l_{G^F}({g'}^{-1}g) \leq d_{G \wr F}(x,y) < n$.
Since
\begin{align*}
& d_G\big(b(\overline{f})_f,c(\overline{f})_f\big) = l_G\big(c(\overline{f})_f^{-1} b(\overline{f})_f\big) \leq l_{G^F}\big(c(\overline{f})^{-1} b(\overline{f})\big) = \\
& l_{G^F}\big(d(\overline{f}) \overline{f}^{-1} g'^{-1} g \overline{f} d(\overline{f})^{-1}\big) = l_{G^F}(g'^{-1} g) < n \leq n \cdot |F|,
\end{align*}
we conclude
\[
d^1_{E_{G,H_f}}\big(f_{G,H_f}(b(\overline{f})_f),f_{G,H_f}(c(\overline{f})_f)\big) < \frac{1}{n |F|}.
\]
Therefore,
\begin{align*}
& d^1_{E_{G \wr F,H}}\big(f_{G \wr F,H}(x),f_{G \wr F,H}(y)\big) = d^1_{E_{G \wr F,H}}\big(f_{G \wr F,H}(g),f_{G \wr F,H}(h)\big) \\
& \leq \frac{1}{|F|} \sum_{\overline{f}, f \in F} d^1_{E_{G,H_f}}\big(f_{G,H_f}(b(\overline{f})_f),f_{G,H_f}(c(\overline{f})_f)\big) \\
& < \frac{1}{|F|} \sum_{\overline{f}, f \in F} \frac{1}{n |F|} = \frac{1}{n}.
\end{align*}
This finishes the proof of Proposition~\ref{prop-FH}.
\end{proof}

\begin{proposition}
Let $G$ be a group which is almost strongly transfer reducible over $\mathcal{F}_G := \{ H < G \mid H \mbox{ satisfies } FJCw \}$. Then $G$ satisfies FJCw.
\end{proposition}
\begin{proof}
Suppose that $G$ is almost strongly transfer reducible over $\mathcal{F}_G$.
Let $F$ be a finite group. We have to show that $G \wr F$ satisfies FJC.
\cite[Theorem~5.1 and Remark~5.7]{BLRR14} imply that $G \wr F$ satisfies FJC with respect to the family $\mathcal{F}^\wr$ which is defined as follows. The family $\mathcal{F}^\wr$ consists of those subgroups $H < G \wr F$ which contain a subgroup of finite index that is isomorphic to a subgroup of $H_1 \times \ldots \times H_n$ for some $n$ and $H_1, \ldots, H_n \in \mathcal{F}$. Using Proposition~\ref{prop-FJCw-1} we conclude that every group in $\mathcal{F}^\wr$ satisfies FJCw. \cite[Theorem~2.10]{BFL14} implies that $G \wr F$ satisfies FJC.
\end{proof}

\begin{example} \label{ex-FJCw-3}
We conclude from Example~\ref{ex-astr}~(\ref{ex-astr-1}) that hyperbolic groups satisfy FJCw. In particular, finitely generated free groups satisfy FJCw. Since FJCw is closed under colimits (see Proposition~\ref{prop-FJCw-1}), all free groups satisfy FJCw.
\end{example}

\begin{proposition}[Inheritance properties, part 2] \label{prop-FJCw-2}
\begin{description}
\item[Extensions] Let
    \[
    \xymatrix{1 \ar[r] & K \ar[r] & G \ar[r]^{p} & Q \ar[r] & 1}
    \]
    be a short exact sequence. If $K$, $Q$, and $p^{-1}(C)$ satisfy FJCw for every infinite cyclic subgroup $C < Q$ then $G$ satisfies FJCw.
\item[Free products] If $G_1$ and $G_2$ satisfy FJCw, then the free product $G_1 \ast G_2$ satisfies FJCw.
\end{description}
\end{proposition}
\begin{proof}
\begin{description}
\item[Extensions] Let $F$ be a finite group. The map $p$ induces a map $\tilde{p} \colon \, G \wr F \to Q \wr F$.
    Note that \cite[Lemma~2.3~(ii)]{BL12a} also holds for the class $\mathcal{FJ}^K$ of groups satisfying the $K$-theoretic Farrell-Jones Conjecture by the remark below the proof of \cite[Lemma~2.3]{BL12a} and \cite[Corollary~1.2]{Weg12}. Hence, it suffices to show that $\tilde{p}^{-1}(V)$ satisfies FJC for every virtually cyclic subgroup $V < Q \wr F$.
    The subgroup $V \cap Q^F < V$ is of finite index. Since the class of virtually cyclic groups is closed under subgroups and quotients, the groups $\pr_f(V \cap Q^F) < Q$ $(f \in F)$ are virtually cyclic. Hence, there exist subgroups $C_f < \pr_f(V \cap Q^F)$ of finite index such that $C_f$ is either infinite cyclic or trivial. By assumption the groups $p^{-1}(C_f)$ satisfy FJCw. Since $\prod_{f \in F} C_f < V$ is a subgroup of finite index, $\prod_{f \in F} p^{-1}(C_f) = \tilde{p}^{-1}(\prod_{f \in F} C_f) < \tilde{p}^{-1}(V)$ is a subgroup of finite index. Using Proposition~\ref{prop-FJCw-1} we conclude that $\tilde{p}^{-1}(V)$ satisfies FJCw.
\item[Free products] We consider the surjective homomorphism $p \colon \, G_1 \ast G_2 \to G_1 \times G_2$. It remains to check that $\ker(p)$ and $p^{-1}(C)$ satisfy FJCw for every infinite cyclic subgroup $C < G_1 \times G_2$. By \cite[Lemma 5.2]{Rou08} those groups are free and hence satisfy FJCw (see Example~\ref{ex-FJCw-3}).
\end{description}
\end{proof}

\section{Reducing the class of solvable groups} \label{sec-reducing}

In \cite{FL03} Farrell and Linnell study isomorphism conjectures for solvable groups. The following result is due to Farrell and Jones (compare \cite[Theorem 1.2]{FL03}).
\begin{proposition} \label{prop-FL}
Suppose that FJCw is true for every nearly crystallographic group. Then every solvable group satisfies FJCw.
\end{proposition}

Recall that a \emph{nearly crystallographic group} $G$ is a finitely generated group which can be written as a semi-direct product $G = A \rtimes C$ such that $A \lhd G$ is isomorphic to a subgroup of $\Q^n$ for some $n \in \N$, $C < G$ is virtually cyclic, and the conjugation action of $C$ on $A$ makes $A \otimes \Q$ into an irreducible $\Q C$-module.

Farrell and Linnell prove Proposition~\ref{prop-FL} for a related conjecture, the Fibered Isomorphism Conjecture (FIC), instead of FJCw (see \cite[Theorem 1.2]{FL03}). The proof is based on properties of FIC which also hold for FJCw. Hence, the proof can be directly transferred to FJCw.

For our purpose we need the similar Proposition~\ref{prop-sdp} stated below.

\begin{definition} \label{def-G_w}
Let $w \in \overline{\Q}^\times$ be a non-zero algebraic number. We define the group $G_w$ as the semi-direct product
\[
G_w := \Z[w,w^{-1}] \rtimes_{\cdot w} \Z.
\]
The multiplication in $G_w$ is given by
\[
(x_1,x_2) \cdot (y_1,y_2) := (x_1 + w^{x_2} y_1, x_2 + y_2)
\]
($x_1,y_1 \in \Z[w,w^{-1}]$, $x_2,y_2 \in \Z$).
\end{definition}

\begin{proposition} \label{prop-sdp}
Suppose that FJCw is true for the groups $G_w$ ($w \in \overline{\Q}^\times$). Then FJCw holds for all solvable groups.
\end{proposition}

The proof is similar to the proof of Proposition~\ref{prop-FL}. It is based on the following lemmata.

\begin{lemma} \label{lem-1}
FJCw holds for every semi-direct product $A \rtimes \Z$ with $A$ torsion abelian.
\end{lemma}

\begin{lemma} \label{lem-2}
The wreath product $\Z \wr \Z$ satisfies FJCw.
\end{lemma}

\begin{proof}[Proof of Lemma~\ref{lem-1} and \ref{lem-2}]
Farrell and Linnell prove Lemma~\ref{lem-1} and \ref{lem-2} for FIC, see \cite[Lemma 4.1, Corollary 4.2, Lemma 4.3]{FL03}. In the proofs they use the following properties of FIC:
\begin{itemize}
\item Let $p \colon \, \Gamma \twoheadrightarrow G$ be an epimorphism of groups such that FIC is true for $G$ and also for $p^{-1}(S)$ for all virtually cyclic subgroups $S < G$. Then FIC is true for $\Gamma$. (see \cite[Lemma 2.2]{FL03})
\item If FIC is true for a group $\Gamma$ then FIC is true for every subgroup of $\Gamma$. (see \cite[Lemma 2.3]{FL03})
\item Let $F$ be a finitely generated free group, let $X$ be a finite set, and let $S$ denote the symmetric group on $X$. Then FIC is true for $F \wr_X S := F^X \rtimes S$. (see \cite[Lemma 2.4]{FL03})
\item Every virtually abelian group satisfies FIC. (see \cite[Lemma 2.7]{FL03})
\item Let $\{G_i \mid i \in I\}$ be a directed system of groups. If each $G_i$ satisfies FIC, then the colimit $\colim_{i \in I} G_i$ satisfies FIC. (see \cite[Theorem 7.1]{FL03})
\end{itemize}
Since these properties are also satisfied for FJCw (see Proposition~\ref{prop-FJCw-1}, Example~\ref{ex-FJCw-2}, Example~\ref{ex-FJCw-3}, Proposition~\ref{prop-FJCw-2}), the proofs can be directly transferred to FJCw.
\end{proof}

\begin{lemma} \label{lem-3}
Suppose that FJCw holds for the groups $G_w$ ($w \in \overline{\Q}^\times$).
Then every semi-direct product $\Q^n \rtimes \Z$ satisfies FJCw.
\end{lemma}
\begin{proof}
At first we show that $\overline{\Q}^n \rtimes_{\phi_{n,w}} \Z$ ($n \in \N$, $w \in \overline{\Q}^\times$) with
\[
\phi_{n,w} \colon \, \overline{\Q}^n \to \overline{\Q}^n, \, (a_1, a_2, \ldots, a_n) \mapsto (w a_1, a_1 + w a_2, \ldots, a_{n-1} + w a_n)
\]
satisfies FJCw. We proceed by induction on $n$. As initial step we have to show that $\overline{\Q} \rtimes_{\phi_{1,w}} \Z = \overline{\Q} \rtimes_{\cdot w} \Z$ satisfies FJCw. Since $\Q$ is a directed colimit of subgroups isomorphic to $\Z$, we obtain
\[
\Q(w) \rtimes_{\cdot w} \Z = \Q[w,w^{-1}] \rtimes_{\cdot w} \Z = \colim \Z[w,w^{-1}] \rtimes_{\cdot w} \Z.
\]
By Proposition~\ref{prop-FJCw-1} (Colimits) this group satisfies FJCw.
Let $b_i$ ($i \in I$) be a basis of the $\Q(w)$-vector space $\overline{\Q}$. Since
\[
\overline{\Q} \rtimes_{\cdot w} \Z = \big( \bigoplus_{i \in I} b_i \Q(w) \big) \rtimes_{\cdot w} \Z < \bigoplus_{i \in I} \big( b_i \Q(w) \rtimes_{\cdot w} \Z \big) \cong \bigoplus_{i \in I} \big( \Q(w) \rtimes_{\cdot w} \Z \big),
\]
we conclude from Proposition~\ref{prop-FJCw-1} that $\overline{\Q} \rtimes_{\cdot w} \Z$ satisfies FJCw.

Induction step $n \to n+1$: Assume that $\overline{\Q}^n \rtimes_{\phi_{n,w}} \Z$ satisfies FJCw.
Consider the the map
\[
p \colon \, \overline{\Q}^{n+1} \rtimes_{\phi_{n+1,w}} \Z \twoheadrightarrow \overline{\Q}^n \rtimes_{\phi_{n,w}} \Z, \, (a_1, \ldots, a_{n+1}, z) \mapsto (a_1, \ldots, a_n, z).
\]
We will apply Proposition~\ref{prop-FJCw-2} (Extensions). Note that $\ker(p) = \overline{\Q}$ satisfies FJCw (see Example~\ref{ex-FJCw-2}). It remains to show that $p^{-1}(C)$ satisfies FJCw for every infinite cyclic subgroup $C < \overline{\Q}^n \rtimes_{\phi_{n,w}} \Z$. Let $(a,z) \in \overline{\Q}^{n+1} \rtimes_{\phi_{n+1,w}} \Z$ be a preimage of a generator of $C$. We obtain a short exact sequence
\[
1 \to \ker(p) \to p^{-1}(C) \to C \to 1.
\]
Note that $(a,z) \cdot x \cdot (a,z)^{-1} = x \cdot w^z$ for all $x \in \ker(p)$. Hence, $p^{-1}(C) \cong \overline{\Q} \rtimes_{\cdot w^z} \Z$ satisfies FJCw. So far, we have shown that FJCw holds for $\overline{\Q}^n \rtimes_{\phi_{n,w}} \Z$ ($n \in \N$, $w \in \overline{\Q}^\times$).

Now, let $\phi \colon \Z \to Gl_n(\Q) < Gl_n(\overline{\Q})$ be any homomorphism. We want to prove that $\Q^n \rtimes_\phi \Z$ satisfies FJCw. Proposition~\ref{prop-FJCw-1} (Subgroups) tells us that it suffices to show that FJCw is true for $\overline{\Q}^n \rtimes_\phi \Z$. A decomposition of $\phi(1) \in Gl_n(\overline{\Q})$ into Jordan blocks gives
\[
\overline{\Q}^n \rtimes_\phi \Z \cong (\bigoplus \overline{\Q}^m) \rtimes_{(\oplus \, \phi_{m,w})} \Z < \bigoplus \big( \overline{\Q}^m \rtimes_{\phi_{m,w}} \Z \big).
\]
We conclude from Proposition~\ref{prop-FJCw-1} (Subgroups) and (Direct products) that $\overline{\Q}^n \rtimes_\phi \Z$ satisfies FJCw.
\end{proof}

Following the proof of \cite[Corollary~4.4]{FL03} we obtain
\begin{corollary} \label{cor-1}
Suppose that FJCw is true for the groups $G_w$ ($w \in \overline{\Q}^\times$).
Then FJCw holds for every semi-direct product $A \rtimes \Z$ with $A$ torsion-free abelian.
\end{corollary}
\begin{proof}
Since $A$ is a torsion-free abelian group, $A \rtimes \Z$ is a subgroup of a semi-direct product $B \rtimes \Z$ with $B := A \otimes_\Z \Q$. The conjugation action of $\Z$ on $B$ makes $B$ into a $\Q\Z$-module. By Proposition~\ref{prop-FJCw-1} it suffices to show that $C \rtimes \Z$ satisfies FJCw for every finitely generated $\Q\Z$-submodule $C < B$. Since $\Q\Z$ is a principal ideal domain, we can decompose $C$ into cyclic $\Q\Z$-submodules: $C \cong \bigoplus_{k=1}^n C_k$. Note that $C \rtimes \Z$ is a subgroup of $\bigoplus_{k=1}^n (C_k \rtimes \Z)$. Hence, it suffices to show that $C_k \rtimes \Z$ satisfies FJCw (see Proposition~\ref{prop-FJCw-1}). There are two types of cyclic $\Q\Z$-modules: free and finite $\Q$-dimensional. If $C_k$ is a free $\Q\Z$-module then we conclude $C_k \rtimes \Z \cong \Q \wr \Z$. Since $\Q$ is a directed colimit of subgroups isomorphic to $\Z$, we obtain $\Q \wr \Z = \colim \Z \wr \Z$. This group satisfies FJCw because of Lemma~\ref{lem-2} and Proposition~\ref{prop-FJCw-1} (Colimits). If $C_k$ has finite $\Q$-dimension then $C_k \rtimes \Z$ satisfies FJCw because of Lemma~\ref{lem-3}.
\end{proof}

\begin{proof}[Proof of Proposition~\ref{prop-sdp}]
Recall that an $n$-step solvable group is a group whose $n$-th term of the derived series is trivial.
If $G$ is $1$-step solvable (i.e., abelian) then $G$ satisfies FJCw (see Example~\ref{ex-FJCw-2}). Now suppose that $G$ is $2$-step solvable. Then $[G,G]$ is abelian. Let $T < [G,G]$  be the torsion subgroup. Since $T$ is characteristic in $[G,G]$, it is normal in $G$. We obtain a short exact sequence
\[
1 \to [G,G]/T \to G/T \to G/[G,G] \to 1.
\]
Proposition~\ref{prop-FJCw-2} (Extensions) and Corollary~\ref{cor-1} imply that FJCw holds for $G/T$.
Proposition~\ref{prop-FJCw-2} applied to the short exact sequence $1 \to T \to G \to G/T \to 1$ and Lemma~\ref{lem-1} imply that $G$ satisfies FJCw.

Now, let $G$ be an $n$-step solvable group with $n \geq 3$. We will prove by induction on $n$ that $G$ satisfies FJCw.
We consider the short exact sequence
\[
1 \to G^{(2)} \to G \to G/G^{(2)} \to 1.
\]
Note that $G^{(2)}$ is $(n-2)$-step solvable and $G/G^{(2)}$ is $2$-step solvable. If $C<G/G^{(2)}$ is an infinite cyclic subgroup then the preimage is an $(n-1)$-step solvable subgroup of $G$. We apply Proposition~\ref{prop-FJCw-2} (Extensions) and conclude that $G$ satisfies FJCw.
\end{proof}

\section{A combination of the Farrell-Hsiang method and transfer reducibility} \label{sec-FHJ}

In this section we introduce and study Farrell-Hsiang-Jones groups. They arise from a combination of the Farrell-Hsiang method and transfer reducibility.\footnote{Some Farrell-Hsiang groups are transfer reducible, see \cite{Win14}.}

\begin{definition}[FHJ group] \label{def-FHJ}
Let $G$ be a finitely generated group and let $\mathcal{F}$ be a family of subgroups. Let $S \subseteq G$ be a finite symmetric subset which generates $G$ and contains the trivial element $e \in G$.
We call $G$ a \emph{Farrell-Hsiang-Jones group (''FHJ group'' for short) with respect to the family $\mathcal{F}$} if there exist a natural number $N$ and surjective homomorphisms $\alpha_n \colon G \to F_n$ ($n \in \N$) onto finite groups $F_n$ such that the following condition is satisfied.
For any hyper-elementary subgroup $H$ of $F_n$ there exist
\begin{itemize}
\item a compact, contractible, controlled $N$-dominated metric space $X_{n,H}$,
\item a strong homotopy $G$-action $\Psi_{n,H}$ on $X_{n,H}$,
\item a positive real number $\Lambda_{n,H}$,
\item a simplicial complex $E_{n,H}$ of dimension at most $N$ with a simplicial $\alpha_n^{-1}(H)$-action whose stabilizers belong to $\mathcal{F}$,
\item and an $\alpha_n^{-1}(H)$-equivariant map $f_{n,H} \colon G \times X_{n,H} \to E_{n,H}$
\end{itemize}
such that
\[
n \cdot d_{E_{n,H}}^1\big(f_{n,H}(g,x),f_{n,H}(h,y)\big) \leq d_{\Psi_{n,H},S^n,n,\Lambda_{n,H}}\big((g,x),(h,y)\big)
\]
for all $(g,x),(h,y) \in G \times X_{n,H}$ with $h^{-1}g \in S^n$.
\end{definition}

\begin{examples}
\begin{enumerate}
\item Every Farrell-Hsiang group is a FHJ group.
    (Choose as $X_{n,H}$ the space consisting of one point.)
\item If a group $G$ is almost strongly transfer reducible over $\mathcal{F}$ then $G$ is a FHJ group with respect to $\mathcal{F}$. This follows from \cite[Proposition~3.6]{Weg12}. (Choose as $F_n$ the trivial group.)
\end{enumerate}
\end{examples}

\begin{lemma} \label{lem-FHJ}
Let $G$ be a FHJ group with respect to the family $\mathcal{F}_G := \{ H < G \mid H \mbox{ satisfies } FJCw \}$.
Let $F$ be a finite group. Then the wreath product $G \wr F$ is a FHJ group with respect to the family $\mathcal{F}_{G \wr F} := \{ H < G \wr F \mid H \mbox{ satisfies } FJCw \}$.
\end{lemma}
\begin{proof}
We fix a finite symmetric generating set $S_G$ for $G$ and define the following finite symmetric generating sets: $S_{G^F} := \{(s_f)_{f \in F} \in G^F \mid s_f \in S_G\}$, $S_F := F$, $S_{G \wr F} := \{ (x,f) \mid x \in S_{G^F}, f \in F \}$.
We denote the respective word length functions by $l_G$, $l_{G^F}$, $l_{G \wr F}$ and the respective word metrics by $d_G$, $d_{G^F}$, $d_{G \wr F}$.

Since $G$ is a FHJ group there exist a natural number $N_G$ and surjective homomorphisms $\alpha_{G,n} \colon G \to F_{G,n}$ ($n \in \N$) onto finite groups $F_{G,n}$ such that the following condition is satisfied.
For any hyper-elementary subgroup $H$ of $F_{G,n}$ there exist
\begin{itemize}
\item a compact, contractible, controlled $N_G$-dominated metric space $X_{G,n,H}$,
\item a strong homotopy $G$-action $\Psi_{G,n,H}$ on $X_{G,n,H}$,
\item a positive real number $\Lambda_{G,n,H}$,
\item a simplicial complex $E_{G,n,H}$ of dimension at most $N_G$ with a simplicial $\alpha_{G,n}^{-1}(H)$-action whose stabilizers belong to $\mathcal{F}_G$,
\item and an $\alpha_{G,n}^{-1}(H)$-equivariant map $f_{G,n,H} \colon G \times X_{G,n,H} \to E_{G,n,H}$
\end{itemize}
such that
\[
n \cdot d_{E_{G,n,H}}^1\big(f_{G,n,H}(g,x),f_{G,n,H}(h,y)\big) \leq d_{\Psi_{G,n,H},S_G^n,n,\Lambda_{G,n,H}}\big((g,x),(h,y)\big)
\]
for all $(g,x),(h,y) \in G \times X_{G,n,H}$ with $d_G(g,h) \leq n$.

We set
\[
N_{G \wr F} := \max\big\{ |F|^2 \cdot N_G, |F| \cdot (N_G + 1)^{|F|^2} - 1 \big\}
\]
and define $\alpha_{G \wr F,n} \colon G \wr F \to F_{G,n|F|^2} \wr F =: F_{G \wr F,n}$ to be the surjective homomorphism induced by $\alpha_{G,n|F|^2}$, i.e., $\alpha_{G \wr F,n}((g_f)_{f \in F}, f') := ((\alpha_{G,n|F|^2}(g_f))_{f \in F}, f')$.

Let $H < F_{G \wr F,n}$ be a hyper-elementary subgroup. We have to construct
\begin{itemize}
\item a compact, contractible, controlled $N_{G \wr F}$-dominated metric space $X_{G \wr F,n,H}$,
\item a strong homotopy $G \wr F$-action $\Psi_{G \wr F,n,H}$ on $X_{G \wr F,n,H}$,
\item a positive real number $\Lambda_{G \wr F,n,H}$,
\item a simplicial complex $E_{G \wr F,n,H}$ of dimension at most $N_{G \wr F}$ with a simplicial $\alpha_{G \wr F,n}^{-1}(H)$-action whose stabilizers belong to $\mathcal{F}_{G \wr F}$,
\item and an $\alpha_{G \wr F,n}^{-1}(H)$-equivariant map $f_{G \wr F,n,H} \colon G \wr F \times X_{G \wr F,n,H} \to E_{G \wr F,n,H}$
\end{itemize}
such that
\[
n \cdot d_{E_{G \wr F,n,H}}^1\big(f_{G \wr F,n,H}(g,x),f_{G \wr F,n,H}(h,y)\big) \leq d_{\Psi_{G \wr F,n,H},S_{G \wr F}^n,n,\Lambda_{G \wr F,n,H}}\big((g,x),(h,y)\big)
\]
for all $(g,x),(h,y) \in G \wr F \times X_{G \wr F,n,H}$ with $d_{G \wr F}(g,h) \leq n$.

Since the class of hyper-elementary groups is closed under taking subgroups and quotients, the groups $H_f := \pr_f(H \cap F_{G,n|F|^2}^F) < G$ are hyper-elementary. We define
\[
\Lambda_{G \wr F,n,H} := \max\big\{ \Lambda_{G,n|F|^2,H_f} \, \big| \, f \in F \big\}.
\]
For the sequel we use the abbreviations
$X_{H_f} := X_{G,n|F|^2,H_f}$, $\Psi_{H_f} := \Psi_{G,n|F|^2,H_f}$, $\Lambda_{H_f} := \Lambda_{G,n|F|^2,H_f}$.

The strong homotopy $G$-actions $\Psi_{H_f}$ on $X_{H_f}$ induce strong homotopy $G \wr F$-actions $\widetilde{\Psi}_{H_f}$ on $X_{H_f}^F$:
\begin{align*}
& \widetilde{\Psi}_{H_f}\big(g_l,t_l,g_{l-1},\ldots,t_1,g_0,x\big)_f := \\
& \Psi_{H_f}\big((a_l)_f,t_l,(a_{l-1})_{f f_l},\ldots,t_1,(a_0)_{f f_l \ldots f_1},x_{f f_l \ldots f_1 f_0}\big)
\end{align*}
where $a_i \in G^F$ and $f_i \in F$ ($i = 0,1,\ldots,l$) are determined by $g_i = a_i f_i$.
We equip $X_{H_f}^F$ with the metric
\[
d_{X_{H_f}^F}\big(x,y\big) := \frac{1}{|F|} \sum_{f \in F} d_{X_{H_f}}(x_f,y_f).
\]
The metric space $X_{H_f}^F$ is compact, contractible, and controlled $|F| \cdot N_G$-dominated.
The group $F$ acts on $X_{H_f}^F$ by $f' \cdot (x_{\overline{f}})_{\overline{f} \in F} := (x_{\overline{f}f'})_{\overline{f} \in F}$.
Since the pseudo-metric
\begin{align*}
& d\big((g,x),(h,y)\big) := \\
& \frac{1}{|F|} \sum_{\overline{f} \in F} d_{\Psi_{H_f},S_G^n,n,\Lambda_{H_f}}\big((g \cdot \pr_F(g)^{-1})_{\overline{f}},(\pr_F(g)x)_{\overline{f}}),((h \cdot \pr_F(h)^{-1})_{\overline{f}},(\pr_F(h)y)_{\overline{f}})\big)
\end{align*}
on $G \wr F \times X_{H_f}^F$ satisfies
\begin{itemize}
\item $d((g,x),(g,y)) \leq \Lambda_{H_f} \cdot d_{X_H}(x,y)$ for all $g \in G \wr F$, $x,y \in X_{H_f}^F$ and
\item $d((g,x),(h,y)) \leq 1$ for all $(h,y) \in S^1_{\widetilde{\Psi}_{H_f},S_{G \wr F}^n,n}(g,x)$,
\end{itemize}
we have
\[
\frac{1}{|F|} \sum_{\overline{f} \in F} d_{\Psi_{H_f},S_G^n,n,\Lambda_{H_f}}\big(g_{\overline{f}},x_{\overline{f}}),(h_{\overline{f}},y_{\overline{f}})\big) \leq d_{\widetilde{\Psi}_{H_f},S_{G \wr F}^n,n,\Lambda_{H_f}}\big((g,x),(h,y)\big)
\]
for all $(g,x), (h,y) \in G^F \times X_{H_f}^F$.

The strong homotopy $G \wr F$-actions $\widetilde{\Psi}_{H_f}$ on $X_{H_f}^F$ give rise to a strong homotopy $G \wr F$-action $\Psi_{G \wr F,n,H}$ on $X_{G \wr F,n,H} := \prod_{f \in F} X_{H_f}^F$ which is given by
\[
\Psi_{G \wr F,n,H}\big(g_l,t_l,g_{l-1},\ldots,t_1,g_0,x\big) := \Big(\widetilde{\Psi}_{H_f}\big(g_l,t_l,g_{l-1},\ldots,t_1,g_0,x_f\big)\Big)_{f \in F}.
\]
The group $F$ acts diagonally on $X_{G \wr F,n,H} := \prod_{f \in F} X_{H_f}^F$, i.e., $f' \cdot (x_f)_{f \in F} := (f' \cdot x_f)_{f \in F}$ for $f' \in F$ and $x_f \in X_{H_f}^F$ ($f \in F$). We equip $X_{G \wr F,n,H}$ with the metric
\[
d_{X_{G \wr F,n,H}}\big(x,y\big) := \frac{1}{|F|} \sum_{f \in F} d_{X_{H_f}^F}(x_f,y_f).
\]
Note that $X_{G \wr F,n,H}$ is a compact, contractible, controlled $|F|^2 \cdot N_G$-dominated metric space.
We abbreviate $X_H := X_{G \wr F,n,H}$, $\Psi_H := \Psi_{G \wr F,n,H}$, $\Lambda_H := \Lambda_{G \wr F,n,H}$.

Since the metric
\[
\big((g,x),(h,y)\big) \mapsto \frac{1}{|F|} \sum_{f \in F} d_{\widetilde{\Psi}_{H_f},S_{G \wr F}^n,n,\Lambda_{H_f}}\big((g,x_f),(h,y_f)\big)
\]
on $G \wr F \times X_{H}$ satisfies
\begin{itemize}
\item $\frac{1}{|F|} \sum_{f \in F} d_{\widetilde{\Psi}_{H_f},S_{G \wr F}^n,n,\Lambda_{H_f}}((g,x),(g,y)) \leq \Lambda_H \cdot d_{X_{H}}(x,y)$ for all $g \in G \wr F$, $x,y \in X_H$ and
\item $\frac{1}{|F|} \sum_{f \in F} d_{\widetilde{\Psi}_{H_f},S_{G \wr F}^n,n,\Lambda_{H_f}}((g,x),(h,y)) \leq 1$ for all $(h,y) \in S^1_{\Psi_H,S_{G \wr F}^n,n}(g,x)$,
\end{itemize}
we have
\[
\frac{1}{|F|} \sum_{f \in F} d_{\widetilde{\Psi}_{H_f},S_{G \wr F}^n,n,\Lambda_{H_f}}\big((g,x),(h,y)\big) \leq d_{\Psi_H,S_{G \wr F}^n,n,\Lambda_H}\big((g,x),(h,y)\big)
\]
for all $(g,x), (h,y) \in G \wr F \times X_H$.

We define a simplicial complex $E'_{G \wr F,n,H}$ as follows. The vertices of $E'_{G \wr F,n,H}$ are tuples $(v_{f,f'})_{f,f' \in F}$, where $v_{f,f'}$ is a vertex of $E_{G,n|F|^2,H_f}$. Distinct vertices $v^{(1)}, \ldots, v^{(m)}$ span an $m$-simplex if and only if for all $f,f' \in F$ the vertices $v^{(1)}_{f,f'}, \ldots, v^{(m)}_{f,f'}$ are contained in a common simplex of $E_{G,n|F|^2,H_f}$. The simplicial complex $E'_{G \wr F,n,H}$ is of dimension at most $(N_G + 1)^{|F|^2} - 1$.

For every $f \in \pr_F(H)$ we fix an element $h_f \in \alpha_{G \wr F,n}^{-1}(H)$ with $\pr_F(h_f) = f$. A simplicial $\alpha_{G \wr F,n}^{-1}(H)$-action on $\coprod_{f \in F} E'_{G \wr F,n,H} = F \times E'_{G \wr F,n,H}$ is defined as follows. Let $g \in \alpha_{G \wr F,n}^{-1}(H)$, $\overline{f} \in F$, and let $(v_{f,f'})_{f,f' \in F}$ be a vertex in $E'_{G \wr F,H}$. We define
\[
g \big(\overline{f},(v_{f,f'})_{f,f' \in F}\big) := \big(\pr_F(g)\overline{f},(\pr_f(g')v_{f,f'})_{f,f' \in F}\big)
\]
with $g' := h_{\pr_F(g)\overline{f}}^{-1} \cdot g \cdot h_{\overline{f}} \in \alpha_{G \wr F,n}^{-1}(H) \cap G^F = \prod_{f \in F} \alpha_{G,n|F|^2}^{-1}(H_f)$.

We fix a right coset transversal $T \subset F$ for $G^F \rtimes \pr_F(H) < G \wr F$. Note that every element in $G \wr F$ can be uniquely written as $h_{\overline{f}} \cdot g \cdot t$ with $\overline{f} \in \pr_F(H)$, $g = (g_f)_{f \in F} \in G^F$, and $t \in T$. We define a map $f'_{G \wr F,n,H} \colon G \wr F \times X_H \to F \times E'_{G \wr F,n,H}$ by
\[
f'_{G \wr F,n,H}(h_{\overline{f}} \cdot g \cdot t,x) = \Big(\overline{f},i\big((f_{G,n|F|^2,H_f}(g_f,(t^{-1}x)_f))_{f,f' \in F}\big)\Big),
\]
where $i \colon \prod_{f,f' \in F} E_{G,n|F|^2,H_f} \to E'_{G \wr F,n,H}$ is given by
\begin{align*}
& \Big(\displaystyle\sum_{v_f \text{ vertex in } E_{G,n|F|^2,H_f}} t_{v_f}(f') \cdot v_f\Big)_{f,f' \in F} \mapsto \\
& \sum_{(v_{f,f'})_{f,f' \in F} \text{ vertex in } E'_{G \wr F,n,H}} \prod_{f,f' \in F} t_{v_f}(f') \cdot (v_{f,f'})_{f,f' \in F}
\end{align*}
($t_{v_f}(f') \geq 0$, $\sum_{v_f} t_{v_f}(f') = 1$).
The map $f'_{G \wr F,n,H}$ is $\alpha_{G \wr F,n}^{-1}(H)$-equivariant.

We would like to define a map $f_{G \wr F,n,H} \colon G \wr F \times X_H \to F \times E'_{G \wr F,n,H}$ by $f_{G \wr F,n,H}(g,x) := \sum_{f \in F} \frac{1}{|F|} f'_{G \wr F,n,H} (gf^{-1},fx)$. But in general, this linear combination will not lie in a simplex of $F \times E'_{G \wr F,n,H}$. Hence, we define the simplicial complex $E_{G \wr F,n,H}$ to be the smallest simplicial complex with the same vertices as $F \times E'_{G \wr F,n,H}$ such that every linear combination $\sum_{f \in F} \frac{1}{|F|} f'_{G \wr F,n,H} (gf^{-1},fx)$ ($g \in G \wr F$, $x \in X_H$) lies in $E_{G \wr F,n,H}$. Then the map
\[
f_{G \wr F,n,H} \colon G \wr F \times X_H \to F \times E_{G \wr F,n,H}, (g,x) \mapsto \sum_{f \in F} \frac{1}{|F|} f'_{G \wr F,n,H} (gf^{-1},fx)
\]
is well defined. The simplicial complex $E_{G \wr F,n,H}$ is of dimension at most $|F| \cdot (N_G + 1)^{|F|^2} -  1$.
For the sequel we use the abbreviations
$E_H := E_{G \wr F,n,H}$, $f_H := f_{G \wr F,n,H}$, $E_{H_f} := E_{G,n|F|^2,H_f}$, $f_{H_f} := f_{G,n|F|^2,H_f}$.

It remains to show that
\[
n \cdot d_{E_H}^1\big(f_H(g,x),f_H(h,y)\big) \leq d_{\Psi_H,S_{G \wr F}^n,n,\Lambda_H}\big((g,x),(h,y)\big)
\]
for all $(g,x),(h,y) \in G \wr F \times X_H$ with $d_{G \wr F}(g,h) \leq n$.
Let $g',h' \in G^F$ and $f_g,f_h \in F$ with $g = g'f_g$ and $h = h'f_h$.
Then $f_H(g,x) = f_H(g',f_g^{-1}x)$ and $f_H(h,y) = f_H(h',f_h^{-1}y)$. For every $\overline{f} \in F$ there exist $a(\overline{f}) \in \pr_F(H)$, $b(\overline{f}), c(\overline{f}) \in G^F$, and $d(\overline{f}) \in T \subset F$ such that $g' \overline{f} = h_{a(\overline{f})} b(\overline{f}) d(\overline{f})$ and $h' \overline{f} = h_{a(\overline{f})} c(\overline{f}) d(\overline{f})$.
We conclude
\begin{eqnarray*}
f_H(g',f_g^{-1}x) & = & \frac{1}{|F|} \sum_{\overline{f} \in F} \Big(a(\overline{f}),i\big((f_{H_f}(b(\overline{f})_f,d(\overline{f})^{-1}f_g^{-1}x)_f)_{f,f' \in F}\big)\Big), \\
f_H(h',f_h^{-1}y) & = & \frac{1}{|F|} \sum_{\overline{f} \in F} \Big(a(\overline{f}),i\big((f_{H_f}(c(\overline{f})_f,d(\overline{f})^{-1}f_h^{-1}y)_f)_{f,f' \in F}\big)\Big).
\end{eqnarray*}
Hence,
\begin{align*}
& d^1_{E_H}(f_H(g,x),f_H(h,y)) \\
& \leq \frac{1}{|F|} \sum_{\overline{f} \in F} d^1_{E'_{G \wr F,n,H}}\Big(i\big((f_{H_f}(b(\overline{f})_f,(d(\overline{f})^{-1}f_g^{-1}x)_f))_{f,f' \in F}\big),i\big((f_{H_f}(c(\overline{f})_f,(d(\overline{f})^{-1}f_h^{-1}y)_f))_{f,f' \in F}\big)\Big) \\
& \leq \frac{1}{|F|} \sum_{\overline{f}, f,f' \in F} d^1_{E_{H_f}}\big(f_{H_f}(b(\overline{f})_f,(d(\overline{f})^{-1}f_g^{-1}x)_f),f_{H_f}(c(\overline{f})_f,(d(\overline{f})^{-1}f_h^{-1}y)_f)\big) \\
& = \sum_{\overline{f}, f \in F} d^1_{E_{H_f}}\big(f_{H_f}(b(\overline{f})_f,(d(\overline{f})^{-1}f_g^{-1}x)_f),f_{H_f}(c(\overline{f})_f,(d(\overline{f})^{-1}f_h^{-1}y)_f)\big).
\end{align*}

The construction of the generating sets $S_{G \wr F}$ and $S_{G^F}$ -- in particular, the invariance under conjugation with elements in $F$ -- implies $d_{G^F}(g',h') \leq d_{G \wr F}(g,h) \leq n$.
Since
\begin{align*}
& d_G\big(b(\overline{f})_f,c(\overline{f})_f\big) = l_G\big(c(\overline{f})_f^{-1} b(\overline{f})_f\big) \leq l_{G^F}\big(c(\overline{f})^{-1} b(\overline{f})\big) = \\
& l_{G^F}\big(d(\overline{f}) \overline{f}^{-1} h'^{-1} g' \overline{f} d(\overline{f})^{-1}\big) = l_{G^F}(h'^{-1} g') \leq n,
\end{align*}
we conclude
\[
d^1_{E_{H_f}}\big(f_{H_f}(b(\overline{f})_f,(d(\overline{f})^{-1}f_g^{-1}x)_f),f_{H_f}(c(\overline{f})_f,(d(\overline{f})^{-1}f_h^{-1}y)_f)\big) < \frac{1}{n|F|^2}.
\]
Therefore,
\begin{align*}
& d^1_{E_H}(f_H(g,x),f_H(h,y)) = d^1_{E_H}(f_H(g',x),f_H(h',y)) \\
& \leq \sum_{\overline{f}, f \in F} d^1_{E_{H_f}}\big(f_{H_f}(b(\overline{f})_f,(d(\overline{f})^{-1}f_g^{-1}x)_f),f_{H_f}(c(\overline{f})_f,(d(\overline{f})^{-1}f_h^{-1}y)_f)\big) \\
& \leq \frac{1}{n|F|^2} \sum_{\overline{f}, f \in F} d_{\Psi_{H_f},S_G^n,n,\Lambda_{H_f}}\big((b(\overline{f})_f,(d(\overline{f})^{-1}f_g^{-1}x)_f),(c(\overline{f})_f,(d(\overline{f})^{-1}f_h^{-1}y)_f)\big) \\
& \leq \frac{1}{n|F|} \sum_{\overline{f} \in F} d_{\widetilde{\Psi}_{H_f},S_{G \wr F}^n,n,\Lambda_{H_f}}\big((b(\overline{f}),d(\overline{f})^{-1}f_g^{-1}x),(c(\overline{f}),d(\overline{f})^{-1}f_h^{-1}y)\big) \\
& \leq \frac{1}{n|F|} \sum_{\overline{f} \in F} d_{\Psi_H,S_{G \wr F}^n,n,\Lambda_H}\big((b(\overline{f}),d(\overline{f})^{-1}f_g^{-1}x),(c(\overline{f}),d(\overline{f})^{-1}f_h^{-1}y)\big) \\
& = \frac{1}{n|F|} \sum_{\overline{f} \in F} d_{\Psi_H,S_{G \wr F}^n,n,\Lambda_H}\big((b(\overline{f})d(\overline{f})f_g,x),(c(\overline{f})d(\overline{f})f_h,y)\big) \\
& = \frac{1}{n|F|} \sum_{\overline{f} \in F} d_{\Psi_H,S_{G \wr F}^n,n,\Lambda_H}\big((h_{a(\overline{f})}b(\overline{f})d(\overline{f})f_g,x),(h_{a(\overline{f})}c(\overline{f})d(\overline{f})f_h,y)\big) \\
& = \frac{1}{n|F|} \sum_{\overline{f} \in F} d_{\Psi_H,S_{G \wr F}^n,n,\Lambda_H}\big((g,x),(h,y)\big) \\
& = \frac{1}{n} d_{\Psi_H,S_{G \wr F}^n,n,\Lambda_H}\big((g,x),(h,y)\big)
\end{align*}
This finishes the proof of Lemma~\ref{lem-FHJ}.
\end{proof}

\begin{proposition} \label{prop-FHJ}
Let $G$ be a FHJ group with respect to a family $\mathcal{F}$ of subgroups.
Suppose that $\mathcal{F}$ is virtually closed, i.e., $K < H < G$ with $K \in \mathcal{F}$ and $[H:K] < \infty$ implies $H \in \mathcal{F}$.
Then $G$ satisfies FJC with respect to $\mathcal{F}$.
\end{proposition}
The proof of Proposition~\ref{prop-FHJ} needs some preparation. We will give a proof in subsection~\ref{subsec-FHJ-K} and \ref{subsec-FHJ-L}.

\begin{remark}
If the simplicial $\alpha_n^{-1}(H)$-action is cell preserving then we can drop the assumption in Proposition \ref{prop-FHJ} that $\mathcal{F}$ is virtually closed.
\end{remark}

Lemma~\ref{lem-FHJ} and Proposition~\ref{prop-FHJ} imply
\begin{corollary} \label{cor-FHJ}
Let $G$ be a FHJ group with respect to the family $\mathcal{F}_G := \{ H < G \mid H \mbox{ satisfies } FJCw \}$. Then $G$ satisfies FJCw.
\end{corollary}

\subsection{The obstruction category}

In this subsection we recall the definition of the obstruction category. In the sequel $\mathcal A$ denotes a small additive category (with a strictly associative direct sum) which is provided with a strict right action of a group $G$.

\begin{definition}[obstruction category]
Let $X$ be a $G$-space and let $(Y,d_Y)$ be a metric space with an isometric $G$-action. We consider the $G$-space $G \times X \times Y \times [1,\infty)$ with the $G$-action given by $h(g,x,y,t) := (hg,hx,hy,t)$. We define the \emph{obstruction category $\mathcal{O}^G(X,(Y,d_Y);\mathcal{A})$} as follows.\\
An object in $\mathcal{O}^G(X,(Y,d_Y);\mathcal{A})$ is a collection $A = (A_{g,x,y,t})_{(g,x,y,t) \in G \times X \times Y \times [1,\infty)}$ of objects in $\mathcal{A}$ with the following properties:
\begin{itemize}
 \item $A$ is locally finite, i.e., for every $z_0 \in G \times X \times Y \times [1,\infty)$ there exists an open neighborhood $U$ such that the set $\{z \in G \times X \times Y \times [1,\infty) \, | \, A_z \neq 0\} \cap U$ is finite.
 \item There is a compact subset $K \subseteq G \times X \times Y$ such that $A_{g,x,y,t} = 0$ whenever $(g,x,y) \notin G \cdot K$.
 \item We have $A_z \cdot g = A_{g^{-1}z}$ for all $z \in G \times X \times Y \times [1,\infty)$ and $g \in G$.
\end{itemize}
A morphism $\phi \colon B \to A$ is a collection of morphisms $\phi_{z,z'} \colon B_{z'} \to A_z$ in $\mathcal{A}$ ($z,z' \in G \times X \times Y \times [1,\infty)$) with the following properties:
\begin{itemize}
 \item The sets $\{ z \in G \times X \times Y \times [1,\infty) \, | \, \phi_{z,z'} \neq 0\}$ and $\{ z \in G \times X \times Y \times [1,\infty)  \, | \, \phi_{z',z} \neq 0\}$ are finite for all $z' \in G \times X \times Y \times [1,\infty)$.
 \item There are $R,T > 0$ and a finite subset $F \subseteq G$ such that $\phi_{(g,x,y,t),(g',x',y',t')} = 0$ whenever $g^{-1}g' \notin F$ or $d_Y(y,y') > R$ or $|t-t'| > T$.
 \item The set
     \[
      \big\{ \big((x,t),(x',t')\big) \in (X \times [1,\infty))^2 \, \big| \, \exists \, g,g' \in G, y,y' \in Y: \phi_{(g,x,y,t),(g',x',y',t')} \neq 0 \big\}
     \]
     lies in the equivariant continuous control condition $\mathcal{E}^X_{Gcc}$ defined in \cite[section 3.2]{BLR08b}.
 \item We have $\phi_{z,z'} \cdot g = \phi_{g^{-1}z,g^{-1}z'}$ for all $z,z' \in G \times X \times Y \times [1,\infty)$ and $g \in G$.
\end{itemize}
Composition is given by matrix multiplication, i.e.,
\[
 (\psi \circ \phi)_{z,z''} := \sum_{z' \in G \times X \times Y \times [1,\infty)} \psi_{z,z'} \circ \psi_{z',z''}.
\]
The obstruction category $\mathcal{O}^G(X,(Y,d_Y);\mathcal{A})$ inherits the structure of an additive category from $\mathcal{A}$.
\end{definition}
We use the same notation as in \cite[subsection~4.4]{BL12a} and \cite[section~4]{Weg12} which slightly differs from the notation used in \cite{BLR08b} (see \cite[Remark 4.10]{BL12a}).

The construction is functorial in $Y$: Let $f \colon Y \to Y'$ be a $G$-equivariant map with the property that for every $r > 0$ there exists $R > 0$ such that $d_{Y'}(f(y_1),f(y_2)) < R$ whenever $d_Y(y_1,y_2) < r$. Then the map $f$ induces a functor $f_* \colon \mathcal{O}^G(X,Y;\mathcal{A}) \to \mathcal{O}^G(X,Y';\mathcal{A})$ with $f_*(A)_{g,x,y',t} := \bigoplus_{y \in f^{-1}(\{y'\})} \, A_{g,x,y,t}$.

If $\mathcal{A}$ is an additive category with involution $I_\mathcal{A}$ (i.e., $I_\mathcal{A} \colon \mathcal{A} \to \mathcal{A}$ is a contravariant functor with $I_\mathcal{A}^2 = \id_\mathcal{A}$) then we obtain an involution $I_\mathcal{O}$ on $\mathcal{O}^G(X,(Y,d_Y);\mathcal{A})$ by $I_\mathcal{O}(A)_{g,x,y,t} := I_\mathcal{A}(A_{g,x,y,t})$ and $I_\mathcal{O}(\phi)_{z,z'} := I_\mathcal{A}(\phi_{z',z})$.

We are mostly interested in $\mathcal{O}^G(E_\mathcal{F}G,\pt;\mathcal{A})$ because of the following proposition which is proven in \cite[Theorem 5.2]{BL12a}.
\begin{proposition} \label{prop-assembly}
\begin{enumerate}
\item Let $G$ be a group and let $\mathcal{F}$ be a family of subgroups. Suppose that for every $m_0 \in \Z$ there exists $m \geq m_0$ such that
    \[
    K_m\big(\mathcal{O}^G(E_\mathcal{F}G,\pt;\mathcal{A})\big) = 0
    \]
    for all additive $G$-categories $\mathcal{A}$.
    Then the assembly map~(\ref{assembly-K}) is an isomorphism for all $m \in \Z$ and all additive $G$-categories $\mathcal{A}$.
\item Let $G$ be a group and let $\mathcal{F}$ be a family of subgroups. Denote by $\mathcal{F}_2$ the family of subgroups which are contained in $\mathcal{F}$ or contain a member of $\mathcal{F}$ as subgroup of index two. Suppose that there exists $m \in \Z$ such that
    \[
    L_m^{\langle -\infty \rangle}\big(\mathcal{O}^G(E_{\mathcal{F}_2}G,\pt;\mathcal{A})\big) = 0
    \]
    for all additive $G$-categories $\mathcal{A}$ with involution.
    Then the assembly map~(\ref{assembly-L}) is an isomorphism for all $m \in \Z$ and all additive $G$-categories $\mathcal{A}$.
\end{enumerate}
\end{proposition}

The reason why we study the category $\mathcal{O}^G(E_\mathcal{F}G,(Y,d_Y);\mathcal{A})$ not only for $Y := \pt$ is that we need room for certain constructions.
Moreover, we want to consider simultaneously metric spaces $(Y_n,d_n)$ with isometric $G$-action ($n \in \N$). In analogy to \cite[subsection 3.4]{BLR08b} we define the additive subcategory
\[
\mathcal{O}^G\big(E_\mathcal{F}G,(Y_n,d_n)_{n \in \N};\mathcal{A}\big) \subseteq \prod_{n \in \N} \mathcal{O}^G\big(E_\mathcal{F}G,(Y_n,d_n);\mathcal{A}\big)
\]
by requiring additional conditions on the morphisms. A morphism $\phi = (\phi(n))_{n \in \N}$ is allowed if there are $R > 0$ and a finite subset $F \subseteq G$ (not depending on $n$) such that $\phi(n)_{(g,x,y,t),(g',x',y',t')} = 0$ whenever $g^{-1}g' \notin F$ or $d_n(y,y') > R$.
Note that a sequence of $G$-equivariant maps $(f_n \colon Y_n \to Y'_n)_{n \in \N}$ induces a functor $(f_n)_* \colon \mathcal{O}^G(E_\mathcal{F}G,(Y_n,d_n)_{n \in \N};\mathcal{A}) \to \mathcal{O}^G(E_\mathcal{F}G,(Y'_n,d'_n)_{n \in \N};\mathcal{A})$ if for every $r > 0$ there exists $R > 0$ such that $d'_n(f_n(y_1),f_n(y_2)) < R$ whenever $d_n(y_1,y_2) < r$.

The inclusion
\[
\bigoplus_{n \in \N} \mathcal{O}^G\big(E_\mathcal{F}G,(Y_n,d_n);\mathcal{A}\big) \to \mathcal{O}^G\big(E_\mathcal{F}G,(Y_n,d_n)_{n \in \N};\mathcal{A}\big)
\]
is a Karoubi filtration and we denote the quotient by $\mathcal{O}^G(E_\mathcal{F}G,(Y_n,d_n)_{n \in \N};\mathcal{A})^{> \oplus}$.

\subsection{Farrell-Hsiang-Jones groups and the $K$-theoretic Farrell-Jones Conjecture} \label{subsec-FHJ-K}

This subsection is dedicated to the proof of
\begin{proposition} \label{prop-FHJ-K}
Let $G$ be a FHJ group with respect to a family $\mathcal{F}$ of subgroups.
Suppose that $\mathcal{F}$ is virtually closed, i.e., $K < H < G$ with $K \in \mathcal{F}$ and $[H:K] < \infty$ implies $H \in \mathcal{F}$.
Then $G$ satisfies the $K$-theoretic Farrell-Jones Conjecture with respect to $\mathcal{F}$.
\end{proposition}
\begin{proof}
We fix a finite symmetric generating subset $S \subseteq G$ which contains the trivial element $e \in G$. We denote by $d_G$ the word metric with respect to $S \setminus \{e\}$.
Since $G$ is a FHJ group with respect to $\mathcal{F}$, there exist a natural number $N$ and surjective homomorphisms $\alpha_n \colon G \to F_n$ ($n \in \N$) with $F_n$ a finite group such that the following condition is satisfied.
For any hyper-elementary subgroup $H$ of $F_n$ there exist
\begin{itemize}
\item a compact, contractible, controlled $N$-dominated metric space $X_{n,H}$,
\item a strong homotopy $G$-action $\Psi_{n,H}$ on $X_{n,H}$,
\item a positive real number $\Lambda_{n,H}$,
\item a simplicial complex $E_{n,H}$ of dimension at most $N$ with a simplicial $\alpha_n^{-1}(H)$-action whose stabilizers belong to $\mathcal{F}$,
\item and an $\alpha_n^{-1}(H)$-equivariant map $f_{n,H} \colon G \times X_{n,H} \to E_{n,H}$
\end{itemize}
such that
\[
n \cdot d_{E_{n,H}}^1\big(f_{n,H}(g,x),f_{n,H}(h,y)\big) \leq d_{\Psi_{n,H},S^n,n,\Lambda_{n,H}}\big((g,x),(h,y)\big)
\]
for all $(g,x),(h,y) \in G \times X_n$ with $h^{-1}g \in S^n$.

We denote by $\mathcal{H}_n$ the family of hyper-elementary subgroups of $F_n$.
We set $S_n := \coprod_{H \in \mathcal{H}_n} G \times_{\alpha_n^{-1}(H)} G$, $Y_n := \coprod_{H \in \mathcal{H}_n} G \times_{\alpha_n^{-1}(H)} G \times X_{n,H}$, and $\Sigma_n := \coprod_{H \in \mathcal{H}_n} G \times_{\alpha_n^{-1}(H)} E_{n,H}$.
We will use the quasi-metrics on $S_n$ and $Y_n$ given by
\begin{eqnarray*}
d_{S_n}\big((g,h)_H,(g',h')_{H'}\big) & := & d_G(gh,g'h'),\\
d_{Y_n}\big((g,h,x)_H,(g',h',x')_{H'}\big) & := & d_{\Psi_{n,H},S^n,n,\Lambda_{n,H}}((gh,x),(g'h',x')),
\end{eqnarray*}
if $H=H'$, $gH=g'H'$, and $(g'h')^{-1}(gh) \in S^n$. Otherwise, we set
\[
d_{S_n}((g,h)_H,(g',h')_{H'}) := \infty \quad \text{and} \quad d_{Y_n}((g,h,x)_H,(g',h',x')_{H'}) := \infty.
\]
We define the map
\[
f_n := \coprod_{H \in \mathcal{H}_n} \id_G \times_{\alpha_n^{-1}(H)} f_{n,H} \colon Y_n \to \Sigma_n.
\]
Note that $n \cdot d_{\Sigma_n}^1(f_n(x),f_n(y)) \leq d_{Y_n}(x,y)$ for all $x,y \in Y_n$.

By Proposition~\ref{prop-assembly} it suffices to show
\[
K_m\big(\mathcal{O}^G(E_\mathcal{F}G,\pt;\mathcal{A})\big) = 0
\]
for all $m \geq 1$.
This is a direct consequence of the following commuting diagram
\begin{equation*} \label{diagram}
\xymatrix{
K_m\big(\mathcal{O}^G(E_\mathcal{F}G,\pt;\mathcal{A})\big) \ar@{^{(}->}[d]^{\diag_*} & \ar[dl]_{\hspace{-5mm} \pr_*} K_m\big(\mathcal{O}^G(E_\mathcal{F}G,(S_n,d_{S_n})_{n \in \N};\mathcal{A})\big)^{> \oplus} \ar[d]^{\trans_*} \\
K_m\big(\mathcal{O}^G(E_\mathcal{F}G,(\pt)_{n \in \N};\mathcal{A})^{> \oplus}\big) \ar[d]^{=} & \ar[l]_{\hspace{-5mm} \pr_*} K_m\big(\mathcal{O}^G(E_\mathcal{F}G,(Y_n,d_{Y_n})_{n \in \N};\mathcal{A})^{> \oplus}\big) \ar[d]^{(f_n)_*} \\
K_m\big(\mathcal{O}^G(E_\mathcal{F}G,(\pt)_{n \in \N};\mathcal{A})^{> \oplus}\big) & \ar[l]_{\hspace{-5mm} \pr_*} K_m\big(\mathcal{O}^G(E_\mathcal{F}G,(\Sigma_n,n \cdot d^1)_{n \in \N};\mathcal{A})^{> \oplus}\big) = 0.
}
\end{equation*}
and the property that for every $x \in K_m(\mathcal{O}^G(E_\mathcal{F}G,\pt;\mathcal{A}))$ there exists an element $y \in K_m(\mathcal{O}^G(E_\mathcal{F}G,(S_n,d_{S_n})_{n \in \N};\mathcal{A})^{> \oplus})$ with $\pr_*(y) = \diag_*(x)$.

Here are some explanations concerning the diagram above: The maps $\pr_*$ are induced by the projections $\pr \colon S_n \to \pt$ resp. $\pr \colon Y_n \to \pt$ resp. $\pr \colon \Sigma_n \to \pt$.
The map $\diag_* \colon K_m(\mathcal{O}^G(E_\mathcal{F}G,\pt;\mathcal{A})) \to K_m(\mathcal{O}^G(E_\mathcal{F}G,(\pt)_{n \in \N};\mathcal{A})^{> \oplus})$ is induced by the diagonal map.
It remains to show:
\begin{enumerate}
\item The map $\diag_* \colon K_m(\mathcal{O}^G(E_\mathcal{F}G,\pt;\mathcal{A})) \to K_m(\mathcal{O}^G(E_\mathcal{F}G,(\pt)_{n \in \N};\mathcal{A})^{> \oplus})$ is injective.
\item There exists a map
    \[
    \trans_* \colon K_m(\mathcal{O}^G(E_\mathcal{F}G,(S_n,d_{S_n})_{n \in \N};\mathcal{A})) \to K_m(\mathcal{O}^G(E_\mathcal{F}G,(Y_n,d_{Y_n})_{n \in \N};\mathcal{A})^{> \oplus})
    \]
    with $\pr_* \circ \trans_* = \pr_*$.
\item $K_m(\mathcal{O}^G(E_\mathcal{F}G,(\Sigma_n,n \cdot d^1)_{n \in \N};\mathcal{A})^{> \oplus}) = 0$.
\item For every $x \in K_m(\mathcal{O}^G(E_\mathcal{F}G,\pt;\mathcal{A}))$ there exists $y \in K_m(\mathcal{O}^G(E_\mathcal{F}G,(S_n,d_{S_n})_{n \in \N};\mathcal{A})^{> \oplus})$ with $\pr_*(y) = \diag_*(x)$.
\end{enumerate}
Let us prove the statements above:
\begin{enumerate}
\item For the injectivity of the map $\diag_*$ we refer to \cite[page~789]{Weg12}.
\item The map
    \[
    \trans_* \colon K_m(\mathcal{O}^G(E_\mathcal{F}G,(S_n,d_{S_n})_{n \in \N};\mathcal{A})) \to K_m(\mathcal{O}^G(E_\mathcal{F}G,(Y_n,d_{Y_n})_{n \in \N};\mathcal{A})^{> \oplus})
    \]
    is a slight modification of the map $\trans_*$ defined in \cite[section~7]{Weg12}. We have to modify the definition of the objects $(A(n) \otimes C_*(n,\alpha))_{n \in \N}$ and $(A(n) \otimes D_*(n,\alpha))_{n \in \N}$, where $(A(n))_{n \in \N}$ is an object in $\mathcal{O}^G(E_\mathcal{F}G,(S_n,d_{S_n})_{n \in \N};\mathcal{A})$:
    \begin{eqnarray*}
    \big(A(n) \otimes C_k(n,\alpha)\big)_{(g,e,(g_1,g_2,x),t)} & := & A_{(g,e,(g_1,g_2),t)} \otimes C_k(n,\alpha)_{(g_1 g_2,x)}, \\
    \big(A(n) \otimes D_k(n,\alpha)\big)_{(g,e,(g_1,g_2,x),t)} & := & A_{(g,e,(g_1,g_2),t)} \otimes D_k(n,\alpha)_{(g_1 g_2,x)}.
    \end{eqnarray*}
    The maps $(\phi(n) \otimes m(n,\alpha))_{n \in \N}$ are given by
    \begin{align*}
    & \big(\phi(n) \otimes m(n,\alpha)\big)_{(g,e,(g_1,g_2,x),t),(g',e',(g'_1,g'_2,x'),t')} = \\
    & \phi(n)_{(g,e,(g_1,g_2),t),(g',e',(g'_1,g'_2),t')} \otimes m_{(g_1 g_2)^{-1}(g'_1 g'_2)}(n,\alpha)_{(g_1 g_2,x),(g'_1 g'_2,x')}.
    \end{align*}
    The proof follows as in \cite[section~7]{Weg12}.
\item The equation $K_m(\mathcal{O}^G(E_\mathcal{F}G,(\Sigma_n,n \cdot d^1)_{n \in \N};\mathcal{A})^{> \oplus}) = 0$ is proved in \cite[Theorem~7.2]{BLR08b}.
\item This is a direct consequence of \cite[Theorem~6.5]{BL12c}. Unfortunately, the notation slightly differs. Our category $\mathcal{O}^G(E_\mathcal{F}G,\pt;\mathcal{A})$ coincides with the category $\mathcal{O}^G(E,G,d_G)$ in \cite{BL12c}. The category $\mathcal{O}^G(E,(S_n,d_{S_n})$ in \cite{BL12c} is a full subcategory of our category $\mathcal{O}^G(E_\mathcal{F}G,(S_n,d'_{S_n})_{n \in \N};\mathcal{A})$ with a modified quasi-metric $d'_{S_n}$. This quasi-metric is defined by
    \[
    d'_{S_n}((g,h)_H,(g',h')_{H'}) := \left\{ \begin{array}{ll} d_G(gh,g'h') & \text{ if } H=H' \text{ and } gH=g'H',\\ \infty & \text{ otherwise}. \end{array} \right.
    \]
    The inclusion of the subcategory is induced by the map
    \[
    E_\mathcal{F}G \times S_n \to G \times E_\mathcal{F}G \times S_n, (e,(g_1,g_2)) \mapsto ((g_1 g_2,e,(g_1,g_2)).
    \]
    Note that the two quotient categories $\mathcal{O}^G(E_\mathcal{F}G,(S_n,d_{S_n})_{n \in \N};\mathcal{A})^{> \oplus}$ and $\mathcal{O}^G(E_\mathcal{F}G,(S_n,d'_{S_n})_{n \in \N};\mathcal{A})^{> \oplus}$ coincide.
\end{enumerate}
This finishes the proof of Proposition~\ref{prop-FHJ-K}.
\end{proof}

\subsection{Farrell-Hsiang-Jones groups and the $L$-theoretic Farrell-Jones Conjecture} \label{subsec-FHJ-L}

This subsection is dedicated to the proof of
\begin{proposition} \label{prop-FHJ-L}
Let $G$ be a FHJ group with respect to a family $\mathcal{F}$ of subgroups.
Suppose that $\mathcal{F}$ is virtually closed, i.e., $K < H < G$ with $K \in \mathcal{F}$ and $[H:K] < \infty$ implies $H \in \mathcal{F}$.
Then $G$ satisfies the $L$-theoretic Farrell-Jones Conjecture with respect to $\mathcal{F}$.
\end{proposition}

We begin this subsection with some preliminaries for the proof of Proposition~\ref{prop-FHJ-L}.
Let $\mathcal{A}$ be an additive $G$-category with involution $I$. For a chain complex $C$ over $\mathcal{A}$ we write $C^{-*}$ for the chain complex over $\mathcal{A}$ with $(C^{-*})_n := I(C_{-n})$ and differential $(-1)^n I(c_{-n+1})$, where $c_n \colon C_n \to C_{n-1}$ denotes the $n$-th differential of $C$. For a map $f \colon C \to D$ of degree $k$ the induced map $f^{-*} \colon D^{-*} \to C^{-*}$ is given by $(f^{-*})_n := (-1)^{nk} I(f_{-n}) \colon (D^{-*})_n \to (C^{-*})_{n+k}$.

\begin{definition}[$0$-dimensional ultra-quadratic Poincar\'e complex]
A \emph{$0$-dimensional ultra-quadratic Poincar\'e complex $(C,\psi)$ over $\mathcal{A}$} is a finite-dimensional chain complex $C$ over $\mathcal{A}$ together with a chain map $\psi \colon C^{-*} \to C$ of degree $0$ such that $\psi + \psi^{-*}$ is a chain homotopy equivalence. If $(C,\psi)$ is concentrated in degree $0$ then it is a \emph{quadratic form over $\mathcal{A}$}.
\end{definition}

\begin{proof}[Proof of Proposition~\ref{prop-FHJ-L}]
By Proposition~\ref{prop-assembly} it suffices to show
\[
L_0^{\langle -\infty \rangle}\big(\mathcal{O}^G(E_\mathcal{F}G,\pt;\mathcal{A})\big) = 0.
\]
Let $a \in L_0^{\langle -\infty \rangle}(\mathcal{O}^G(E_\mathcal{F}G,\pt;\mathcal{A}))$.
Because of Proposition~\ref{prop-FHJ-K} and the fact that $\mathcal{O}^G(E_\mathcal{F}G,\pt;\mathcal{A})$ is the cofiber of the assembly map (compare \cite[section 3.3]{BLR08b}) we have $K_m(\mathcal{O}^G(E_\mathcal{F}G,\pt;\mathcal{A})) = 0$ for all $m \in \Z$. This implies that the natural map
\[
L_0^{\langle 1 \rangle}(\mathcal{O}^G(E_\mathcal{F}G,\pt;\mathcal{A})) \to L_0^{\langle -\infty \rangle}(\mathcal{O}^G(E_\mathcal{F}G,\pt;\mathcal{A}))
\]
is an isomorphism (see \cite[Theorem~17.2 on page~146]{Ran92}).
Let $a'$ be the preimage of $a$. We pick a quadratic form $(M,\alpha)$ over the category $\mathcal{O}^G(E_\mathcal{F}G,\pt;\mathcal{A})$ such that $(M,\alpha) = a'$.

We fix a finite symmetric generating subset $S \subseteq G$ which contains the trivial element $e \in G$. We denote by $d_G$ the word metric with respect to $S \setminus \{e\}$.
Since $G$ is a FHJ group with respect to $\mathcal{F}$, there exist a natural number $N$ and surjective homomorphisms $\alpha_n \colon G \to F_n$ ($n \in \N$) with $F_n$ a finite group such that the following condition is satisfied.
For any hyper-elementary subgroup $H$ of $F_n$ there exist
\begin{itemize}
\item a compact, contractible, controlled $N$-dominated metric space $X_{n,H}$,
\item a strong homotopy $G$-action $\Psi_{n,H}$ on $X_{n,H}$,
\item a positive real number $\Lambda_{n,H}$,
\item a simplicial complex $E_{n,H}$ of dimension at most $N$ with a simplicial $\alpha_n^{-1}(H)$-action whose stabilizers belong to $\mathcal{F}$,
\item and an $\alpha_n^{-1}(H)$-equivariant map $f_{n,H} \colon G \times X_{n,H} \to E_{n,H}$
\end{itemize}
such that
\[
n \cdot d_{E_{n,H}}^1\big(f_{n,H}(g,x),f_{n,H}(h,y)\big) \leq d_{\Psi_{n,H},S^n,n,\Lambda_{n,H}}\big((g,x),(h,y)\big)
\]
for all $(g,x),(h,y) \in G \times X_n$ with $h^{-1}g \in S^n$.

We denote by $\mathcal{H}_n$ the family of hyper-elementary subgroups of $F_n$.
We define $S_n := \coprod_{H \in \mathcal{H}_n} G \times_{\alpha_n^{-1}(H)} G$ with the quasi metric
\[
d_{S_n}\big((g,h)_H,(g',h')_{H'}\big) := d_G(gh,g'h')
\]
if $H=H'$, $gH=g'H'$, and $(g'h')^{-1}(gh) \in S^n$. Otherwise, we set $d_{S_n}((g,h)_H,(g',h')_{H'}) := \infty$.

By \cite[Proposition~6.4]{BL12c} there are functors of additive categories with involutions
\[
G_n^+ = (G_n^+,E_n^+), G_n^- = (G_n^-,E_n^-) \colon \mathcal{O}^G(E_\mathcal{F}G,\pt;\mathcal{A}) \to \mathcal{O}^G(E_\mathcal{F}G,(S_n,d_{S_n});\mathcal{A})
\]
with the following properties:
\begin{itemize}
\item The functors $G_n^+$ and $G_n^-$ give rise to functors
    \[
    G^+ = (G_n^+)_{n \in \N}, G^- = (G_n^-)_{n \in \N} \colon \mathcal{O}^G(E_\mathcal{F}G,\pt;\mathcal{A}) \to \mathcal{O}^G(E_\mathcal{F}G,(S_n,d_{S_n});\mathcal{A})
    \]
\item For each object $M \in \mathcal{O}^G(E_\mathcal{F}G,\pt;\mathcal{A})$ the isomorphisms $E_n^+ \colon G_n^+(I(M)) \to I(G_n^+(M))$ and $E_n^- \colon G_n^-(I(M)) \to I(G_n^-(M))$ give rise to isomorphisms
    \begin{eqnarray*}
    E^+ = (E_n^+)_{n \in \N} \colon && G^+(I(M)) \to I(G^+(M)), \\
    E^- = (E_n^-)_{n \in \N} \colon && G^-(I(M)) \to I(G^-(M)).
    \end{eqnarray*}
    (Here, $I$ denotes the involution on the respective category.)
\item $L_0^{\langle -\infty \rangle}(\pr_k \circ G^+) - L_0^{\langle -\infty \rangle}(\pr_k \circ G^-)$ is the identity on $L_0^{\langle -\infty \rangle}(\mathcal{O}^G(E_\mathcal{F}G,\pt;\mathcal{A}))$.
\end{itemize}
We obtain quadratic forms $G^+(M,\alpha) := (G^+(M),G^+(\alpha) \circ (E^+)^{-1})$ and $G^-(M,\alpha) := (G^-(M),G^-(\alpha) \circ (E^-)^{-1})$ over $\mathcal{O}^G(E_\mathcal{F}G,(S_n,d_{S_n})_{n \in \N};\mathcal{A})$.

We denote the space of unordered pairs of two points in a space $X$ by $P_2(X)$, i.e., $P_2(X) := X \times X / \sim$, where $(x,y) \sim (y,x)$ for all $x,y \in X$. If $X$ is a metric space then we equip $P_2(X)$ with the metric
\[
d_{P_2(X)}\big([(x,y)],[(x',y')]\big) := \min\big\{ d_X(x,x')+d_X(y,y'),d_X(x,y')+d_X(x',y) \big\}.
\]
See \cite[section 9]{BL12a} for detailed information on the space $P_2(X)$.

We define the spaces
\begin{eqnarray*}
Y_n & := & \coprod_{H \in \mathcal{H}_n} G \times_{\alpha_n^{-1}(H)} G \times X_{n,H}, \\
Z_n & := & \coprod_{H \in \mathcal{H}_n} G \times_{\alpha_n^{-1}(H)} G \times P_2(X_{n,H}), \\
\Sigma_n & := & \coprod_{H \in \mathcal{H}_n} G \times_{\alpha_n^{-1}(H)} E_{n,H}.
\end{eqnarray*}
We will use the quasi-metrics on $Y_n$ and $Z_n$ given by
\begin{eqnarray*}
d_{Y_n}\big((g,h,x)_H,(g',h',x')_{H'}\big) & := & d_{\Psi_{n,H},S^n,n,\Lambda_{n,H}}((gh,x),(g'h',x')), \\
d_{Z_n}\big((g,h,x)_H,(g',h',x')_{H'}\big) & := & d_{P_2(\Psi_{n,H}),S^n,n,\Lambda_{n,H}}((gh,x),(g'h',x')),
\end{eqnarray*}
if $H=H'$, $gH=g'H'$, and $(g'h')^{-1}(gh) \in S^n$. Otherwise, $d_{Y_n}((g,h,x)_H,(g',h',x')_{H'}) := \infty$ and $d_{Z_n}((g,h,x)_H,(g',h',x')_{H'}) := \infty$.
We define the maps
\begin{align*}
& f_n := \coprod_{H \in \mathcal{H}_n} \id_G \times_{\alpha_n^{-1}(H)} f_{n,H} \colon Y_n \to \Sigma_n, \\
& p_n \colon Z_n \to S_n, (g,h,x)_H \mapsto (g,h)_H, \\
& w_n \colon Z_n \to P_2(Y_n), (g,h,[(x,y)])_H \mapsto [((g,h,x)_H,(g,h,y)_H)].
\end{align*}

Let us consider the following commuting diagram, where $\Idem$ stands for the idempotent completion.
\begin{equation*} \label{diagram2}
\xymatrix{
\mathcal{O}^G(E_\mathcal{F}G,(S_n,d_{S_n})_{n \in \N};\mathcal{A}) \ar@{^{(}->}[d] \ar[r]_{\hspace{7mm} \pr_k} & \mathcal{O}^G(E_\mathcal{F}G,\pt;\mathcal{A}) \ar@{^{(}->}[d] \\
\Idem\big( \mathcal{O}^G(E_\mathcal{F}G,(S_n,d_{S_n})_{n \in \N};\mathcal{A}) \big) \ar[r]_{\hspace{7mm} \pr_k} & \Idem\big( \mathcal{O}^G(E_\mathcal{F}G,\pt;\mathcal{A}) \big) \ar[d]^{=} \\
\Idem\big( \mathcal{O}^G(E_\mathcal{F}G,(Z_n,d_{Z_n})_{n \in \N};\mathcal{A}) \big) \ar[u]_{(p_n)*} \ar[d]^{(w_n)_*} \ar[r]_{\hspace{7mm} \pr_k} & \Idem\big( \mathcal{O}^G(E_\mathcal{F}G,\pt;\mathcal{A}) \big) \ar[d]^{=} \\
\Idem\big( \mathcal{O}^G(E_\mathcal{F}G,(P_2(Y_n),d_{P_2(Y_n)})_{n \in \N};\mathcal{A}) \big) \ar[d]^{(P_2(f_n))_*} \ar[r]_{\hspace{13mm} \pr_k} & \Idem\big( \mathcal{O}^G(E_\mathcal{F}G,\pt;\mathcal{A}) \big) \ar[d]^{=} \\
\Idem\big( \mathcal{O}^G(E_\mathcal{F}G,(P_2(\Sigma_n),n \cdot d^1_{P_2(\Sigma_n)})_{n \in \N};\mathcal{A}) \big) \ar[r]_{\hspace{16mm} \pr_k} & \Idem\big( \mathcal{O}^G(E_\mathcal{F}G,\pt;\mathcal{A}) \big)
}
\end{equation*}
The functor $(w_n)_*$ is well-defined because
\[
d_{P_2(Y_n)}\big(w_n(x),w_n(y)\big) \leq 2 \cdot d_{Z_n}(x,y)
\]
for all $n \in \N$ and $x,y \in Z_n$ (compare \cite[Lemma~9.7]{BL12a}).
Since
\[
n \cdot d_{E_{n,H}}^1\big(f_{n,H}(g,x),f_{n,H}(h,y)\big) \leq d_{\Psi_{n,H},S^n,n,\Lambda_{n,H}}\big((g,x),(h,y)\big)
\]
for all $(g,x),(h,y) \in G \times X_n$ with $h^{-1}g \in S^n$, \cite[Lemma~9.4 and 9.6]{BL12a} imply that the functor $(P_2(f_n))_*$ is well-defined.

We conclude as in \cite[Lemma~10.3]{BL12a} that there exist $0$-dimensional ultra-quadratic Poincar\'e complexes $(C^+,\psi^+)$, $(C^-,\psi^-)$ over $\Idem( \mathcal{O}^G(E_\mathcal{F}G,(Z_n,d_{Z_n})_{n \in \N};\mathcal{A}))$ with
\[
\big[(p_n)_*(C^\pm,\psi^\pm)\big] = \big[(G^\pm(M),G^\pm(\alpha) \circ (E^\pm)^{-1})\big] \in L_0^{\langle 1 \rangle}(\Idem\big( \mathcal{O}^G(E_\mathcal{F}G,(S_n,d_{S_n})_{n \in \N};\mathcal{A}) \big).
\]

We conclude from \cite[Theorem~5.3~(ii)]{BL12a} that the images of $[(P_2(f_n))_*(w_n)_*(C^\pm,\psi^\pm)]$ under the map
\begin{align*}
& L_0^{\langle -\infty \rangle}\big( \Idem\big( \mathcal{O}^G(E_\mathcal{F}G,(P_2(\Sigma_n),n \cdot d^1_{P_2(\Sigma_n)})_{n \in \N};\mathcal{A}) \big) \big) \to \\
& L_0^{\langle -\infty \rangle}\big( \Idem\big( \mathcal{O}^G(E_\mathcal{F}G,(P_2(\Sigma_k),k \cdot d^1_{P_2(\Sigma_k)});\mathcal{A}) \big) \big),
\end{align*}
which is induced by the projection on the $k$-th factor, vanish for $k$ sufficiently large.
For such a number $k \in \N$ we obtain in $L_0^{\langle -\infty \rangle}(\Idem(\mathcal{O}^G(E_\mathcal{F}G,\pt;\mathcal{A})))$ the equation
\[
\pr_k\big([(C^\pm,\psi^\pm)]\big) = \pr_k\big([P_2(f_n))_*(w_n)_*(C^\pm,\psi^\pm)]\big) = 0.
\]
Hence,
\[
\pr_k \big([G^\pm(M,\alpha)] \big) = \pr_k \circ (p_n)_* \big([(C^\pm,\psi^\pm)]\big) = \pr_k \big([(C^\pm,\psi^\pm)]\big) = 0.
\]
Since
\[
\pr_k \big([G^+(M,\alpha)] \big) - \pr_k \big([G^-(M,\alpha)] \big) = 0
\]
is the image of $a$ under the isomorphism
\[
L_0^{\langle -\infty \rangle}\big( \mathcal{O}^G(E_\mathcal{F}G,\pt;\mathcal{A}) \big) \to L_0^{\langle -\infty \rangle}\big( \Idem\big( \mathcal{O}^G(E_\mathcal{F}G,\pt;\mathcal{A}) \big) \big),
\]
we obtain $a=0$.
\end{proof}

\section{The Farrell-Jones Conjecture for the groups $G_w$}
\label{sec-f}

In this section we investigate the groups $G_w := \Z[w,w^{-1}] \rtimes_{\cdot w} \Z$, where $w$ is a non-zero algebraic number. We will construct overgroups (in particular, $\mathcal{O}_w \rtimes \mathcal{O}_w^\times$) and study their actions on certain spaces. We will finally show that the groups $G_w$ are FHJ groups (see Proposition~\ref{prop-final}).
For the convenience of the reader we have added a list of notation at the end of this section.

In the sequel $w \in \overline{\Q}^\times$ will be a fixed non-zero algebraic number.

\subsection{The ring $\mathbf{\mathcal{O}_w}$} \label{subsec-O_w}

Let $\mathcal{O}$ be the ring of integers in the algebraic number field $\Q(w)$, i.e., $\mathcal{O} \subset \Q(w)$ is the subring consisting of all elements which are integral. Recall that an algebraic number is integral if and only if it is a root of a monic polynomial with rational integer coefficients.

For a prime ideal $\mathfrak{p} \subset \mathcal{O}$ we denote by
\[
\mathcal{O}_\mathfrak{p} := \big\{ \frac{x}{y} \, \big| \, x, y \in \mathcal{O}, y \notin \mathfrak{p} \big\} \subset \Q(w)
\]
the localization of $\mathcal{O}$ at $\mathfrak{p}$. Let $v_\mathfrak{p} \colon \Q(w)^\times \to \Z$ be the corresponding valuation, i.e., $x \mathcal{O}_\mathfrak{p} = \mathfrak{p}^{v_\mathfrak{p}(x)} \mathcal{O}_\mathfrak{p}$. We extend the valuation $v_\mathfrak{p}$ to $\Q(w)$ by the convention $v_\mathfrak{p}(0) = \infty$. Note that $\mathcal{O}_\mathfrak{p} = \{ x \in \Q(w) \mid v_\mathfrak{p}(x) \geq 0 \}$.

A fractional ideal is a finitely generated $\mathcal{O}$-submodule $\mathfrak{a} \neq 0$ of $\Q(w)$. By \cite[Corollary 3.9 on page 22]{Neu99} every fractional ideal possesses a unique factorization $\mathfrak{a} = \prod_{\mathfrak{p}} \mathfrak{p}^{\nu_\mathfrak{p}}$ with $\nu_\mathfrak{p} \in \Z$ and $\nu_\mathfrak{p} = 0$ for almost all prime ideals $\mathfrak{p}$. For $\mathfrak{a} = x \mathcal{O}$ with $x \in \Q(w)^\times$ we have $\nu_\mathfrak{p} = v_\mathfrak{p}(x)$.
We conclude that
\[
M_w := \big\{ \mathfrak{p} \subset \mathcal{O} \text{ prime ideal } \big| \, v_\mathfrak{p}(w) \neq 0 \big\}
\]
is a finite set.
We define the ring
\[
\mathcal{O}_w := \big\{ x \in \Q(w) \, \big| \, v_\mathfrak{p}(x) \geq 0 \text{ for all prime ideals } \mathfrak{p} \notin M_w \big\}.
\]
Note that $\mathcal{O} \subseteq \mathcal{O}_w$ and $w, w^{-1} \in \mathcal{O}_w$.
The group of units in the ring $\mathcal{O}_w$ is given by
\[
\mathcal{O}_w^\times = \big\{ x \in \Q(w) \, \big| \, v_\mathfrak{p}(x) = 0 \text{ for all prime ideals } \mathfrak{p} \notin M_w \big\}.
\]

\begin{example}
For $w \in \Q$ we obtain $\Q(w) = \Q$, $\mathcal{O} = \Z$, $\mathcal{O}_w = \Z[w,w^{-1}]$.
\end{example}

\begin{lemma} \label{lem-units}
\begin{enumerate}
\item The abelian group $\mathcal{O}_w^\times$ is the direct product of the group $\mathcal{O}^\times$ of units in $\mathcal{O}$ and a finitely generated free abelian group. \label{lem-units-1}
\item The group $\mathcal{O}^\times$ is the direct product of a finite cyclic group with a finitely generated free abelian group. \label{lem-units-2}
\end{enumerate}
\end{lemma}
\begin{proof}
\begin{enumerate}
\item This follows from the short exact sequence
    \[
    1 \to \mathcal{O}^\times \to \mathcal{O}_w^\times \to \im\big(\mathcal{O}_w^\times \to \Z^{M_w}, x \mapsto (v_\mathfrak{p}(x))_{\mathfrak{p} \in M_w}\big) \to 1.
    \]
\item See Dirichlet's unit theorem \cite[Theorem 7.4 on page 42]{Neu99}.
\end{enumerate}
\end{proof}

We consider the semi-direct product $\mathcal{O}_w \rtimes \mathcal{O}_w^\times$, where $\mathcal{O}_w^\times$ acts on $\mathcal{O}_w$ by multiplication. There is the inclusion
\[
G_w := \Z[w,w^{-1}] \rtimes_{\cdot w} \Z \hookrightarrow \mathcal{O}_w \rtimes \mathcal{O}_w^\times, (x,y) \mapsto (x,w^y).
\]

\begin{lemma} \label{lem-fin-gen}
The groups $G_w$ and $\mathcal{O}_w \rtimes \mathcal{O}_w^\times$ are finitely generated.
\end{lemma}
\begin{proof}
The group $G_w$ is generated by $(1,0)$ and $(0,1)$.

By \cite[Proposition~2.10 on page 13]{Neu99}, $\mathcal{O}$ is a free $\Z$-module of rank $\deg(w) := \dim_{\Q}(\Q(w))$. By Lemma~\ref{lem-units}, $\mathcal{O}_w^\times$ is finitely generated. We fix finite generating sets $S$ resp. $S'$ for $\mathcal{O}$ resp. $\mathcal{O}_w^\times$. We will show that $S \times S'$ is a generating set for $\mathcal{O}_w \rtimes \mathcal{O}_w^\times$. Let $(x,y) \in \mathcal{O}_w \rtimes \mathcal{O}_w^\times$. Let $n_\mathfrak{p} \in \N$ ($\mathfrak{p} \in M_w$) such that $n_\mathfrak{p} + v_\mathfrak{p}(x) \geq 0$ and $\mathfrak{p}^{n_\mathfrak{p}}$ is a principal ideal. Such a number $n_\mathfrak{p}$ exists, since the ideal class group is finite (see \cite[Theorem~6.3 on page 36]{Neu99}). For $\mathfrak{p} \in M_w$ let $a_\mathfrak{p} \in \mathcal{O}$ with $a_\mathfrak{p} \mathcal{O} = \mathfrak{p}^{n_\mathfrak{p}}$. We set $a := \prod_{\mathfrak{p} \in M_w} a_\mathfrak{p} \in \mathcal{O}_w^\times$. Then $a \cdot x \in \mathcal{O}$. Hence,
\[
(x,y) = (0,a^{-1}) (ax,0) (0,ay)
\]
is generated by $S \times S'$.
\end{proof}

\begin{definition}[$\R^{n_w}$] \label{def-units}
We fix a $\Z$-basis $e_1, \ldots, e_{n_w}$ for the finitely generated free abelian subgroup of $\mathcal{O}^\times < \mathcal{O}_w^\times$ mentioned in Lemma~\ref{lem-units}~(\ref{lem-units-2}).
In the sequel we will denote this subgroup by $\Z^{n_w}$ (where $n_w$ is the rank of $\mathcal{O}^\times$).

We fix a projection $\alpha_w \colon \mathcal{O}_w^\times \to \Z^{n_w}$. Consider the semi-direct product $\Q(w) \rtimes \mathcal{O}_w^\times$, where $\mathcal{O}_w^\times$ acts on $\Q(w)$ by multiplication. We define an action of $\Q(w) \rtimes \mathcal{O}_w^\times$ on $\R^{n_w}$ by $(x,y)z := \alpha_w(y) + z$ ($x \in \Q(w)$, $y \in \mathcal{O}_w^\times$, $z \in \R^{n_w}$).
\end{definition}

\subsection{The tree $\mathbf{T(v_\mathfrak{p})}$} \label{ss-tree}

There is a tree associated to the valuation $v_\mathfrak{p}$. We shortly explain the construction. For details we refer to Serre's book on trees \cite[Chapter II, \S 1]{Ser03}.  We consider the (right) $\Q(w)$-vector space $\Q(w) \oplus \Q(w)$. A lattice of $\Q(w) \oplus \Q(w)$ is a finitely generated (right) $\mathcal{O}_\mathfrak{p}$-submodule of $\Q(w) \oplus \Q(w)$ which generates the $\Q(w)$-vector space. Such a module is free of rank 2. We call two lattices $L_1$, $L_2$ equivalent if and only if $L_1 = L_2 \cdot x$ for some $x \in \Q(w)^\times$. The vertices of the graph $T(v_\mathfrak{p})$ are the lattice classes. We fix an element $\pi_\mathfrak{p} \in \Q(w)$ with $v_\mathfrak{p}(\pi_\mathfrak{p})=1$. The distance between two classes $[L_1]$ and $[L_2]$ is defined as follows. There exist an $\mathcal{O}_\mathfrak{p}$-basis $(e_1,e_2)$ for $L_1$ and integers $a_1, a_2 \in \Z$ such that $(e_1\pi_\mathfrak{p}^{a_1},e_2\pi_\mathfrak{p}^{a_2})$ is an $\mathcal{O}_\mathfrak{p}$-basis for $L_2$. The distance is defined by $d([L_1],[L_2]) := |a_1-a_2|$. The vertices $[L_1]$, $[L_2]$ in the graph $T(v_\mathfrak{p})$ are adjacent if and only if $d([L_1],[L_2]) = 1$. See \cite[Theorem 1 on page 70]{Ser03} for a proof that the graph $T(v_\mathfrak{p})$ is a tree. More precisely, $T(v_\mathfrak{p})$ is an infinite regular $(n+1)$-valent tree, where $n$ denotes the order of the finite field $\mathcal{O}_\mathfrak{p} / \pi_\mathfrak{p} \mathcal{O}_\mathfrak{p} \cong \mathcal{O} / \mathfrak{p}$ (compare \cite[paragraph "Projective lines" on page 72]{Ser03}).

Note that the group $\GL_2(\Q(w))$ acts (isometrically) on the tree $T(v_\mathfrak{p})$. We consider the semi-direct product $\Q(w) \rtimes \Q(w)^\times$, where $\Q(w)^\times$ acts on $\Q(w)$ by multiplication. Using the homomorphism
\[
\Q(w) \rtimes \Q(w)^\times \to \GL_2(\Q(w)), \; (x,y) \mapsto \left( \begin{array}{cc} y & x \\ 0 & 1 \end{array} \right),
\]
we obtain an action of $\Q(w) \rtimes \Q(w)^\times$ on the tree $T(v_\mathfrak{p})$.

We define the Busemann function $f_\mathfrak{p} \colon T(v_\mathfrak{p}) \to \R$ by
\[
f_\mathfrak{p}(p) := \lim_{n \to \infty} n - d\big(p,[L_\mathfrak{p}(n)]\big)
\]
with $L_\mathfrak{p}(n) := \pi_\mathfrak{p}^{-n}\mathcal{O}_\mathfrak{p} \oplus \mathcal{O}_\mathfrak{p} \subset \Q(w) \oplus \Q(w)$.
Note that $f_\mathfrak{p}([L_\mathfrak{p}(n)])=n$. If the point $p$ does not lie on the line through the vertices $L_\mathfrak{p}(n)$ ($n \in \Z$) then $f_\mathfrak{p}(p) = m - d(p,[L_\mathfrak{p}(m)])$ with $[L_\mathfrak{p}(m)]$ the closest vertex in the line from $p$.

\begin{lemma} \label{lem-tree-1}
For $(x,y) \in \Q(w) \rtimes \Q(w)^\times$ and $p \in T(v_\mathfrak{p})$ we have
\[
f_\mathfrak{p}\big((x,y) \cdot p\big) = f_\mathfrak{p}(p) - v_\mathfrak{p}(y).
\]
\end{lemma}
\begin{proof}
We have
\begin{eqnarray*}
f_\mathfrak{p}\big((x,y) \cdot p\big) & = & \lim_{n \to \infty} n - d\big((x,y) \cdot p,[L_\mathfrak{p}(n)]\big) \\
& = & \lim_{n \to \infty} n - d\big(p,(x,y)^{-1} \cdot [L_\mathfrak{p}(n)]\big)
\end{eqnarray*}
with
\[
(x,y)^{-1} \cdot [L_\mathfrak{p}(n)] = \big[\left( \begin{array}{cc} y^{-1} & -x \\ 0 & 1 \end{array} \right) L_\mathfrak{p}(n)\big].
\]
The lattice $\left( \begin{array}{cc} y^{-1} & -x \\ 0 & 1 \end{array} \right) L_\mathfrak{p}(n)$ is the $\mathcal{O}_\mathfrak{p}$-module generated by $\left( \begin{array}{c} y^{-1} \cdot \pi_\mathfrak{p}^{-n} \\ 0 \end{array} \right)$ and $\left( \begin{array}{c} -x \\ 1 \end{array} \right)$.
There exists a unit $s \in \mathcal{O}_\mathfrak{p}^\times$ such that $y^{-1} \cdot \pi_\mathfrak{p}^{-n} = s \cdot \pi_\mathfrak{p}^{-n - v_\mathfrak{p}(y)}$. Hence, if $v_\mathfrak{p}(x) \geq - n - v_\mathfrak{p}(y)$ then
\[
\left( \begin{array}{cc} y^{-1} & -x \\ 0 & 1 \end{array} \right) L_\mathfrak{p}(n) = L_\mathfrak{p}(n + v_\mathfrak{p}(y)).
\]
We conclude
\begin{eqnarray*}
f_\mathfrak{p}\big((x,y) \cdot p\big) & = & \lim_{n \to \infty} n - d\big(p,(x,y)^{-1} \cdot [L_\mathfrak{p}(n)]\big) \\
& = & \lim_{n \to \infty} n - d\big(p,[L_\mathfrak{p}(n + v_\mathfrak{p}(y))]\big) \\
& = & \lim_{m \to \infty} (m - v_\mathfrak{p}(y)) - d\big(p,[L_\mathfrak{p}(m)]\big) \\
& = & f_\mathfrak{p}\big(p\big) - v_\mathfrak{p}(y).
\end{eqnarray*}
\end{proof}

\begin{lemma} \label{lem-tree-2}
Let $L$ be a lattice in $\Q(w) \oplus \Q(w)$. We denote by $z_1(L),z_2(L) \in \Z$ the elements satisfying $L \cap (\Q(w) \oplus 0) = \pi_\mathfrak{p}^{z_1(L)} \mathcal{O}_\mathfrak{p} \oplus 0$ and $L \cap (0 \oplus \Q(w)) = 0 \oplus \pi_\mathfrak{p}^{z_2(L)} \mathcal{O}_\mathfrak{p}$. Let $(x,y) \in \Q(w) \rtimes \Q(w)^\times$ with $(x,y) [L] = [L]$. Then $v_\mathfrak{p}(x) \geq z_1(L) - z_2(L)$ and $v_\mathfrak{p}(y) = 0$.
\end{lemma}
\begin{proof}
Lemma~\ref{lem-tree-1} implies $v_\mathfrak{p}(y) = 0$. We conclude $(x,y) L = L$ because
\begin{align*}
& ((x,y) L) \cap (\Q(w) \oplus 0) = (x,y) \big( L \cap (\Q(w) \oplus 0) \big) = (x,y) \big( \pi_\mathfrak{p}^{z_1(L)} \mathcal{O}_\mathfrak{p} \oplus 0 \big) = \\
& = y \pi_\mathfrak{p}^{z_1(L)} \mathcal{O}_\mathfrak{p} \oplus 0 = \pi_\mathfrak{p}^{z_1(L)} \mathcal{O}_\mathfrak{p} \oplus 0 = L \cap (\Q(w) \oplus 0).
\end{align*}
Since
\[
\left( \begin{array}{c} x \pi_\mathfrak{p}^{z_2(L)} \\ \pi_\mathfrak{p}^{z_2(L)} \end{array} \right) = \left( \begin{array}{cc} y & x \\ 0 & 1 \end{array} \right) \left( \begin{array}{c} 0 \\ \pi_\mathfrak{p}^{z_2(L)} \end{array} \right) \in (x,y) L = L,
\]
we have $(x \pi_\mathfrak{p}^{z_2(L)},0) \in L$ and hence $v_\mathfrak{p}(x) + z_2(L) \geq z_1(L)$.
\end{proof}

\begin{lemma} \label{lem-tree-3}
Let $L$ be a lattice in $\Q(w) \oplus \Q(w)$. We denote by $z_1(L) \in \Z$ the element satisfying $L \cap (\Q(w) \oplus 0) = \pi_\mathfrak{p}^{z_1(L)} \mathcal{O}_\mathfrak{p} \oplus 0$. We set
\[
z'_2(L) := \min\big\{v_\mathfrak{p}(b) \, \big| \, (a,b) \in L \big\}.
\]
Let $x \in \Q(w)$ with $v_\mathfrak{p}(x) \geq z_1(L) - z'_2(L)$. Then $(x,1) L = L$.
\end{lemma}
\begin{proof}
Let $(a,b) \in L$. Then $(xb,0) \in L$ because $v_\mathfrak{p}(xb) = v_\mathfrak{p}(x) + v_\mathfrak{p}(b) \geq z_1(L)$. Hence, $(a+xb,b) \in L$. This shows that $(x,1) L \subseteq L$.
Analogously, we conclude $(x,1)^{-1} L = (-x,1) L \subseteq L$.
Therefore, $(x,1) L = L$.
\end{proof}

\subsection{The Minkowski space $\mathbf{\Q(w)_\R}$} \label{ss-minkowski}

There are $\deg(w)$ embeddings $\tau \colon \Q(w) \hookrightarrow \C$, where $\deg(w) := \dim_\Q(\Q(w))$ denotes the degree of $w$.
We consider the map
\[
j \colon \Q(w) \to \Q(w)_\C := \prod_\tau \C, \, x \mapsto (\tau(x))_\tau.
\]
The \emph{Minkowski space $\Q(w)_\R$} is defined as the subset
\[
\Q(w)_\R := \big\{ (z_\tau)_\tau \in \prod_\tau \C \, \big| \, z_{\overline{\tau}} = \overline{z_\tau} \text{ for all } \tau \big\} \subset \Q(w)_\C.
\]
This is a real vector space of dimension $n$. Note that the image of the map $j$ is contained in $\Q(w)_\R$.
The $\C$-vector space $\Q(w)_\C$ is equipped with the hermitian scalar product $\langle x , y \rangle =  \sum_\tau x_\tau \overline{y_\tau}$. We obtain a scalar product on $\Q(w)_\R$ by restriction.

We define an equivalence relation on the embeddings $\tau \colon \Q(w) \hookrightarrow \C$ by $\tau \sim \tau'$ if $\tau' \in \{\tau,\overline{\tau}\}$.
We decompose $\Q(w)_\C = \bigoplus_{[\tau]} \Q(w)_{\C,[\tau]}$ with
\[
\Q(w)_{\C,[\tau]} := \prod_{\tau' \in [\tau]} \C \subseteq \prod_{\tau'} \C = \Q(w)_\C
\]
and $\Q(w)_\R = \prod_{[\tau]} \Q(w)_{\R,[\tau]}$ with
\[
\Q(w)_{\R,[\tau]} = \Q(w)_{\C,[\tau]} \cap \Q(w)_\R.
\]

The group $\Q(w) \rtimes \Q(w)^\times$ acts on $\Q(w)_\C$ and $\Q(w)_\R$ by
\[
(x,y) \cdot (z_\tau)_\tau := \big( \tau(y) \cdot z_\tau + \tau(x) \big)_\tau.
\]
We obtain actions of $\Q(w) \rtimes \Q(w)^\times$ on $\Q(w)_{\C,[\tau]}$ respectively $\Q(w)_{\R,[\tau]}$ by restriction.

By \cite[Proposition~5.2 on page 31 and Definition~4.1 on page 24]{Neu99} there exists for every ideal $\mathfrak{a} \neq 0$ in $\mathcal{O}$ a basis $v_1, \ldots, v_{\deg(w)} \in \Q(w)_\R$ such that $j(\mathfrak{a}) = \Z v_1 + \ldots + \Z v_{\deg(w)}$. A fundamental domain for the free action of the group $\mathfrak{a} = \mathfrak{a} \rtimes \{1\} < \Q(w) \rtimes \Q(w)^\times$ on $\Q(w)_\R$ is
\[
\big\{ r_1 v_1 + \ldots + r_{\deg(w)} v_{\deg(w)} \, \big| \, r_i \in \R, 0 \leq r_i < 1 \big\}.
\]
We conclude
\begin{lemma} \label{lem-min}
Let $\mathfrak{a} \neq 0$ be an ideal in $\mathcal{O}$. The additive group $\mathfrak{a}$ acts freely, properly, and cocompactly on $\Q(w)_\R$.
\end{lemma}

For more information on the Minkowski space $\Q(w)_\R$ we refer to \cite[Chapter 1, \S5]{Neu99}.

\subsection{The metric spaces $\mathbf{X_w}$ and $\mathbf{Y_w}$} \label{subsec-XY_w}

We define $X_w := \R^{n_w} \times \prod_{\mathfrak{p} \in M_w} T(v_\mathfrak{p})$. Recall that $n_w$ is the rank of $\mathcal{O}^\times$ (see Definition~\ref{def-units}), $T(v_\mathfrak{p})$ is the tree defined in subsection~\ref{ss-tree}. Note that $X_w$ is a finite-dimensional proper CAT(0)-space, since each factor has this property. We will need the space $X_w$ for the proof that $\mathcal{O}_w \rtimes_{\cdot w} \Z$ is a FHJ group.
Consider the semi-direct product $\Q(w) \rtimes \mathcal{O}_w^\times$, where $\mathcal{O}_w^\times$ acts on $\Q(w)$ by multiplication. The group $\Q(w) \rtimes \mathcal{O}_w^\times$ acts diagonally on $X_w$. This action is isometric because it is isometric on each factor.

We are also interested in the space $X_w \times \Q(w)_{\R}$, where $\Q(w)_{\R}$ is the Minkowski space (see subsection~\ref{ss-minkowski}). For $(x,y) \in \Q(w) \rtimes \mathcal{O}_w^\times$ and $z_\tau, z'_\tau \in \C \subset \Q(w)_{\C}$ we calculate
\[
\big|(x,y) \cdot z_\tau - (x,y) \cdot z'_\tau\big| = \big|\big(\tau(y) \cdot z_\tau + \tau(x)\big) - \big(\tau(y) \cdot z'_\tau + \tau(x)\big)\big| = |\tau(y)| \cdot |z_\tau - z'_\tau|.
\]
This shows that the action of $\Q(w) \rtimes \mathcal{O}_w^\times$ on $\Q(w)_{\R}$ respectively $X_w \times \Q(w)_{\R}$ is not isometric in general.
Since we want an isometric action on $X_w \times \Q(w)_{\R}$, we have to modify the metric.

\begin{definition}[warped product $M \times_f N$]
Let $(M,d_M)$ and $(N,d_N)$ be length spaces and let $f \colon M \to \R^{>0}$ be a continuous function on $M$.
The \emph{warped product $M \times_f N$} is defined as follows.
The length $l(\gamma)$ of a path $\gamma = (\gamma_M,\gamma_N) \colon [0,1] \to M \times N$ is defined as
\[
l(\gamma) := \lim \sum_{i=1}^n \sqrt{d_M(\gamma_M(t_{i-1}),\gamma_M(t_i))^2 + f(\gamma_M(t_{i-1}))^2 \cdot d_N(\gamma_N(t_{i-1}),\gamma_N(t_i))^2}
\]
where the limit is taken with respect to the refinement ordering of partitions
\[
0 = t_0 < t_1 < t_2 < \cdots < t_n = 1
\]
of the interval $[0,1]$.
The distance between two points $x, y \in M \times_f N$ is given as
\[
d(x,y) := \inf \{ l(\gamma) \mid \text{$\gamma$ is a path from $x$ to $y$.} \}.
\]
(See \cite[Proposition~3.1]{Che99} for the proof that $d$ is a metric.)
\end{definition}

\begin{lemma} \label{lem-warped}
\begin{enumerate}
\item The topology of the warped product $M \times_f N$ coincides with the product topology.
\item If the metric spaces $M$ and $N$ are proper then the warped product $M \times_f N$ is proper.
\end{enumerate}
\end{lemma}
\begin{proof}
\begin{enumerate}
\item Let $B_r(x) \subset M \times_f N$ be a ball. We have to show that there exist open subsets $U \subset M$ and $V \subset N$ with $x \in U \times V \subset B_r(x)$. We set $U := B_{r/2}(\pr_M(x)) \subset M$, $m := \max\{f(y) \mid y \in \overline{B}_{r/2}(\pr_M(x))\}$, and $V := B_{r/2m}(\pr_N(x)) \subset N$. Then $x \in U \times V \subset B_r(x)$.\\
    Now, let $U \times V \subset M \times N$ be an open neighborhood of $(y,z)$ with respect to the product topology. We have to show that there exists $r > 0$ such that $B_r(y,z) \subset U \times V$. Let $r_U, r_V > 0$ with $B_{r_U}(y) \subset U$ and $B_{r_V}(z) \subset V$. We define $r := \min\{r_U,n \cdot r_V\}$ with $n := \min\{f(y) \mid y \in \overline{B}_{r_U}(y)\}$. Then $d_{M \times_f N}((y,z),(y',z')) < r$ implies $d_M(y,y') < r \leq r_U$ and $n \cdot d_N(z,z') < r \leq n \cdot r_V$. Hence, $(y',z') \in B_{r_U}(y) \times B_{r_V}(z) \subset U \times V$.
\item Let $\overline{B}_r(x) \subset M \times_f N$ be a closed ball. We have to show that this ball is compact.
    The space $\pr_M(\overline{B}_r(x))$ is compact because it is contained in the compact set $\overline{B}_r(\pr_M(x))$. We set $m := \min\{f(y) \mid y \in \pr_M(\overline{B}_r(x))\}$. Note that $\pr_N(\overline{B}_r(x))$ is contained in $\overline{B}_{r/m}(\pr_N(x))$ and hence compact.
    We conclude that $\overline{B}_r(x) \subset \pr_M(\overline{B}_r(x)) \times \pr_N(\overline{B}_r(x))$ is compact.
\end{enumerate}
\end{proof}

\begin{definition}[metric space $(Y_w,d_{Y_w})$] \label{def-Y_w}
Since the ideal class group of $\mathcal{O}$ is finite (see \cite[Theorem 6.3 on page 36]{Neu99}), there exists for every $\mathfrak{p} \in M_w$ an element $y_\mathfrak{p} \in \mathcal{O}$ such that $y_\mathfrak{p} \mathcal{O} = \mathfrak{p}^{v_{\mathfrak{p}}(y_\mathfrak{p})}$. We fix such an element $y_\mathfrak{p}$ for every $\mathfrak{p} \in M_w$.
We define the metric space $(Y_w,d_{Y_w})$ as
\[
Y_w := \big\{ (x_{[\tau]},y_{[\tau]})_{[\tau]} \in \prod_{[\tau]} X_w \times_{f^{[\tau]}} \Q(w)_{\R,[\tau]} \, \big| \, x_{[\tau]} = x_{[\tau']} \text{ for all } \tau, \tau' \big\}
\]
with $f^{[\tau]}(r,(p_\mathfrak{p})_{\mathfrak{p} \in M_w}) := \prod_{i=1}^{n_w} |\tau(e_i)|^{-r_i} \cdot \prod_{\mathfrak{p} \in M_w} |\tau(y_\mathfrak{p})|^{f_{\mathfrak{p}}(p_\mathfrak{p})/v_{\mathfrak{p}}(y_\mathfrak{p})}$.
Here, $e_1, \ldots, e_{n_w}$ denotes the $\Z$-basis for $\Z^{n_w} < \mathcal{O}^\times < \mathcal{O}_w^\times$ (see Definition~\ref{def-units}).
\end{definition}

\begin{lemma} \label{lem-action-isometric}
The action of $\Q(w) \rtimes \mathcal{O}_w^\times$ on $(Y_w,d_{Y_w})$ is isometric.
\end{lemma}
\begin{proof}
We have to prove the equation
\[
f^{[\tau]}((x,y)r,((x,y)p_\mathfrak{p})_{\mathfrak{p} \in M_w}) \cdot |\tau(y)| = f^{[\tau]}(r,(p_\mathfrak{p})_{\mathfrak{p} \in M_w})
\]
for all $(x,y) \in \Q(w) \rtimes \mathcal{O}_w^\times$, $r \in \R^{n_w}$ and $p_\mathfrak{p} \in T(v_\mathfrak{p})$ ($\mathfrak{p} \in M_w$).
This equation is equivalent to
\[
\prod_{i=1}^{n_w} |\tau(e_i)|^{\alpha_w(y)_i} \cdot \prod_{\mathfrak{p} \in M_w} |\tau(y_\mathfrak{p})|^{v_{\mathfrak{p}}(y)/v_{\mathfrak{p}}(y_\mathfrak{p})} = |\tau(y)|.
\]
There exists a natural number $l$ such that $y^l = \prod_{i=1}^{n_w} e_i^{l_i} \cdot \prod_{\mathfrak{p} \in M_w} y_\mathfrak{p}^{l_\mathfrak{p}}$
for some $l_i, l_\mathfrak{p} \in \Z$.
We conclude
\begin{align*}
& \Big( \prod_{i=1}^{n_w} |\tau(e_i)|^{\alpha_w(y)_i} \cdot \prod_{\mathfrak{p} \in M_w} |\tau(y_\mathfrak{p})|^{v_{\mathfrak{p}}(y)/v_{\mathfrak{p}}(y_\mathfrak{p})} \Big)^l = \\
& \prod_{i=1}^{n_w} |\tau(e_i)|^{\alpha_w(y^l)_i} \cdot \prod_{\mathfrak{p} \in M_w} |\tau(y_\mathfrak{p})|^{v_{\mathfrak{p}}(y^l)/v_{\mathfrak{p}}(y_\mathfrak{p})} = \\
& \prod_{i=1}^{n_w} |\tau(e_i)|^{l_i} \cdot \prod_{\mathfrak{p} \in M_w} |\tau(y_\mathfrak{p})|^{l_{\mathfrak{p}}} = |\tau(y)|^l.
\end{align*}
This implies
\[
\prod_{i=1}^{n_w} |\tau(e_i)|^{\alpha_w(y)_i} \cdot \prod_{\mathfrak{p} \in M_w} |\tau(y_\mathfrak{p})|^{v_{\mathfrak{p}}(y)/v_{\mathfrak{p}}(y_\mathfrak{p})} = |\tau(y)|.
\]
\end{proof}

\begin{lemma} \label{lem-proper}
\begin{enumerate}
\item The topology of the metric space $(Y_w,d_{Y_w})$ coincides with the product topology of $X_w \times \Q(w)_\R$. \label{lem-proper-1}
\item The metric space $(Y_w,d_{Y_w})$ is proper. \label{lem-proper-2}
\end{enumerate}
\end{lemma}
\begin{proof}
This follows from Lemma~\ref{lem-warped}.
\end{proof}

Next, we show that the action of $\mathcal{O}_w \rtimes \mathcal{O}_w^\times < \Q(w) \rtimes \mathcal{O}_w^\times$ on $Y_w$ is proper and cocompact.

\begin{lemma} \label{lem-action-proper}
The action of $\mathcal{O}_w \rtimes \mathcal{O}_w^\times$ on $Y_w$ is proper.
\end{lemma}
\begin{proof}
Let $K_1, K_2 \subset Y_w$ be a compact subsets. We will prove by contradiction that the set
\[
S := \big\{ (x,y) \in \mathcal{O}_w \rtimes \mathcal{O}_w^\times \, \big| \, K_1 \cap (x,y) K_2 \neq \emptyset \big\}
\]
is finite.
Assume that this set is infinite.
Since
\[
A_i(\mathfrak{p}) := \big\{ [L] \text{ vertex of } T(v_\mathfrak{p}) \, \big| \, d_{T(v_\mathfrak{p})}([L],p) < 1 \text{ for some } p \in \pr_{T(v_\mathfrak{p})}(K_i) \big\}
\]
($i = 1,2$, $\mathfrak{p} \in M_w$) are finite sets, there exist elements $[L_i(\mathfrak{p})] \in A_i(\mathfrak{p})$ and an infinite subset $S' \subset S$ such that $(x,y) [L_2(\mathfrak{p})] = [L_1(\mathfrak{p})]$ for all $(x,y) \in S'$ and $\mathfrak{p} \in M_w$. We choose $(x_0,y_0) \in S'$ and set $\widetilde{S} := S' (x_0,y_0)^{-1}$. Then $(x,y) [L_1(\mathfrak{p})] = [L_1(\mathfrak{p})]$ for all $(x,y) \in \widetilde{S}$ and $\mathfrak{p} \in M_w$. By Lemma~\ref{lem-tree-2} we have $v_\mathfrak{p}(y) = 0$ for all $(x,y) \in \widetilde{S}$. Hence, $y \in \mathcal{O}^\times$ for all $(x,y) \in \widetilde{S}$. Since $\mathcal{O}^\times$ acts properly on $\R^{n_w}$ and $\pr_{\R^{n_w}}(K_i) \subset \R^{n_w}$ ($i=1,2$) are compact subsets, we conclude that the set $\{ y \in \mathcal{O}^\times \mid (x,y) \in \widetilde{S} \}$ is finite. Therefore, there exists $y' \in \mathcal{O}^\times$ such that $T := \{ x \in \mathcal{O}_w \mid (x,y') \in \widetilde{S} \}$ is an infinite set. Lemma~\ref{lem-tree-2} implies that there exists $z \in \Z$ such that $v_\mathfrak{p}(x) \geq z$ for all $x \in T$ and $\mathfrak{p} \in M_w$. We choose $x' \in \mathcal{O}_w^\times$ with $v_\mathfrak{p}(x') \geq -z$ for all $\mathfrak{p} \in M_w$. Then $T' := T x'$ is an infinite subset of $\mathcal{O}$. We define the following compact subsets of $\Q(w)_\R$: $K'_1 := \pr_{\Q(w)_\R}((0,x')K_1)$, $K'_2 := \pr_{\Q(w)_\R}((0,x'y')(x_0,y_0)K_2)$. Since
\[
(0,x')K_1 \cap (t,1)(0,x'y')(x_0,y_0)K_2 \neq \emptyset,
\]
for all $t \in T'$, we have $K'_1 \cap t K'_2 \neq \emptyset$. This is a contradiction because $\mathcal{O}$ acts properly on $\Q(w)_\R$ (see Lemma~\ref{lem-min}).
\end{proof}

\begin{lemma} \label{lem-action-cocompact}
The action of $\mathcal{O}_w \rtimes \mathcal{O}_w^\times$ on $Y_w$ is cocompact.
\end{lemma}
\begin{proof}
We define the compact subsets
\[
K_\mathfrak{p} := \overline{B}_{v_\mathfrak{p}(y_\mathfrak{p})}([\mathcal{O}_\mathfrak{p} \oplus \mathcal{O}_\mathfrak{p}]) \subset T(v_\mathfrak{p})
\]
for $\mathfrak{p} \in M_w$.

Note that $(0,y) [\mathcal{O}_\mathfrak{p} \oplus \mathcal{O}_\mathfrak{p}] = [\mathcal{O}_\mathfrak{p} \oplus \mathcal{O}_\mathfrak{p}]$ for all $y \in \mathcal{O}^\times$. This shows that $\mathcal{O}^\times$ permutes the finitely many vertices in the closed ball $K_\mathfrak{p}$. Hence, there exist $n_i \in \N$ ($i=1,\ldots,n_w)$ such that $\Z n_i e_i  \in \Z^{n_w} < \mathcal{O}^\times$ acts trivially on $K_\mathfrak{p}$ for every $\mathfrak{p} \in M_w$.

By Lemma~\ref{lem-tree-3} there exist $z_\mathfrak{p} \in \Z$ such that $(x,1)p = p$ for all $p \in K_\mathfrak{p}$ and all $x \in \Q(w)$ with $v_\mathfrak{p}(x) \geq z_\mathfrak{p}$.
Since the ideal $\mathfrak{a} := \{ x \in \mathcal{O} \mid v_\mathfrak{p}(x) \geq z_\mathfrak{p} \}$ acts cocompactly on $\Q(w)_\R$ (see Lemma~\ref{lem-min}), there exists a compact subset $K' \subset \Q(w)_\R$ such that $\mathfrak{a} \cdot K' = \Q(w)_\R$.

We consider the compact subset
\[
K := \prod_{\mathfrak{p} \in M_w} K_\mathfrak{p} \times \prod_{i=1}^{n_w} [0,n_i] \times K' \subset Y_w.
\]
We have to show that $\mathcal{O}_w \rtimes \mathcal{O}_w^\times \cdot K = Y_w$.
Since $\mathfrak{a} < \mathcal{O} \rtimes \{1\}$ acts trivially on $\prod_{\mathfrak{p} \in M_w} K_\mathfrak{p} \times \R^{n_w}$, we have
\[
\mathfrak{a} K = \prod_{\mathfrak{p} \in M_w} K_\mathfrak{p} \times \prod_{i=1}^{n_w} [0,n_i] \times \Q(w)_\R.
\]
Therefore, it suffices to prove the equation
\[
\mathcal{O}_w \rtimes \mathcal{O}_w^\times \cdot \prod_{\mathfrak{p} \in M_w} K_\mathfrak{p} \times \prod_{i=1}^{n_w} [0,n_i] = X_w.
\]
Since $\bigoplus_{i=1}^{n_w} \Z n_i e_i < \mathcal{O}^\times$ acts trivially on $\prod_{\mathfrak{p} \in M_w} K_\mathfrak{p}$,
we conclude
\[
\bigoplus_{i=1}^{n_w} \Z n_i e_i \cdot \prod_{\mathfrak{p} \in M_w} K_\mathfrak{p} \times \prod_{i=1}^{n_w} [0,n_i] = \prod_{\mathfrak{p} \in M_w} K_\mathfrak{p} \times \R^{n_w}.
\]
Hence, it remains to show
\[
\mathcal{O}_w \rtimes \mathcal{O}_w^\times \cdot \prod_{\mathfrak{p} \in M_w} K_\mathfrak{p} = \prod_{\mathfrak{p} \in M_w} T(v_\mathfrak{p}).
\]
Let $[L_\mathfrak{p}] \in T(v_\mathfrak{p})$ ($\mathfrak{p} \in M_w$) be vertices. We will construct an element $g \in \mathcal{O}_w \rtimes \mathcal{O}_w^\times$ such that $d_{T(v_\mathfrak{p})}(g[L_\mathfrak{p}],[\mathcal{O}_\mathfrak{p} \oplus \mathcal{O}_\mathfrak{p}]) \leq v_\mathfrak{p}(y_\mathfrak{p}) - 1$.

Let $(a_\mathfrak{p},b_\mathfrak{p}), (c_\mathfrak{p},d_\mathfrak{p})$ be an $\mathcal{O}_\mathfrak{p}$-basis for $L_\mathfrak{p}$.
Without loss of generality we assume $v_\mathfrak{p}(b_\mathfrak{p}) \geq v_\mathfrak{p}(d_\mathfrak{p})$. By adding a multiple of $(c_\mathfrak{p},d_\mathfrak{p})$ to $(a_\mathfrak{p},b_\mathfrak{p})$ we can and will assume that $b_\mathfrak{p} = 0$.
Let $k_\mathfrak{p} \in \N$ such that
\[
k_\mathfrak{p} v_\mathfrak{p}(y_\mathfrak{p}) \geq \max\{v_\mathfrak{p}(d_\mathfrak{p}) - v_\mathfrak{p}(c_\mathfrak{p}), v_\mathfrak{p}(d_\mathfrak{p}) - v_\mathfrak{p}(a)\}.
\]
Then there exist $x_\mathfrak{p} \in \mathcal{O}_\mathfrak{p}$ such that
\[
\prod_{\mathfrak{p'} \in M_w} y_\mathfrak{p'}^{k_\mathfrak{p'}} \cdot c_\mathfrak{p} + x_\mathfrak{p} \cdot d_\mathfrak{p} = 0.
\]
The elements $x_\mathfrak{p}$ represent elements
\[
[x_\mathfrak{p}] \in \mathcal{O}_\mathfrak{p} / \mathfrak{p}^{v_\mathfrak{p}(a_\mathfrak{p}) - v_\mathfrak{p}(d_\mathfrak{p}) + k_\mathfrak{p} v_\mathfrak{p}(y_\mathfrak{p})} \mathcal{O}_\mathfrak{p} \cong \mathcal{O} / \mathfrak{p}^{v_\mathfrak{p}(a_\mathfrak{p}) - v_\mathfrak{p}(d_\mathfrak{p}) + k_\mathfrak{p} v_\mathfrak{p}(y_\mathfrak{p})}
\]
(see \cite[Corollary~11.2 on page~70]{Neu99} for the isomorphism).
By applying the Chinese Reminder Theorem
\[
\mathcal{O} / \prod_{\mathfrak{p} \in M_w} \mathfrak{p}^{v_\mathfrak{p}(a_\mathfrak{p}) - v_\mathfrak{p}(d_\mathfrak{p}) + k_\mathfrak{p} v_\mathfrak{p}(y_\mathfrak{p})} \cong \bigoplus_{\mathfrak{p} \in M_w} \mathcal{O} / \mathfrak{p}^{v_\mathfrak{p}(a_\mathfrak{p}) - v_\mathfrak{p}(d_\mathfrak{p}) + k_\mathfrak{p} v_\mathfrak{p}(y_\mathfrak{p})}.
\]
we obtain an element $x' \in \mathcal{O}$ such that
\[
[x'] = [x_\mathfrak{p}] \in \mathcal{O}_\mathfrak{p} / \mathfrak{p}^{v_\mathfrak{p}(a_\mathfrak{p}) - v_\mathfrak{p}(d_\mathfrak{p}) + k_\mathfrak{p} v_\mathfrak{p}(y_\mathfrak{p})} \mathcal{O}_\mathfrak{p}.
\]
We conclude
\[
v_\mathfrak{p}\big( \prod_{\mathfrak{p'} \in M_w} y_\mathfrak{p'}^{k_\mathfrak{p'}} \cdot c_\mathfrak{p} + x' \cdot d_\mathfrak{p} \big) = v_\mathfrak{p}(x' - x_\mathfrak{p}) + v_\mathfrak{p}(d_\mathfrak{p}) \geq v_\mathfrak{p}(a_\mathfrak{p}) + k_\mathfrak{p} v_\mathfrak{p}(y_\mathfrak{p}).
\]
We set
\[
\widetilde{x} := x' \cdot \prod_{\mathfrak{p'} \in M_w} y_\mathfrak{p'}^{-k_\mathfrak{p'}} \in \mathcal{O}_w
\]
and obtain $v_\mathfrak{p}\big( c_\mathfrak{p} + \widetilde{x} \cdot d_\mathfrak{p} \big) \geq v_\mathfrak{p}(a_\mathfrak{p})$.
This implies that $(a_\mathfrak{p},0), (0,d_\mathfrak{p})$ is an $\mathcal{O}_\mathfrak{p}$-basis for $(\widetilde{x},1) L_\mathfrak{p}$.
We choose $l_\mathfrak{p} \in \Z$ such that $|l_\mathfrak{p} v_\mathfrak{p}(y_\mathfrak{p}) + v_\mathfrak{p}(a_\mathfrak{p}) - v_\mathfrak{p}(d_\mathfrak{p})| \leq v_\mathfrak{p}(y_\mathfrak{p}) - 1$ and set $y := \prod_{\mathfrak{p'} \in M_w} y_\mathfrak{p'}^{l_\mathfrak{p'}}$.
We conclude
\[
d_{T(v_\mathfrak{p})}\big((0,y)(\widetilde{x},1)[L_\mathfrak{p}],[\mathcal{O}_\mathfrak{p} \oplus \mathcal{O}_\mathfrak{p}]\big) \leq v_\mathfrak{p}(y_\mathfrak{p}) - 1
\]
because $(y a_\mathfrak{p},0), (0,d_\mathfrak{p})$ is an $\mathcal{O}_\mathfrak{p}$-basis for $(0,y) (\widetilde{x},1) L_\mathfrak{p}$ and
$|v_\mathfrak{p}(y a_\mathfrak{p}) - v_\mathfrak{p}(d_\mathfrak{p})| \leq v_\mathfrak{p}(y_\mathfrak{p}) - 1$.
\end{proof}

\begin{remark}
For a prime number $w \in \N$ the metric space $Y_w$ coincides with the metric space $T_d \times_{f_d} \R$ constructed in \cite[section~2]{FW14}.
\end{remark}

\subsection{The flow space $\mathbf{FS_w}$}

In this subsection we define a flow space for $Y_w$. It will be used to construct convenient open covers. These open covers will be needed to construct the simplicial complexes appearing in Definition~\ref{def-FHJ}.

\begin{definition}[generalized geodesic]
Let $(X,d_X)$ be a metric space. A continuous map $c \colon \R \to X$ is called a \emph{generalized geodesic} if there are $c_-, c_+ \in \R \cup \{\pm \infty\}$ with $\infty \neq c_- \leq c_+ \neq -\infty$ such that $c$ restricts to an isometry on the interval $(c_-,c_+)$ and is locally constant on the complement of this interval.
\end{definition}

\begin{definition}[flow space $FS(X)$] \label{def-FS}
Let $(X,d_X)$ be a metric space. The \emph{flow space $FS(X)$} is the set of all generalized geodesics in $X$. We provide $FS(X)$ with the metric
\[
d_{FS(X)}(c,d) := \int_{-\infty}^\infty \frac{d_X(c(t),d(t))}{2 \cdot e^{|t|}} \, dt.
\]
We define a $G$-equivariant flow $\Phi \colon FS(X) \times \R \to FS(X)$ by $\Phi_\tau(c)(t) := c(t + \tau)$
\end{definition}
We refer to \cite[section~1]{BL12b} for more information on the flow space $FS(X)$.

Analogously to \cite[section~3]{FW14} we will use the subspace
\[
FS_w := FS(X_w) \times \Q(w)_\R \subset FS\big( Y_w \big)
\]
as \emph{flow space for $Y_w$}. (See also \cite[Lemma~3.2]{Che99}.)
This subspace consists of all generalized geodesics which are constant on $\Q(w)_\R$.

\begin{lemma} \label{lem-FS_w}
The flow space $FS_w$ has the following properties:
\begin{enumerate}
\item The metric space $FS_w$ is proper. \label{lem-FS_w-1}
\item The action of $\mathcal{O}_w \rtimes \mathcal{O}_w^\times$ on $FS_w$ is isometric, proper, and cocompact. \label{lem-FS_w-2}
\item Let $FS_w^\R := \{ c \in FS_w \mid \Phi_t(c) = c \text{ for all } t \in \R\}$ denote the $\R$-fixed point set. The space $FS_w - FS_w^\R$ is locally connected and has finite covering dimension.
\item There is an upper bound for the orders of finite subgroups of $\mathcal{O}_w \rtimes \mathcal{O}_w^\times$.
\end{enumerate}
\end{lemma}
\begin{proof}
\begin{enumerate}
\item By \cite[Proposition~1.9]{BL12b} and Lemma~\ref{lem-proper}~(\ref{lem-proper-2}), $FS(Y_w)$ is a proper metric space. Since $FS_w$ is a closed subspace of $FS(Y_w)$, the assertion follows.
\item The action of $\mathcal{O}_w \rtimes \mathcal{O}_w^\times$ on $FS_w$ is isometric, since this group acts isometrically on $Y_w$ (see Lemma~\ref{lem-action-isometric}).
    The map $FS_w \to Y_w, c \mapsto c(0)$ is equivariant, continuous, and proper (see \cite[Lemma~1.10]{BL12b}). This implies that the action of $\mathcal{O}_w \rtimes \mathcal{O}_w^\times$ on $FS_w$ is proper and cocompact because this group acts properly and cocompactly on $Y_w$ (see Lemma~\ref{lem-action-proper} and \ref{lem-action-cocompact}).
\item Note that $FS_w - FS_w^\R = FS(X_w) \times \Q(w)_\R - FS(X_w)^\R \times \Q(w)_\R = (FS(X_w) - FS(X_w)^\R) \times \Q(w)_\R$. The statement follows because $FS(X_w) - FS(X_w)^\R$ and $\Q(w)_\R$ are locally connected and have finite covering dimension (see \cite[Proposition~2.9 and 2.10]{BL12b} and Lemma~\ref{lem-proper}.
\item Let $F < \mathcal{O}_w \rtimes \mathcal{O}_w^\times$ be a finite subgroup. Since $\mathcal{O}_w < \Q(w) < \R$ is torsion-free, the composition $F \hookrightarrow \mathcal{O}_w \rtimes \mathcal{O}_w^\times \twoheadrightarrow \mathcal{O}_w^\times$ is injective.
    It remains to show that there is an upper bound for the orders of finite subgroups of $\mathcal{O}_w^\times$. Let $F' < \mathcal{O}_w^\times$ be a finite subgroup. There is a short exact sequence
    \[
    1 \to \mathcal{O}^\times \to \mathcal{O}_w^\times \to \im\big(\mathcal{O}_w^\times \to \Z^{M_w}, x \mapsto (v_{\mathfrak{p}}(x))_{\mathfrak{p} \in M_w}\big) \to 1.
    \]
    Since the quotient group is a subgroup of $\Z^{M_w}$, it is torsion-free. We conclude $F' < \mathcal{O}^\times$. Recall that $\mathcal{O}^\times$ is the direct product of a finite cyclic group $C$ and a free abelian group (see Lemma~\ref{lem-units}). This shows $F' < C$ and hence $|F'| \leq |C|$.
\end{enumerate}
\end{proof}

\begin{proposition} \label{prop-cover-pre}
There exists a natural number $M \in \N$ such that for every $\gamma > 0$ there
is an $\mathcal{O}_w \rtimes \mathcal{O}_w^\times$-invariant open cover $\mathcal{V}$ of $FS_w$ with the following properties:
\begin{enumerate}
\item For $g \in \mathcal{O}_w \rtimes \mathcal{O}_w^\times$ and $V \in \mathcal{V}$ we have $gV = V$ or $V \cap gV = \emptyset$.
\item For all $V \in \mathcal{V}$ the subgroup $G_V := \{ g \in \mathcal{O}_w \rtimes \mathcal{O}_w^\times \mid gV = V \}$ is virtually cyclic.
\item $\mathcal{O}_w \rtimes \mathcal{O}_w^\times \setminus \mathcal{V}$ is finite.
\item $\dim(\mathcal{V}) \leq M$.
\item There exists $\epsilon > 0$ such that for every $c \in FS_w$ with
    \[
    \inf\big\{ \tau > 0 \, \big| \, \exists \, g \in \mathcal{O}_w \rtimes \mathcal{O}_w^\times \text{ with } \Phi_\tau(c) = gc \big\} \leq \gamma
    \]
    there is $V \in \mathcal{V}$ satisfying
    \[
    B_\epsilon\big(\big\{ \Phi_t(c) \, \big| \, t \in [-\gamma,\gamma] \big\}\big) \subseteq V.
    \]
\end{enumerate}
\end{proposition}
\begin{proof}
This is a modification of \cite[Theorem~4.2]{BL12b}.
We explain the main adaptations.
We abbreviate $G := \mathcal{O}_w \rtimes \mathcal{O}_w^\times$.
For $g \in G$ the displacement function $d_g \colon X_w \to \R^{\geq 0}$ is defined by $d_g(x) := d_{X_w}(gx,x)$. The translation length of $g \in G$ is the number $l(g) := \inf\{ d_g(x) \mid x \in X_w \}$. An element $g \in G$ is called hyperbolic if $d_g$ attains a strictly positive minimum. By \cite[II.6.8~(1) on page 231]{BH99}, an element $g$ is hyperbolic if and only if there exists $c \in FS(X_w)$ and $\tau > 0$ with $g c = \Phi_\tau(c)$; in this case $\tau = l(g)$. Note that $g$ is hyperbolic if and only if there exists $d \in FS_w$ with $g d = \Phi_{l(g)}(d)$ (define $d:=(c,j(\frac{g_1}{1-g_2}))$).

Analogously to \cite[Notation~4.5]{BL12b}, we denote by $G^{hyp}_{\leq \gamma} \subset G$ the set of all hyperbolic elements $g$ of translation length $l(g) \leq \gamma$. We define an equivalence relation on $G^{hyp}_{\leq \gamma}$  by $g \sim g'$ if and only if there exists $c_g, c_{g'} \in FS(X_w)$ with $g c_g = \Phi_{l(g)}(c_g)$ and $g' c_{g'} = \Phi_{l(g')}(c_{g'})$ such that $d_{X_w}(c_g(t),c_{g'}(t))$ is constant in $t$. Note that we obtain the same equivalence relation if we replace $FS(X_w)$ by $FS_w$. We set $A_{\leq \gamma} := G^{hyp}_{\leq \gamma} / \sim$. The conjugation action of $G$ on $G^{hyp}_{\leq \gamma}$ descends to an action on $A_{\leq \gamma}$. For $a \in A_{\leq \gamma}$ we set $G_a := \{ g \in \mathcal{O}_w \rtimes \mathcal{O}_w^\times \mid ga=a\}$ and denote by $FS_a \subset FS_w$ the subspace that consists of all geodesics $c \in FS_w$ with $g c_g = \Phi_{l(g)}(c_g)$ for some $g \in a$. Let $Y_a := FS_a / \Phi$ be the quotient of $FS_a$ by the action of the flow.

We proceed as in \cite[subsection~4.2]{BL12b} and obtain analogously to \cite[Proposition~4.13]{BL12b}: For every $\gamma > 0$ and $a \in A_{\leq \gamma}$ there is a $G_a$-invariant open cover $\mathcal{V}_a$ of $Y_a$ such that following conditions are satisfied:
\begin{itemize}
\item For $g \in G$ and $V \in \mathcal{V}_a$ we have $gV = V$ or $V \cap gV = \emptyset$.
\item For all $V \in \mathcal{V}_a$ the subgroup $G_V := \{ g \in G \mid gV = V \}$ is virtually cyclic.
\item $G \setminus \mathcal{V}_a$ is finite.
\item $\dim(\mathcal{V}_a) \leq \dim(Y_w)$.
\end{itemize}

We conclude as in the proof of \cite[Lemma~4.15]{BL12b} that there is $\epsilon_\R > 0$ and a $G$-invariant collection $\mathcal{V}_\R$ of open subsets of $FS_w$ with the following properties:
\begin{itemize}
\item For every constant geodesic $c \in FS_w$ there exists $V \in \mathcal{V}_\R$ such that $B_{\epsilon_\R}(c) \subseteq V$.
\item For $g \in G$ and $V \in \mathcal{V}_\R$ we have $gV = V$ or $V \cap gV = \emptyset$.
\item For all $V \in \mathcal{V}_\R$ the subgroup $G_V := \{ g \in G \mid gV = V \}$ is finite.
\item $G \setminus \mathcal{V}_\R$ is finite.
\item $\dim(\mathcal{V}_\R) < \infty$.
\end{itemize}

Now, the statement follows as in \cite[proof of Theorem~4.2 on page 1380ff.]{BL12b} using \cite[Lemma~4.14]{BL12b}.
\end{proof}

\begin{proposition} \label{prop-cover}
There exists a natural number $N$ such that for every $\alpha > 0$ there
is an $\mathcal{O}_w \rtimes \mathcal{O}_w^\times$-invariant open cover $\mathcal{U}$ of $FS_w$ with the following properties:
\begin{enumerate}
\item For all $U \in \mathcal{U}$ the subgroup $G_U := \{ g \in \mathcal{O}_w \rtimes \mathcal{O}_w^\times \mid gU = U \}$ is virtually cyclic.
\item $G \setminus \mathcal{U}$ is finite.
\item We have $\dim(\mathcal{U}) \leq N$.
\item There exists $\epsilon > 0$ such that for every $c \in FS_w$ there is $U \in \mathcal{U}$ satisfying
\[
B_\epsilon\big(\big\{ \Phi_t(c) \, \big| \, t \in [-\alpha,\alpha] \big\}\big) \subseteq U.
\]
\end{enumerate}
\end{proposition}
\begin{proof}
This follows from \cite[Theorem~5.7 and Lemma~5.8]{BL12b}. Note that \cite[Convention~5.1]{BL12b} is satisfied (see Lemma~\ref{lem-FS_w}). Proposition~\ref{prop-cover-pre} implies that $FS_w$ admits long $\mathcal{V}cyc$-covers at infinity and periodic flow lines (see \cite[Definition~5.5]{BL12b}), where $\mathcal{V}cyc$ denotes the family of virtually cyclic subgroups of $\mathcal{O}_w \rtimes \mathcal{O}_w^\times$.
\end{proof}

\begin{definition}[strong homotopy action $\Psi^R$ on $X_w^R$] \label{def-sha}
We fix a base point $x_0$ in the CAT(0)-space $X_w$. For $R > 0$ we denote by $X_w^R \subset X_w$ the closed ball of radius $R$ around the base point. By \cite[Lemma 6.2]{BL12b}, $X_w^R$ is a compact, contractible, controlled $(2 \cdot \dim(X_w) + 1)$-dominated metric space.
We define a strong homotopy action
\[
\Psi^R \colon \coprod_{j=0}^\infty \big( (G \times [0,1])^j \times G \times X_w^R \big) \to X_w^R
\]
as follows.
For $x,y \in X_w$ we denote by $c_{x,y}$ the generalized geodesic satisfying $(c_{x,y})_- = 0$, $c_{x,y}(-\infty) = x$ and $c_{x,y}(\infty) = y$. We consider the deformation retraction $H^R \colon X_w \times [0,1] \to X_w$ on the ball $X_w^R$ by projecting along geodesics, i.e.,
\[
H^R(x,t) := c_{x,y_0}\big((d_{X_w}(x,y_0)-R) \cdot (1-t)\big).
\]
Note that $H^R_t \circ H^R_{t'} = H^R_{t \cdot t'}$. We define $\Psi^R$ as the strong homotopy action associated to $H^R$ (see Example~\ref{ex-sha}~(\ref{ex-sha-2})).
\end{definition}

\begin{definition}[$\mu_\beta$, $\iota$, $j_\beta$, $\kappa_{\beta,z}$] \label{def-div}
For $\beta \in \Q^\times$ we define the group automorphism
\[
\mu_\beta \colon \Q(w) \rtimes \mathcal{O}_w^\times \to \Q(w) \rtimes \mathcal{O}_w^\times, (g_1,g_2) \mapsto (\beta \cdot g_1,g_2).
\]
We define the maps
\[
\begin{array}{lll}
\iota \colon & \Q(w) \rtimes \mathcal{O}_w^\times \times X_w^R \to FS(X_w), & (g,x) \mapsto c_{gx_0,gx}, \\
j_\beta \colon & \Q(w) \rtimes \mathcal{O}_w^\times \times X_w^R \to FS_w, & (g,x) \mapsto \big(c_{\mu_\beta(g)x_0,\mu_\beta(g)x}, \mu_\beta(g) \cdot 0 \big), \\
\kappa_{\beta,z} \colon & \Q(w) \rtimes \mathcal{O}_w^\times \times X_w^R \to FS_w, & (g,x) \mapsto \big(c_{gx_0,gx}, \mu_\beta(g) \cdot 0 + z\big),
\end{array}
\]
where $\beta \in \Q$ and $z \in \Q(w)_\R$.
\end{definition}
Note that $\mu_\beta(g) \cdot 0 = j(\beta \cdot g_1) = \beta \cdot j(g_1)$. (See subsection~\ref{ss-minkowski} for the definition of the $\Q$-linear map $j \colon \Q(w) \to \Q(w)_{\R}$.)

\begin{lemma} \label{lem-alpha}
Let $S \subseteq \mathcal{O}_w \rtimes \mathcal{O}_w^\times$ be a finite symmetric subset containing the trivial element. For every $k,n \in \N$ there exists $\alpha > 0$ with the following property:
For all $\epsilon > 0$ there are $R,T > 0$ such that for every $(g,x) \in \mathcal{O}_w \rtimes \mathcal{O}_w^\times \times X_w^R$ and $(h,y) \in S^n_{\Psi^R,S,k}(g,x)$ there is $\tau \in [-\alpha,\alpha]$ with
\[
d_{FS(X_w)}\big(\Phi_T \circ \iota(g,x),\Phi_{T+\tau} \circ \iota(h,y)\big) \leq \epsilon.
\]
\end{lemma}
\begin{proof}
See the proof of \cite[Lemma~3.5]{Weg12}.
\end{proof}

\begin{lemma} \label{lem-ball}
Let $S \subseteq \mathcal{O}_w \rtimes \mathcal{O}_w^\times$ be a finite symmetric subset containing the trivial element. Let $k,n$ be natural numbers and $R,T >0$ positive real numbers. Let $K \subset \Q(w)_\R$ be a compact subset.
Let $\mathcal{U}$ be a finite-dimensional $\mathcal{O}_w \rtimes \mathcal{O}_w^\times$-invariant open cover of $FS_w$ such that for all $x \in X_w^R$ and $z \in K$ there exists $U \in \mathcal{U}$ with
\[
\kappa_{0,z}\big(S^n_{\Psi^R,S,k}(e,x)\big) \subseteq U.
\]
Then there are positive real numbers $\Lambda, B > 0$ with the property that for all $x \in X_w^R$ and $z \in K$ there exists $U \in \mathcal{U}$ such that $\kappa_{\beta,z}(h,y) \in U$ for all $0 \leq \beta \leq B$ and all $(h,y)$ with $d_{\Psi^R,S,k,\Lambda}((e,x),(h,y)) < n$.
\end{lemma}
\begin{proof}
We proceed by contradiction. Assume that $\Lambda, B > 0$ with the desired properties do not exist.
Then there exist
\begin{itemize}
\item a monotone increasing sequence $(\Lambda_m)_{m \in \N}$ with $\lim_{m \to \infty} \Lambda_m = \infty$,
\item a monotone decreasing sequence $(\beta_m)_{m \in \N}$ with $\lim_{m \to \infty} \beta_m = 0$,
\item a sequence $(x_m)_{m \in \N} \subset X_w^R$,
\item a sequence $(z_m)_{m \in \N} \subset K$,
\item for every $U \in \mathcal{U}$ a sequence $(h^U_m,y^U_m)_{m \in \N} \subset \mathcal{O}_w \rtimes \mathcal{O}_w^\times \times X_w^R$,
\end{itemize}
such that $\kappa_{\beta_m,z_m}(h^U_m,y^U_m) \notin U$ and $d_{\Psi^R,S,k,\Lambda_m}((e,x_m),(h^U_m,y^U_m)) < n$.
Since $X_w^R$ and $K$ are compact, we can arrange by passing to subsequences that $\lim_{m \to \infty} x_m = x$ and $\lim_{m \to \infty} z_m = z$.
We set $\mathcal{U}_{x,z} := \{ U \in \mathcal{U} \mid \kappa_{0,z}(e,x) \in U \}$.
Since $h^U_m$ lies in the finite set $S^{2n}$, $y^U_m$ lies in the compact set $X_w^R$, and $\mathcal{U}_{x,z}$ is finite, we can also arrange by passing to subsequences that for every $U \in \mathcal{U}_{x,z}$ we have $h^U_m = h^U$ for all $m \in \N$ and $\lim_{m \to \infty} y^U_m = y^U$.

We will use the abbreviation $d_m$ for $d_{\Psi^R,S,k,\Lambda_m}$.
For $U \in \mathcal{U}_{x,z}$ and $m \leq m'$ we conclude (since $\Lambda_m \leq \Lambda_{m'}$)
\begin{align*}
& d_m\big((e,x),(h^U,y^U)\big) \\
& \leq d_m\big((e,x),(e,x_{m'})\big) + d_m\big((e,x_{m'}),(h^U,y^U_{m'})\big) + d_m\big((h^U,y^U_{m'}),(h^U,y^U)\big) \\
& \leq d_m\big((e,x),(e,x_{m'})\big) + d_{m'}\big((e,x_{m'}),(h^U,y^U_{m'})\big) + d_m\big((h^U,y^U_{m'}),(h^U,y^U)\big) \\
& \leq d_m\big((e,x),(e,x_{m'})\big) + n + d_m\big((h^U,y^U_{m'}),(h^U,y^U)\big).
\end{align*}
Lemma~\ref{lem-qm}~(\ref{lem-qm3}) implies
\[
\lim_{m' \to \infty} d_m\big((e,x),(e,x_{m'})\big) = 0 \quad \text{and} \quad \lim_{m' \to \infty} d_m\big((h^U,y^U_{m'}),(h^U,y^U)\big) = 0.
\]
Therefore, $d_{\Psi^R,S,k,\Lambda_m}\big((e,x),(h^U,y^U)\big) \leq n$ for all $m \in \N$ and $U \in \mathcal{U}_{x,z}$.
By Lemma~\ref{lem-qm}~(\ref{lem-qm2}) we have $(h^U,y^U) \in S^n_{\Psi^R,S,k}(e,x)$ for all $U \in \mathcal{U}_{x,z}$.
Hence, by assumption there exists $U_0 \in \mathcal{U}_{x,z}$ with $\kappa_{0,z}(h^U,y^U) \in U_0$ for all $U \in \mathcal{U}_{x,z}$. In particular, $\kappa_{0,z}(h^{U_0},y^{U_0}) \in U_0$.
Since $\kappa_{\beta_m,z_m}(h^{U_0},y^{U_0}_m) \notin U_0$, we obtain the desired contradiction:
\begin{align*}
& \kappa_{0,z}(h^{U_0},y^{U_0}) = \big(\iota(h^{U_0},y^{U_0}),z\big) = \lim_{m \to \infty} \big(\iota(h^{U_0},y^{U_0}_m),\beta_m \cdot j(h^{U_0}) + z_m \big) = \\
& = \lim_{m \to \infty} \kappa_{\beta_m,z_m}(h^{U_0},y^{U_0}_m) \in FS_w - U_0.
\end{align*}
(Note that $FS_w - U_0$ is closed.)
\end{proof}

\begin{proposition} \label{prop-sc}
There exists a natural number $N$ with the following property.
Let $S \subseteq \mathcal{O}_w \rtimes \mathcal{O}_w^\times$ be a finite symmetric subset which generates $\mathcal{O}_w \rtimes \mathcal{O}_w^\times$ and contains the trivial element. Let $k \in \N$. Let $q$ be a natural number such that $q \notin \mathfrak{p}$ for all $\mathfrak{p} \in M_w$.
Then there are positive real numbers $R,\Lambda$, a natural number $M$, and for every positive multiple $m$ of $M$
\begin{itemize}
\item a simplicial complex $\Sigma$ of dimension at most $N$ with a simplicial $q^m \mathcal{O}_w \rtimes \mathcal{O}_w^\times$-action whose stabilizers are virtually cyclic and
\item a $q^m \mathcal{O}_w \rtimes \mathcal{O}_w^\times$-equivariant map $f \colon \mathcal{O}_w \rtimes \mathcal{O}_w^\times \times X_w^R \to \Sigma$
\end{itemize}
such that
\[
k \cdot d_\Sigma^1\big(f(g,x),f(h,y)\big) \leq d_{\Psi^R,S,k,\Lambda}\big((g,x),(h,y)\big)
\]
for all $(g,x),(h,y) \in \mathcal{O}_w \rtimes \mathcal{O}_w^\times \times X_w^R$.
\end{proposition}
\begin{proof}
Let $N$ be the natural number appearing in Proposition~\ref{prop-cover}.
Let $S \subseteq \mathcal{O}_w \rtimes \mathcal{O}_w^\times$ and $k,q \in \N$ be given.
We set $n := 4Nk$ and pick $\alpha > 0$ as in Lemma~\ref{lem-alpha}.
By Proposition~\ref{prop-cover} there is an $\mathcal{O}_w \rtimes \mathcal{O}_w^\times$-invariant open cover $\mathcal{U}$ of $FS_w$ and $\epsilon > 0$ with the following properties:
\begin{enumerate}
\item For all $U \in \mathcal{U}$ the subgroup $\{ g \in \mathcal{O}_w \rtimes \mathcal{O}_w^\times \mid gU = U \}$ is virtually cyclic.
\item $\mathcal{O}_w \rtimes \mathcal{O}_w^\times \setminus \mathcal{U}$ is finite.
\item $\dim(\mathcal{U}) \leq N$.
\item For every $c \in FS_w$ there is $U \in \mathcal{U}$ satisfying
\[
B_\epsilon\big(\big\{ \Phi_t(c) \, \big| \, t \in [-\alpha,\alpha] \big\}\big) \subseteq U.
\]
\end{enumerate}
We choose $R,T > 0$ as in Lemma~\ref{lem-alpha}.

Note that the sets
\[
K' := \big\{ (h,y) \in \mathcal{O}_w \rtimes \mathcal{O}_w^\times \times X_w^R \, \big| \, d_{\Psi^R,S,k,\Lambda}((e,x),(h,y)) \leq n \text{ for some } x \in X_w^R \big\}
\]
and $K'' := \{ h x_0, h y \in X_w \mid (h,y) \in K' \}$ are compact.
Hence, there are only finitely many vertices $[L] \in T(v_\mathfrak{p})$ with $d_{T(v_\mathfrak{p})}([L],\pr_{T(v_\mathfrak{p})}(K'')) < 1$.
By Lemma~\ref{lem-tree-3} there exist $z_\mathfrak{p} \in \Z$ such that $(a,1)p = p$ for all $p \in \pr_{T(v_\mathfrak{p})}(K'')$ and all $a \in \Q(w)$ with $v_\mathfrak{p}(a) \geq z_\mathfrak{p}$. Then the ideal
\[
\mathfrak{a} := \big\{ a \in \mathcal{O}_w \, \big| \, v_\mathfrak{p}(a) \geq \max\{z_\mathfrak{p},0\} \text{ for all } \mathfrak{p} \in M_w \big\} \subseteq \mathcal{O}
\]
has the property that $(a,1)x=x$ for all $a \in \mathfrak{a}$ and $x \in K''$. Since the ideal $\mathfrak{a}$ acts cocompactly on $\Q(w)_\R$ (see Lemma~\ref{lem-min}), there exists a compact subset $K \subset \Q(w)_\R$ such that $\mathfrak{a} \cdot K = \Q(w)_\R$.

By Lemma~\ref{lem-alpha} there exists for every $x \in X_w^R$, $z \in K$, and $(h,y) \in S^n_{\Psi^R,S,k}(e,x)$ an element $\tau \in [-\alpha,\alpha]$ such that
\[
d_{FS_w}\big(\Phi_T \circ \kappa_{0,z}(e,x),\Phi_{T+\tau} \circ \kappa_{0,z}(h,y)\big) = d_{FS(X_w)}\big(\Phi_T \circ \iota(e,x),\Phi_{T+\tau} \circ \iota(h,y)\big) \leq \epsilon.
\]
Recall that for every $c \in FS_w$ there is $U \in \mathcal{U}$ satisfying
\[
B_\epsilon\big(\big\{ \Phi_t(c) \, \big| \, t \in [-\alpha,\alpha] \big\}\big) \subseteq U.
\]
Hence, for every $x \in X_w^R$ and $z \in K$ there exists $U \in \mathcal{U}$ with
\[
\Phi_T \circ \kappa_{0,z}\big(S^n_{\Psi^R,S,k}(e,x)\big) \subseteq U.
\]
We apply Lemma~\ref{lem-ball} and fix $\Lambda, B > 0$ with the property that for all $x \in X_w^R$ and $z \in K$ there exists $V \in \Phi_T^{-1}(\mathcal{U})$ such that $\kappa_{\beta,z}(h,y) \in V$ for all $0 \leq \beta \leq B$ and all $(h,y)$ with $d_{\Psi^R,S,k,\Lambda}((e,x),(h,y)) \leq n$.
We can generalize this statement to all $z \in \Q(w)_\R$ by the following argument:
We can write $z = z' + j(a)$ with $a \in \mathfrak{a}$ and $z' \in K$. Then $\kappa_{\beta,z'}(h,y) \in V$ implies $\kappa_{\beta,z}(h,y) \in (a,1)V$.

We choose the natural number $M \in \N$ such that $q^{-M} \leq B$ and
\[
v_\mathfrak{p}(q^{M}-1) \geq \max\{z_\mathfrak{p},0\} - v_\mathfrak{p}(h_1) + v_\mathfrak{p}(h_2)
\]
for all $h \in \pr_{\mathcal{O}_w \rtimes \mathcal{O}_w^\times}(K')$ and $\mathfrak{p} \in M_w$. Let $m$ be a positive multiple of $M$.

Let $\Sigma := |\mathcal{V}|$ be the realization of the nerve of $\mathcal{V} := (\Phi_T \circ j_{q^{-m}})^{-1}(\mathcal{U})$ and let $f$ be the map induced by $\mathcal{V}$, i.e.,
\[
f \colon \mathcal{O}_w \rtimes \mathcal{O}_w^\times \times X_w^R \to |\mathcal{V}|, \, (g,x) \mapsto \sum_{V \in \mathcal{V}} \frac{a_V(g,x)}{s(g,x)} \, V
\]
with $a_V(g,x) := \inf\{d_{\Psi^R,S,k,\Lambda}((g,x),(h,y)) \, | \, (h,y) \notin V\}$ and $s(g,x) := \sum_V a_V(g,x)$.

We define a $q^m \mathcal{O}_w \rtimes \mathcal{O}_w^\times$-action on $FS_w$ by $g \bullet c := \mu_{q^{-m}}(g) \cdot c$.
Since the map $j_{q^{-m}}$ is equivariant with respect to this action on $FS_w$, we conclude that the cover $\mathcal{V}$ is $q^m \mathcal{O}_w \rtimes \mathcal{O}_w^\times$-invariant. Hence, the $q^m \mathcal{O}_w \rtimes \mathcal{O}_w^\times$-action on $\Sigma$ is simplicial. Since
\[
\mu_{q^{-m}}(h) j_{q^{-m}}(g,x) = j_{q^{-m}}(hg,x)
\]
for all $g,h \in \mathcal{O}_w \rtimes \mathcal{O}_w^\times$ and $x \in X_w^R$, the map $f$ is $q^m \mathcal{O}_w \rtimes \mathcal{O}_w^\times$-equivariant.

Before proving the inequality
\[
k \cdot d_\Sigma^1\big(f(g,x),f(h,y)\big) \leq d_{\Psi^R,S,k,\Lambda}\big((g,x),(h,y)\big)
\]
for all $(g,x),(h,y) \in \mathcal{O}_w \rtimes \mathcal{O}_w^\times \times X_w^R$, we will show the following statements:
\begin{enumerate}
\item For every $(g,x) \in \mathcal{O}_w \rtimes \mathcal{O}_w^\times \times X_w^R$ there exists $U \in \mathcal{U}$ such that $\kappa_{q^{-m},0}(h,y) \in \Phi_T^{-1}(U)$ for all $(h,y)$ with $d_{\Psi^R,S,k,\Lambda}((g,x),(h,y)) \leq n$.\label{enum1}
\item For all $h \in \pr_{\mathcal{O}_w \rtimes \mathcal{O}_w^\times}(K')$ and $x \in K''$ we have $\mu_{q^{-m}}(h) x = h x$.\label{enum2}
\item For every $(g,x) \in \mathcal{O}_w \rtimes \mathcal{O}_w^\times \times X_w^R$ the $n$-ball around $(g,x)$ with respect to the metric $d_{\Psi^R,S,k,\Lambda}$ lies in some $V \in \mathcal{V}$.\label{enum3}
\end{enumerate}

Proof of (\ref{enum1}): Let $(g,x) \in \mathcal{O}_w \rtimes \mathcal{O}_w^\times \times X_w^R$. For $z := -(1-q^{-m}) \cdot j(g_1 g_2^{-1}) \in \Q(w)_\R$ there exists $U' \in \mathcal{U}$ such that $\kappa_{q^{-m},z}(h',y) \in \Phi_T^{-1}(U')$ for all $(h',y)$ with $d_{\Psi^R,S,k,\Lambda}((e,x),(h',y)) \leq n$.
Let $(h,y) \in \mathcal{O}_w \rtimes \mathcal{O}_w^\times \times X_w^R$ with $d_{\Psi^R,S,k,\Lambda}((g,x),(h,y)) \leq n$. Then $h' := g^{-1} h$ satisfies $d_{\Psi^R,S,k,\Lambda}((e,x),(h',y)) \leq n$ and hence $\kappa_{q^{-m},z}(h',y) \in \Phi_T^{-1}(U')$. We set $U := g U' \in \mathcal{U}$ and conclude
\[
\kappa_{q^{-m},0}(h,y) = g \kappa_{q^{-m},z}(h',y) \in g \Phi_T^{-1}(U') = \Phi_T^{-1}(U).
\]

Proof of (\ref{enum2}): Let $h \in \pr_{\mathcal{O}_w \rtimes \mathcal{O}_w^\times}(K')$, $x \in K''$. Then
\[
v_\mathfrak{p}(q^m-1) \geq v_\mathfrak{p}(q^M-1) \geq \max\{z_\mathfrak{p},0\} - v_\mathfrak{p}(h_1) + v_\mathfrak{p}(h_2)
\]
for all $\mathfrak{p} \in M_w$.
Hence, $(q^{-m}-1)h_1 h_2^{-1} \in \mathfrak{a}$.
This implies
\[
h^{-1} \mu_{q^{-m}}(h) x = ((q^{-m}-1)h_1 h_2^{-1},1) x = x
\]
and hence $\mu_{q^{-m}}(h) x = h x$.

Proof of (\ref{enum3}): Let $(g,x) \in \mathcal{O}_w \rtimes \mathcal{O}_w^\times \times X_w^R$.
We choose $b \in \mathcal{O}_w$ such that $v_\mathfrak{p}(g_1 - b q^m) \geq \max\{z_\mathfrak{p},0\} + v_\mathfrak{p}(g_2) - v_\mathfrak{p}(q^{-m}-1)$ for all $\mathfrak{p} \in M_w$. Then $g' := ((g_1 - b q^m) g_2^{-1},1)$ satisfies $\mu_{q^{-m}}(g') p = g' p$ for all $p \in K''$ because ${g'}^{-1} \mu_{q^{-m}}(g') = ((q^{-m}-1) (g_1 - b q^m) g_2^{-1},1)$ and $v_\mathfrak{p}((q^{-m}-1) (g_1 - b q^m) g_2^{-1}) \geq \max\{z_\mathfrak{p},0\}$ for all $\mathfrak{p} \in M_w$ .

There exists an open set $U \in \mathcal{U}$ such that $\kappa_{q^{-m},0}(h',y) \in \Phi_T^{-1}(U)$ for all $(h',y)$ with $d_{\Psi^R,S,k,\Lambda}((g',x),(h',y)) \leq n$.
Since ${g'}^{-1}h' \in \pr_{\mathcal{O}_w \rtimes \mathcal{O}_w^\times}(K')$, we have $\mu_{q^{-m}}({g'}^{-1}h') x_0 = {g'}^{-1}h' x_0$ and $\mu_{q^{-m}}({g'}^{-1}h') y = g^{-1}h y$. Hence,
\[
\mu_{q^{-m}}(h') x_0 = \mu_{q^{-m}}(g') \mu_{q^{-m}}({g'}^{-1}h') x_0 =  \mu_{q^{-m}}(g') {g'}^{-1}h' x_0 = g' {g'}^{-1}h' x_0 = h' x_0.
\]
Analogously, we conclude $\mu_{q^{-m}}(h') y = h' y$.
This shows that $j_{q^{-m}}(h',y) = \kappa_{q^{-m},0}(h',y)$ for all $(h',y)$ with $d_{\Psi^R,S,k,\Lambda}((g',x),(h',y)) \leq n$. Hence,
$j_{q^{-m}}(h',y) \in \Phi_T^{-1}(U)$ for all $(h',y)$ with $d_{\Psi^R,S,k,\Lambda}((g',x),(h',y)) \leq n$. Therefore,
\[
j_{q^{-m}}((b q^m,g_2)h',y) = (b,g_2) j_{q^{-m}}(h',y) \in \Phi_T^{-1}((b,g_2)U)
\]
for all $(h',y)$ with $d_{\Psi^R,S,k,\Lambda}((g',x),(h',y)) \leq n$.
Since $(b q^m,g_2)g' = g$, we have $j_{q^{-m}}(h,y) \in \Phi_T^{-1}((b,g_2)U)$ and hence $(h,y) \in (\Phi_T \circ j_{q^{-m}})^{-1}((b,g_2)U)$ for all $(h,y)$ with $d_{\Psi^R,S,k,\Lambda}((g,x),(h,y)) \leq n$. This finishes the proof of statement (\ref{enum3}).

It remains to prove the inequality
\[
k \cdot d_\Sigma^1\big(f(g,x),f(h,y)\big) \leq d_{\Psi^R,S,k,\Lambda}\big((g,x),(h,y)\big)
\]
for all $(g,x),(h,y) \in \mathcal{O}_w \rtimes \mathcal{O}_w^\times \times X_w^R$.
Recall that $f(g,x) = \sum_{V \in \mathcal{V}} \frac{a_V(g,x)}{s(g,x)} \, V$ with $a_V(g,x) := \inf\{d_{\Psi^R,S,k,\Lambda}((g,x),(h,y)) \, | \, (h,y) \notin V\}$ and $s(g,x) := \sum_V a_V(g,x)$.

The statement (\ref{enum3}) implies that $s(g,x) \geq n$.
Note that $|a_V(g,x) - a_V(h,y)| \leq d_{\Psi^R,S,k,\Lambda}((g,x),(h,y))$ and hence
\[
\sum_{V \in \mathcal{V}} |a_V(g,x)-a_V(h,y)| \leq 2N \cdot d_{\Psi^R,S,k,\Lambda}\big((g,x),(h,y)\big)
\]
for all $(g,x), (h,y) \in \mathcal{O}_w \rtimes \mathcal{O}_w^\times \times X_w^R$.

We calculate
\begin{eqnarray*}
d^1_\Sigma(f(g,x),f(h,y)) & = & \sum_{V \in \mathcal{V}} \Big| \frac{a_V(g,x)}{s(g,x)} - \frac{a_V(h,y)}{s(h,y)} \Big| \\
& = & \sum_{V \in \mathcal{V}} \Big| \frac{a_V(g,x)-a_V(h,y)}{s(g,x)} + \frac{a_V(h,y) \cdot \big(s(h,y)-s(g,x)\big)}{s(g,x) \cdot s(h,y)} \Big| \\
& \leq & \frac{\sum_{V \in \mathcal{V}} |a_V(g,x)-a_V(h,y)|}{s(g,x)} + \frac{\big|s(h,y)-s(g,x)\big|}{s(g,x)} \\
& \leq & 2 \cdot \frac{\sum_{V \in \mathcal{V}} |a_V(g,x)-a_V(h,y)|}{s(g,x)} \\
& \leq & 4N \cdot \frac{d_{\Psi^R,S,k,\Lambda}\big((g,x),(h,y)\big)}{s(g,x)} \\
& \leq & \frac{d_{\Psi^R,S,k,\Lambda}\big((g,x),(h,y)\big)}{k}.
\end{eqnarray*}
\end{proof}

\subsection{Hyper-elementary subgroups}

In this subsection we study hyper-elementary subgroups of finite quotients of $\mathcal{O}_w \rtimes_{\cdot w} \Z$.

\begin{definition}[$t_w(q,s)$, $t_w(\mathfrak{q},s)$] \label{def-t_w}
Let $q \in \N$ be a prime number and let $\mathfrak{q} \notin M_w$ be a prime ideal in $\mathcal{O}$ which contains $q$. Let $s$ be a natural number.
We define $t_w(q,s), t_w(\mathfrak{q},s) \in \N_0$ by
\begin{eqnarray*}
t_w(q,s) \Z & = & \big\{ z \in \Z \, \big| \, w^z \equiv 1 \mod q^s \mathcal{O}_w \big\}, \\
t_w(\mathfrak{q},s) \Z & = & \big\{ z \in \Z \, \big| \, w^z \equiv 1 \mod \mathfrak{q}^s \mathcal{O}_w \big\}.
\end{eqnarray*}
\end{definition}

\begin{remark}
Let $q \mathcal{O} = \mathfrak{q}_1^{v_1} \cdots \mathfrak{q}_r^{v_r}$ be the prime ideal factorization in $\mathcal{O}$.
Then $t_w(q,s)$ is the least common multiple of the numbers $t_w(\mathfrak{q_i},s v_i)$ ($i = 1,\ldots,r$).
\end{remark}

\begin{lemma} \label{lem-q}
Let $q \in \N$ be a prime number and let $\mathfrak{q} \notin M_w$ be a prime ideal in $\mathcal{O}$ which contains $q$.
\begin{enumerate}
\item The ring $\mathcal{O}_w / \mathfrak{q} \mathcal{O}_w$ is a finite field of characteristic $q$. \label{lem-q-1}
\item $t_w(\mathfrak{q},1)$ and $q$ are coprime. In particular, $t_w(\mathfrak{q},1) \neq 0$. \label{lem-q-2}
\item For every $s \in \N$ we have $t_w(\mathfrak{q},s+1) = q \cdot t_w(\mathfrak{q},s)$ or $t_w(\mathfrak{q},s+1) = t_w(\mathfrak{q},s)$.\label{lem-q-3}
\item The normal subgroup $\mathcal{O}_w / \mathfrak{q}^s \mathcal{O}_w \rtimes_{\cdot w} t_w(\mathfrak{q},1) \Z / t_w(\mathfrak{q},s) \Z < \mathcal{O}_w / \mathfrak{q}^s \mathcal{O}_w \rtimes_{\cdot w} \Z / t_w(\mathfrak{q},s) \Z$ is a $q$-Sylow subgroup. \label{lem-q-4}
\item Let $(a,b) \in \mathcal{O}_w / \mathfrak{q}^s \mathcal{O}_w \rtimes_{\cdot w} \Z / t_w(\mathfrak{q},s) \Z$ such that $b \notin t_w(\mathfrak{q},1) \Z / t_w(\mathfrak{q},s) \Z$. Then $(a,b)$ is conjugated to $(0,b)$. \label{lem-q-5}
\end{enumerate}
\end{lemma}
\begin{proof}
\begin{enumerate}
\item The ring $\mathcal{O} / \mathfrak{q}$ is a field because the ideal $\mathfrak{q} \subset \mathcal{O}$ is maximal (see \cite[Theorem~3.1 on page~17]{Neu99}). By \cite[Proposition~2.12 on page~15]{Neu99} the index $(\mathcal{O} : \mathfrak{q})$ is finite. Obviously, the characteristic of $\mathcal{O} / \mathfrak{q}$ is $q$.

    The inclusion $\mathcal{O} \subset \mathcal{O}_w$ induces a homomorphism $\mathcal{O} / \mathfrak{q} \to \mathcal{O}_w / \mathfrak{q} \mathcal{O}_w$. We will show that this homomorphism is bijective. Injectivity follows from $\mathfrak{q} \notin M_w$. It remains to prove surjectivity. Every element in $\mathcal{O}_w$ is of the shape $a/b$ with $a, b \in \mathcal{O}$ and $v_\mathfrak{p}(b) = 0$ for all prime ideals $\mathfrak{p} \notin M_w$. Since the index $(\mathcal{O} : \mathfrak{q})$ is finite, there exist $m > n$ with $b^m \equiv b^n \mod \mathfrak{q}$. Since $v_\mathfrak{q}(b) = 0$, we obtain $b^{m-n} \equiv 1 \mod \mathfrak{q}$. This shows that $b$ is invertible in $\mathcal{O} / \mathfrak{q}$.
\item The map
    \[
    \Z / t_w(\mathfrak{q},1) \Z \to (\mathcal{O}_w / \mathfrak{q} \mathcal{O}_w)^\times, [z] \mapsto [w^z]
    \]
    is injective. Since $\mathcal{O}_w / \mathfrak{q} \mathcal{O}_w$ is a finite field of characteristic $q$, the prime number  $q$ does not divide the order of $(\mathcal{O}_w / \mathfrak{q} \mathcal{O}_w)^\times$. Since $\Z / t_w(\mathfrak{q},1) \Z$ is isomorphic to a subgroup of $(\mathcal{O}_w / \mathfrak{q} \mathcal{O}_w)^\times$, $q$ does not divide the order of $\Z / t_w(\mathfrak{q},1) \Z$.
\item Note that $t_w(\mathfrak{q},s+1)$ is a multiple of $t_w(\mathfrak{q},s)$. It remains to show that $t_w(\mathfrak{q},s+1)$ divides $q \cdot t_w(\mathfrak{q},s)$. We calculate
    \[
    w^{q \cdot t_w(\mathfrak{q},s)} - 1 = (w^{t_w(\mathfrak{q},s)} - 1) \sum_{k=0}^{q-1} w^{k \cdot t_w(\mathfrak{q},s)},
    \]
    where
    \[
    \sum_{k=0}^{q-1} w^{k \cdot t_w(\mathfrak{q},s)} \equiv \sum_{k=0}^{q-1} 1 \equiv q \equiv 0 \mod \mathfrak{q} \mathcal{O}_w.
    \]
    This shows
    \[
    w^{q \cdot t_w(\mathfrak{q},s)} - 1 \equiv 0 \mod \mathfrak{q}^{s+1} \mathcal{O}_w.
    \]
    Hence, $q \cdot t_w(\mathfrak{q},s)$ is a multiple of $t_w(\mathfrak{q},s+1)$.
\item The group $\mathcal{O}_w / \mathfrak{q}^s \mathcal{O}_w \rtimes t_w(\mathfrak{q},1) \Z / t_w(\mathfrak{q},s) \Z$ is a $q$-group because $\mathcal{O}_w / \mathfrak{q}^s \mathcal{O}_w$ and $t_w(\mathfrak{q},1) \Z / t_w(\mathfrak{q},s) \Z$ are $q$-groups (see (\ref{lem-q-3})).
    The short exact sequence
    \[
    1 \to \mathcal{O}_w / \mathfrak{q}^s \mathcal{O}_w \rtimes t_w(\mathfrak{q},1) \Z / t_w(\mathfrak{q},s) \Z \to \mathcal{O}_w / \mathfrak{q}^s \mathcal{O}_w \rtimes \Z / t_w(\mathfrak{q},s) \Z \to \Z / t_w(\mathfrak{q},1) \Z \to 1
    \]
    implies that $\mathcal{O}_w / \mathfrak{q}^s \mathcal{O}_w \rtimes t_w(\mathfrak{q},1) \Z / t_w(\mathfrak{q},s) \Z$ is a $q$-Sylow subgroup because $t_w(\mathfrak{q},1)$ and $q$ are coprime (see (\ref{lem-q-2})).
\item We have $w^b \not\equiv 1 \mod \mathfrak{q} \mathcal{O}_w$ because $b \notin t_w(\mathfrak{q},1) \Z / t_w(\mathfrak{q},s) \Z$.
    Since $\mathcal{O}_w / \mathfrak{q} \mathcal{O}_w$ is a field (see (\ref{lem-q-1})), there exists $z \in \mathcal{O}_w$ such that $(w^b - 1) z \equiv 1 \mod \mathfrak{q} \mathcal{O}_w$. This implies
    \[
    \big((w^b - 1) z\big)^{q^{s-1}} \equiv 1 \mod \mathfrak{q}^s \mathcal{O}_w.
    \]
    Hence, $w^b - 1$ is a unit in $\mathcal{O}_w / \mathfrak{q}^s \mathcal{O}_w$. We set $x := (w^b-1)^{-1} a \in \mathcal{O}_w / \mathfrak{q^s} \mathcal{O}_w$.
    Then $(x,0)(a,b)(x,0)^{-1} = (x(1-w^b)+a,b) = (0,b)$.
\end{enumerate}
\end{proof}

Analogously to \cite[Theorem 4.5]{FW14} we have
\begin{proposition} \label{prop-hyp-el}
Let $q \in \N$ be a prime number and let $\mathfrak{q} \notin M_w$ be a prime ideal in $\mathcal{O}$ which contains $q$. Let $s$ be a natural number.
Every hyper-elementary subgroup of $H < \mathcal{O}_w / \mathfrak{q}^s \mathcal{O}_w \rtimes_{\cdot w} \Z / t_w(\mathfrak{q},s) \Z$ is
\begin{itemize}
\item a subgroup of $\mathcal{O}_w / \mathfrak{q}^s \mathcal{O}_w \rtimes t_w(\mathfrak{q},1) \Z / t_w(\mathfrak{q},s) \Z$ or
\item a subgroup of $\mathcal{O}_w / \mathfrak{q}^s \mathcal{O}_w \rtimes \frac{t_w(\mathfrak{q},s)}{t_w(\mathfrak{q},1)} \Z / t_w(\mathfrak{q},s) \Z$ or
\item conjugated to a subgroup of $\{0\} \rtimes \Z / t_w(\mathfrak{q},s) \Z$.
\end{itemize}
\end{proposition}
\begin{proof}
Let $1 \to C \to H \to P \to 1$ be a short exact sequence of groups, where $C$ is a cyclic group of order $n$ and $P$ is a $p$-group for some prime number $p$ which does not divide $n$.

We first consider the case that $C$ is the trivial group. Then $H$ is a $p$-group. If $p = q$ then $H$ is a subgroup of $\mathcal{O}_w / \mathfrak{q}^s \mathcal{O}_w \rtimes t_w(\mathfrak{q},1) \Z / t_w(\mathfrak{q},s) \Z$ by Sylow's theorem and Lemma~\ref{lem-q}~(\ref{lem-q-4}). If $p \neq q$ then $H$ is a subgroup of $\mathcal{O}_w / \mathfrak{q}^s \mathcal{O}_w \rtimes \frac{t_w(\mathfrak{q},s)}{t_w(\mathfrak{q},1)} \Z / t_w(\mathfrak{q},s) \Z$ because $\frac{t_w(\mathfrak{q},s)}{t_w(\mathfrak{q},1)}$ is a power of $q$ by Lemma~\ref{lem-q}~(\ref{lem-q-3}).

From now on we assume that the cyclic group $C$ is non-trivial. Let $(a,b) \in \mathcal{O}_w / \mathfrak{q}^s \mathcal{O}_w \rtimes_{\cdot w} \Z / t_w(\mathfrak{q},s) \Z$ be a generator of $C$.

Let us consider the case $a = 0$. Suppose that there exists $(x,y) \in H$ with $x \neq 0$. We calculate $(x,y)(0,b)(x,y)^{-1} = (x(1-w^b),b)$. Since $C$ is normal in $H$, we conclude $x(1-w^b) = 0$. This shows that $C$ lies in the center of $H$ and that $w^b \equiv 1 \mod \mathfrak{q} \mathcal{O}_w$. We conclude $(a,b) \in \{0\} \rtimes t_w(\mathfrak{q},1) \Z / t_w(\mathfrak{q},s) \Z$ which implies by Lemma~\ref{lem-q}~(\ref{lem-q-3}) that $C$ is a $q$-group. Hence $p \neq q$. Let $H_p$ be a $p$-Sylow subgroup of $H$. We have $H = C \times H_p$ because $C$ lies in the center of $H$.
Since the index $[\mathcal{O}_w / \mathfrak{q}^s \mathcal{O}_w \rtimes \Z / t_w(\mathfrak{q},s) \Z : \{0\} \rtimes \frac{t_w(\mathfrak{q},s)}{t_w(\mathfrak{q},1)} \Z / t_w(\mathfrak{q},s) \Z]$ is a power of $q$ (see Lemma~\ref{lem-q}~(\ref{lem-q-4})), every $p$-Sylow subgroup of $\{0\} \rtimes \frac{t_w(\mathfrak{q},s)}{t_w(\mathfrak{q},1)} \Z / t_w(\mathfrak{q},s) \Z$ is a $p$-Sylow subgroup of $\mathcal{O}_w / \mathfrak{q}^s \mathcal{O}_w \rtimes \Z / t_w(\mathfrak{q},s) \Z$. By Sylow's theorem $H_p$ is conjugated to a subgroup of $\{0\} \rtimes \frac{t_w(\mathfrak{q},s)}{t_w(\mathfrak{q},1)} \Z / t_w(\mathfrak{q},s) \Z$. In particular, $H_p$ is a cyclic group. Since $C$ and $H_p$ are cyclic groups with coprime orders, $H = C \times H_p$ is cyclic. Let $(a',b') \in H$ be a generator. If $b' \notin t_w(\mathfrak{q},1) \Z / t_w(\mathfrak{q},s) \Z$ then $H$ is conjugated to a subgroup of $\{0\} \rtimes \Z / t_w(\mathfrak{q},s) \Z$ (see Lemma~\ref{lem-q}~(\ref{lem-q-5})).

Next we consider the case $b = 0$. Then $C$ is a subgroup of $\mathcal{O}_w / \mathfrak{q}^s \mathcal{O}_w \rtimes \{0\}$ and hence a $q$-group. Since $P$ is a $p$-group with $p \neq q$, $H$ is a subgroup of $\mathcal{O}_w / \mathfrak{q}^s \mathcal{O}_w \rtimes \frac{t_w(\mathfrak{q},s)}{t_w(\mathfrak{q},1)} \Z / t_w(\mathfrak{q},s) \Z$.

Now suppose that both $a$ and $b$ are not zero. If $b \notin t_w(\mathfrak{q},1) \Z / t_w(\mathfrak{q},s) \Z$ then $(a,b)$ can be conjugated to $(0,b)$ (see Lemma~\ref{lem-q}~(\ref{lem-q-5})), which can be included in the case $a = 0$. Here, we used the fact that $\mathcal{O}_w / \mathfrak{q}^s \mathcal{O}_w \rtimes t_w(\mathfrak{q},1) \Z / t_w(\mathfrak{q},s) \Z$ and $\mathcal{O}_w / \mathfrak{q}^s \mathcal{O}_w \rtimes \frac{t_w(\mathfrak{q},s)}{t_w(\mathfrak{q},1)} \Z / t_w(\mathfrak{q},s) \Z$ are invariant under conjugation (i.e., normal subgroups of $\mathcal{O}_w / \mathfrak{q}^s \mathcal{O}_w \rtimes \Z / t_w(\mathfrak{q},s) \Z$).
If $b \in t_w(\mathfrak{q},1) \Z / t_w(\mathfrak{q},s) \Z$ then $C$ is a $q$-group (see Lemma~\ref{lem-q}~(\ref{lem-q-4})). We conclude $p \neq q$. Let $H_p$ be a $p$-Sylow subgroup of $H$. By Sylow's theorem we can and will conjugate $H$ such that $H_p$ becomes a subgroup of $\{0\} \rtimes \frac{t_w(\mathfrak{q},s)}{t_w(\mathfrak{q},1)} \Z / t_w(\mathfrak{q},s) \Z$. In particular, $H_p$ is a cyclic group.
Let $(0,y)$ be a generator of $H_p$. Since $C$ is normal in $H$, we have
\[
(w^y a,b) = (0,y) (a,b) (0,y)^{-1} = (a,b)^n = (a \sum_{k=0}^n w^{kb},nb)
\]
for some $n \in \N$. If $a \neq 0$ then $w^y \equiv \sum_{k=0}^n w^{kb} \equiv n \mod \mathfrak{q} \mathcal{O}_w$.
Since $C$ is a $q$-group and $b \neq 0$, $nb = b$ implies that $n-1$ is a multiple of $q$. Hence, $w^y \equiv 1 \mod \mathfrak{q}  \mathcal{O}_w$ and $(0,y) \in \{0\} \rtimes t_w(\mathfrak{q},1) \Z / t_w(\mathfrak{q},s) \Z$. Since $H_p$ is a subgroup of $\{0\} \rtimes \frac{t_w(\mathfrak{q},s)}{t_w(\mathfrak{q},1)} \Z / t_w(\mathfrak{q},s) \Z$, $H_p$ is the trivial group. This implies $H = C < \mathcal{O}_w / \mathfrak{q}^s \mathcal{O}_w \rtimes t_w(\mathfrak{q},1) \Z / t_w(\mathfrak{q},s) \Z$.
\end{proof}

\begin{corollary} \label{cor-hyp-el}
Let $q \in \N$ be a prime number such that $q \notin \mathfrak{p}$ for all $\mathfrak{p} \in M_w$.
We consider the prime ideal factorization $q \mathcal{O} = \mathfrak{q}_1^{v_1} \cdots \mathfrak{q}_r^{v_r}$ in $\mathcal{O}$.
Let $s$ be a natural number.
Let $H < \mathcal{O}_w / q^s \mathcal{O}_w \rtimes_{\cdot w} \Z / t_w(q,s) \Z$ be a hyper-elementary subgroup such that the index $[\Z / t_w(q,s) \Z,\pi(H)]$ satisfies
\[
[\Z / t_w(q,s) \Z,\pi(H)] < \min\big\{t_w(\mathfrak{q_i},1),t_w(\mathfrak{q_i},s v_i)/t_w(\mathfrak{q_i},1) \, \big| \, i = 1,\ldots,r \big\},
\]
where $\pi \colon \mathcal{O}_w / q^s \mathcal{O}_w \rtimes \Z / t_w(q,s) \Z \to \Z / t_w(q,s) \Z$ denotes the quotient map.
Then there exists $(x,y) \in \mathcal{O}_w / q^s \mathcal{O}_w \rtimes (\mathcal{O}_w / q^s \mathcal{O}_w)^\times$ such that $(x,y)H(x,y)^{-1}$ is a subgroup of $\{0\} \rtimes \Z / t_w(q,s) \Z$.\\
(Note that $\mathcal{O}_w / q^s \mathcal{O}_w \rtimes \Z / t_w(q,s) \Z$ is a normal subgroup of $\mathcal{O}_w / q^s \mathcal{O}_w \rtimes (\mathcal{O}_w / q^s \mathcal{O}_w)^\times$, where $(\mathcal{O}_w / q^s \mathcal{O}_w)^\times$ acts on $\mathcal{O}_w / q^s \mathcal{O}_w$ by multiplication.)
\end{corollary}
\begin{proof}
We set $l := [\Z / t_w(q,s) \Z,\pi(H)]$. This implies $\pi(H) = l \Z / t_w(q,s) \Z$.
By the Chinese Reminder Theorem we have
\[
\mathcal{O}_w / q^s \mathcal{O}_w \cong \bigoplus_{i=1}^r \mathcal{O}_w / \mathfrak{q}_i^{s v_i} \mathcal{O}_w.
\]
Let $\pr_i \colon \mathcal{O}_w / q^s \mathcal{O}_w \to \mathcal{O}_w / \mathfrak{q}_i^{s v_i} \mathcal{O}_w$ denote the projection.
Note that $\pr_i$ maps $[z] \in \Z / t_w(q,s) \Z < (\mathcal{O}_w / q^s \mathcal{O}_w)^\times$ to $[z] \in \Z / t_w(\mathfrak{q}_i,s v_i) \Z < (\mathcal{O}_w / \mathfrak{q}_i^{s v_i} \mathcal{O}_w)^\times$.
We denote by $H_i$ the image of $H$ under the projection $\pr_i \rtimes \pr_i$.
Since $l < t_w(\mathfrak{q}_i,1)$, $H_i$ is not a subgroup of $\mathcal{O}_w / \mathfrak{q}_i^{s v_i} \mathcal{O}_w \rtimes t_w(\mathfrak{q}_i,1) \Z / t_w(\mathfrak{q}_i,s v_i) \Z$. Since $l < \frac{t_w(\mathfrak{q}_i,s v_i)}{t_w(\mathfrak{q}_i,1)}$, $H_i$ is not a subgroup of $\mathcal{O}_w / \mathfrak{q}_i^{s v_i} \mathcal{O}_w \rtimes \frac{t_w(\mathfrak{q}_i,s v_i)}{t_w(\mathfrak{q}_i,1)} \Z / t_w(\mathfrak{q}_i,s v_i) \Z$. Hence, by Proposition~\ref{prop-hyp-el} there exists
\[
(x_i,y_i) \in \mathcal{O}_w / \mathfrak{q}_i^{s v_i} \mathcal{O}_w \rtimes \Z / t_w(\mathfrak{q}_i,s v_i) \Z < \mathcal{O}_w / \mathfrak{q}_i^{s v_i} \mathcal{O}_w \rtimes (\mathcal{O}_w / \mathfrak{q}_i^{s v_i} \mathcal{O}_w)^\times
\]
such that $(x_i,y_i)H_i(x_i,y_i)^{-1}$ is a subgroup of $\{0\} \rtimes \Z / t_w(\mathfrak{q}_i,s v_i) \Z$. Under the isomorphism given by the Chinese Reminder Theorem the elements $(x_i,y_i)$ ($i = 1,\ldots,r$) correspond to an element $(x,y) \in \mathcal{O}_w / q^s \mathcal{O}_w \rtimes (\mathcal{O}_w / q^s \mathcal{O}_w)^\times$. We conclude that $(x,y)H(x,y)^{-1}$ is a subgroup of $\{0\} \rtimes \Z / t_w(q,s) \Z$.
\end{proof}

\subsection{The groups $\mathbf{G_w}$ satisfy FJCw}

In this subsection we complete the proof of Theorem~\ref{thm}. Because of Proposition~\ref{prop-sdp} it suffices to show that the groups $G_w$ ($w \in \overline{\Q}^\times$) satisfy FJCw.

\begin{lemma}
Let $w \in \overline{\Q}^\times$ such that $w^n = 1$ for some $n \in \N$.
Then $G_w$ satisfies FJCw.
\end{lemma}
\begin{proof}
This follows from the fact that $\Z[w,w^{-1}] \times n\Z$ is an abelian subgroup of finite index (see Proposition~\ref{prop-FJCw-1} (Overgroups of finite index) and Example~\ref{ex-FJCw-2}).
\end{proof}

Now suppose that $w^n \neq 1$ for all $n \in \N$. We will prove that the group $G_w$ is a FHJ group with respect to $\mathcal{F} := \{ H < G_w \mid H \mbox{ satisfies FJCw} \}$. Then by Corollary~\ref{cor-FHJ} the group $G_w$ satisfies FJCw.

\begin{proposition} \label{prop-final}
Let $w \in \overline{\Q}^\times$ such that $w^n \neq 1$ for all $n \in \N$.
Then the group $G_w$ is a FHJ group with respect to $\mathcal{F} := \{ H < G_w \mid H \mbox{ satisfies FJCw} \}$.
\end{proposition}
\begin{proof}
Let $N \in \N$ be the natural number appearing in Proposition~\ref{prop-sc}.
Let $S \subseteq G_w$ be a finite symmetric subset which generates $G_w$ and contains the trivial element.
We set $m_2 := \max\{ |g_2| \mid g = (g_1,g_2) \in S^n \subseteq \mathcal{O}_w \rtimes \Z \}$.
Let $S' \subseteq \mathcal{O}_w \rtimes \mathcal{O}_w^\times$ be a finite symmetric subset which generates $\mathcal{O}_w \rtimes \mathcal{O}_w^\times$ and contains $S$.
We fix a natural number $n$.

We choose a (large) prime number $q \in \N$ such that
\begin{itemize}
\item $q \notin \mathfrak{p}$ for all $\mathfrak{p} \in M_w$,
\item $q$ lies in none of the prime ideals which appear in the prime factorization of the fractional ideal $(w^m-1) \mathcal{O}$ with $1 \leq m \leq 2 \cdot n \cdot m_2$.
\end{itemize}
Note that such a prime number exists, since there are infinitely many prime numbers and every prime ideal contains at most one prime number.
Let $q \mathcal{O} = \mathfrak{q}_1^{v_1} \cdots \mathfrak{q}_r^{v_r}$ be the prime ideal factorization.
The second assumption implies $2 \cdot n \cdot m_2 < t_w(\mathfrak{q_i},1)$ for all $1 \leq i \leq r$.

By Proposition~\ref{prop-sc} there exist $R,\Lambda > 0$, $M \in \N$, and for every positive multiple $m$ of $M$
\begin{itemize}
\item a simplicial complex $\Sigma$ of dimension at most $N$ with a simplicial $q^m \mathcal{O}_w \rtimes \mathcal{O}_w^\times$-action whose stabilizers are virtually cyclic and
\item a $q^m \mathcal{O}_w \rtimes \mathcal{O}_w^\times$-equivariant map $f \colon \mathcal{O}_w \rtimes \mathcal{O}_w^\times \times X_w^R \to \Sigma$
\end{itemize}
such that
\[
n \cdot d_\Sigma^1\big(f(g,x),f(h,y)\big) \leq d_{\Psi^R,{S'}^n,n,\Lambda}\big((g,x),(h,y)\big)
\]
for all $(g,x),(h,y) \in \mathcal{O}_w \rtimes \mathcal{O}_w^\times \times X_w^R$.

We choose a (large) natural number $m$ satisfying $q^{-m} \leq B$ and $w^z-1 \notin \mathfrak{q_i}^{m v_i} \mathcal{O}_w$ for $z = 1,\ldots,2 \cdot n \cdot m_2 \cdot t_w(\mathfrak{q_i},1)$ and $i = 1,\ldots,r$.
This implies
\[
2 \cdot n \cdot m_2 < \frac{t_w(\mathfrak{q_i},m v_i)}{t_w(\mathfrak{q_i},1)}
\]
for $i = 1,\ldots,r$.
We define
\[
F_n := \mathcal{O}_w / q^m \mathcal{O}_w \rtimes \Z / t_w(q,m) \Z
\]
and denote the composition of the inclusion $G_w \hookrightarrow \mathcal{O}_w \rtimes \Z$ and the quotient map $\mathcal{O}_w \rtimes \Z \to F_n$ by $\alpha_n$.

Let $H < F_n$ be a hyper-elementary subgroup.
Since
\[
2 \cdot n \cdot m_2 < \min\big\{ t_w(\mathfrak{q_i},1) , \frac{t_w(\mathfrak{q_i},m v_i)}{t_w(\mathfrak{q_i},1)} \, \big| \, i = 1,\ldots,r \big\},
\]
Corollary~\ref{cor-hyp-el} implies that
\begin{enumerate}
\item there exists $(x',y') \in \mathcal{O}_w / q^m \mathcal{O}_w \rtimes (\mathcal{O}_w / q^m \mathcal{O}_w)^\times$ such that $(x',y')H(x',y')^{-1}$ is a subgroup of $\{0\} \rtimes \Z / t_w(q,m) \Z$ or \label{item-1}
\item $2 \cdot n \cdot m_2 < [\Z / t_w(q,m) \Z, \pi(H)]$, where
    \[
    \pi \colon \mathcal{O}_w / q^m \mathcal{O}_w \rtimes \Z / t_w(q,m) \Z \to \Z / t_w(q,m) \Z
    \]
    is the quotient map. \label{item-2}
\end{enumerate}

We first consider the case (\ref{item-1}).
We choose elements $\overline{x} \in \mathcal{O}_w$ respectively $\overline{y} \in \mathcal{O}_w^\times$ which represent $x'$ respectively $y'$.
As mentioned above, we have
\begin{itemize}
\item a simplicial complex $\Sigma$ of dimension at most $N$ with a simplicial $q^m \mathcal{O}_w \rtimes \mathcal{O}_w^\times$-action whose stabilizers are virtually cyclic and
\item a $q^m \mathcal{O}_w \rtimes \mathcal{O}_w^\times$-equivariant map $f \colon \mathcal{O}_w \rtimes \mathcal{O}_w^\times \times X_w^R \to \Sigma$
\end{itemize}
such that
\[
n \cdot d_\Sigma^1\big(f(g,x),f(h,y)\big) \leq d_{\Psi^R,{S'}^n,n,\Lambda}\big((g,x),(h,y)\big)
\]
for all $(g,x),(h,y) \in \mathcal{O}_w \rtimes \mathcal{O}_w^\times \times X_w^R$.
Note that
\[
d_{\Psi^R,{S'}^n,n,\Lambda}\big((g,x),(h,y)\big) \leq d_{\Psi^R,S^n,n,\Lambda}\big((g,x),(h,y)\big)
\]
for all $(g,x),(h,y) \in G_w \times X_w^R$.
We set $X_{n,H} := X_w^R$, $\Psi_{n,H} := \Psi^R$, $\Lambda_{n,H} := \Lambda$.
We define $E_{n,H} := \Sigma$, where the action of $\alpha_n^{-1}(H)$ on $\Sigma$ is given by $h \bullet s := ((x',y')^{-1}h(x',y'))s$, and $f_{n,H} \colon G_w \times X_w^R \to \Sigma$ by $f_{n,H}(g,x) := f((\overline{x},\overline{y})^{-1}g,x)$.

It remains to consider the case (\ref{item-2}): $2 \cdot n \cdot m_2 < [\Z / t_w(q,m) \Z, \pi(H)]$.
We set $l := [\Z / t_w(q,m) \Z, \pi(H)]$.
Let $X_{n,H} := \pt$ be the space consisting of one point. Then $d_{\Psi_{n,H},S^n,n,\Lambda_{n,H}}((g,\pt),(h,\pt)) = l_{S^{2n}}(h^{-1}g)$, where $l_{S^{2n}}$ is the word length function with respect to the symmetric generating subset $S^{2n} \subseteq G_w$.
We define the simplicial complex $E_{n,H}$ to be the real line $\R$ (dimension 1) with the elements $l \cdot \Z$ as vertices.
There is an $G_w$-action on $E_{n,H}$ given by $(x,y)r := r+y$ ($(x,y) \in G_w$, $r \in E_{n,H}$).
The stabilizers are the abelian group $\Z[w,w^{-1}]$ and hence belong to $\mathcal{F}$.
The restriction of this action to an $\alpha_n^{-1}(H)$-action is simplicial.
We define
\[
f_{n,H} \colon G_w \to E_{n,H}, (x,y) \mapsto y.
\]
This map is $G_w$-equivariant.
We have to show
\[
n \cdot d_{E_{n,H}}^1\big(f_{n,H}(g),f_{n,H}(h)\big) \leq l_{S^{2n}}(h^{-1}g)
\]
for all $g = (g_1,g_2), h = (h_1,h_2) \in G_w$ with $h^{-1}g \in S^n$.
If $g \neq h$ then
\begin{align*}
& d^1_{E_{n,H}}(f_{n,H}(g),f_{n,H}(h)) \leq \frac{2}{l} \cdot d_\R(f_{n,H}(g),f_{n,H}(h)) = \frac{2}{l} \cdot |g_2-h_2| = \\
& = \frac{2}{l} \cdot |(h^{-1}g)_2| \leq \frac{2}{l} \cdot m_2 < \frac{1}{n} \leq \frac{1}{n} \cdot l_{S^{2n}}(h^{-1}g).
\end{align*}
This finishes the proof of Proposition~\ref{prop-final}.
\end{proof}

\section*{List of notation}

\begin{tabularx}{\textwidth}{lX}
$FS(X_w)$ & flow space for $X_w$, see Definition~\ref{def-FS} on page~\pageref{def-FS}\\
$FS_w$ & $:= FS(X_w) \times \Q(w)_\R$, see Definition~\ref{def-FS} on page~\pageref{def-FS}\\
$G_w$ & $:= \Z[w,w^{-1}] \rtimes_{\cdot w} \Z$, see Definition~\ref{def-G_w} on page~\pageref{def-G_w}\\
$M_w$ & $:= \{ \mathfrak{p} \subset \mathcal{O} \text{ prime ideal} \mid v_\mathfrak{p}(w) \neq 0 \}$, see subsection~\ref{subsec-O_w} on page~\pageref{subsec-O_w}\\
$\mu_\beta$, $\iota$, $j_\beta$, $\kappa_{\beta,z}$ & see Definition~\ref{def-div} on page~\pageref{def-div}\\
$n_w$ & rank of $\mathcal{O}^\times$, see Definition~\ref{def-units} on page~\pageref{def-units}\\
$\mathcal{O}$ & ring of integers in $\Q(w)$, see subsection~\ref{subsec-O_w} on page~\pageref{subsec-O_w}\\
$\mathcal{O}_w$ & $:= \{ x \in \Q(w) \mid v_\mathfrak{p}(x) \geq 0 \text{ for all } \mathfrak{p} \notin M_w \}$, see subsection~\ref{subsec-O_w} on page~\pageref{subsec-O_w}\\
$\Psi^R$ & strong homotopy action on $X_w^R$, see Definition~\ref{def-sha} on page~\pageref{def-sha}\\
$\mathbf{\Q(w)_\R}$ & Minkowski space, see subsection~\ref{ss-minkowski} on page~\pageref{ss-minkowski}\\
$t_w(q,s)$, $t_w(\mathfrak{q},s)$ & see Definition~\ref{def-t_w} on page~\pageref{def-t_w}\\
$T(v_\mathfrak{p})$ & tree associated to $v_\mathfrak{p}$, see subsection~\ref{ss-tree} on page~\pageref{ss-tree}\\
$v_\mathfrak{p}$ & valuation with respect to the prime ideal $\mathfrak{p}$, see subsection~\ref{subsec-O_w} on page~\pageref{subsec-O_w}\\
$X_w$ & $:= \R^{n_w} \times \prod_{\mathfrak{p} \in M_w} T(v_\mathfrak{p})$, see subsection~\ref{subsec-XY_w} on page~\pageref{subsec-XY_w}\\
$X_w^R$ & closed ball of radius $R$ in $X_w$, see Definition~\ref{def-sha} on page~\pageref{def-sha}\\
$Y_w$ & $:= X_w \times \Q(w)_{\R}$ with modified metric, see Definition~\ref{def-Y_w} on page~\pageref{def-Y_w}
\end{tabularx}

\end{document}